\documentclass{amsart}

\usepackage{amsmath}
\usepackage{amsthm}
\usepackage{amsopn}
\usepackage{amssymb}
\usepackage[all]{xy}
\usepackage{graphicx}

\newcommand{\br}[1]{\overline{#1}}

\newcommand{\td}[1]{\widetilde{#1}}

\newcommand{\FF}{\mathbb{F}}
\newcommand{\GG}{\mathbb{G}}
\newcommand{\MS}{\mathbb{S}}

\newcommand{\HH}{\mathbb{H}}
\newcommand{\Sp}{\mathrm{Sp}}

\newcommand{\Set}{\mathrm{Set}}
\newcommand{\Mod}{\mathrm{Mod}}
\newcommand{\Alg}{\mathrm{Alg}}

\theoremstyle{definition}
 \newtheorem{thm}{Theorem}[subsection]
 \newtheorem{cor}[thm]{Corollary}
 \newtheorem{lem}[thm]{Lemma}
 \newtheorem{prop}[thm]{Proposition}
 \newtheorem{defn}[thm]{Definition}

 \newtheorem{rmk}[thm]{Remark}
 \newtheorem{assumption}[thm]{Assumption}
\newtheorem*{thm*}{Theorem}
\newtheorem*{cor*}{Corollary}
\newtheorem*{lem*}{Lemma}
\newtheorem*{prop*}{Proposition}
\newtheorem*{defn*}{Definition}
\newtheorem*{ex*}{Example}
\newtheorem*{exs*}{Examples}
\newtheorem*{rmk*}{Remark}
\newtheorem*{claim*}{Claim}

\numberwithin{equation}{section}
\numberwithin{figure}{section}

\DeclareMathOperator{\Res}{Res}
\DeclareMathOperator{\Hom}{Hom}

\DeclareMathOperator{\Sta}{Stab}
\DeclareMathOperator{\Map}{Map}
\DeclareMathOperator{\MAP}{\mathbf{Map}}
\DeclareMathOperator*{\holim}{holim}
\DeclareMathOperator*{\hocolim}{hocolim}
\DeclareMathOperator*{\colim}{colim}

\DeclareMathOperator{\Ind}{Ind}
\DeclareMathOperator{\CoInd}{CoInd}

\title{The homotopy fixed point spectra of profinite Galois extensions}

\author[Mark Behrens]{Mark Behrens$\sp 1$}
\address{
Department of Mathematics \\
Massachusetts Institute of Technology \\
Cambridge, MA 02139, U.S.A.}

\author[Daniel G. Davis]{Daniel G. Davis$\sp 2$}
\address{
Department of Mathematics \\
University of Louisiana at Lafayette \\
Lafayette, LA 70504, U.S.A. }

\begin{document}

\begin{abstract}
Let $E$ be a $k$-local profinite $G$-Galois extension of
an $E_\infty$-ring spectrum $A$ (in the sense of Rognes).  We show that
$E$ may be regarded as producing
a discrete $G$-spectrum.
Also, we prove that if $E$ is a profaithful $k$-local 
profinite extension which satisfies
certain extra conditions, then the forward direction of 
Rognes's Galois correspondence extends to the profinite setting. 
We show that the function spectrum $F_A((E^{hH})_k, (E^{hK})_k)$ is equivalent
to the localized homotopy fixed point spectrum
$((E[[G/H]])^{hK})_k$ where $H$ and $K$ are closed 
subgroups of $G$.  Applications to Morava $E$-theory are given, including 
showing that the homotopy fixed points defined by Devinatz and Hopkins for
closed subgroups of the extended 
Morava stabilizer group agree with those defined with respect to a 
continuous action in terms 
of the derived functor of fixed points.
\end{abstract}

\maketitle

\footnotetext[1]{The first author was supported by NSF grant DMS-0605100, 
the Sloan Foundation, and DARPA.}
\footnotetext[2]{Part of the second author's work on this paper was 
supported by an NSF VIGRE grant at Purdue University, 
a visit to the Mittag-Leffler Institute, and 
a grant from the Louisiana Board of Regents Support Fund.}

\setcounter{tocdepth}{1}
\tableofcontents

\section{Introduction}

In \cite{Rognes}, John Rognes develops a Galois theory of commutative
$S$-algebras which mimics Galois theory for commutative rings.  
Let $k$ be an $S$-module, and let $(-)_k$ denote Bousfield localization with
respect to $k$.  Given a $k$-local cofibrant commutative $S$-algebra $A$, and a 
cofibrant commutative
$A$-algebra $E$ that is $k$-local, Rognes gives the following definition of a
finite $k$-local Galois extension.

\begin{defn}[Finite Galois extension]\label{defn:finiteGalois}
The spectrum $E$ is a \emph{$k$-local $G$-Galois extension} of $A$, for a 
finite discrete group $G$, if it satisfies the following conditions:
\begin{enumerate}
\item $G$ acts on $E$ through commutative $A$-algebra maps.
\item The canonical map $A \rightarrow E^{hG}$ is an equivalence.
\item The canonical map $(E \wedge_A E)_k \rightarrow \Map(G, E)$ is an 
equivalence.
\end{enumerate}
$E$ is said to be \emph{$k$-locally faithful} over $A$ if 
$(M \wedge_A E)_k \simeq \ast$ 
implies
that $M_k \simeq \ast$ for every $A$-module $M$.  In the context of
$k$-local Galois extensions, we shall simply refer to such extensions as
\emph{faithful}.
\end{defn}

\begin{rmk}\label{rmk:Hilbert90}
Rognes (\cite[Prop.~6.3.3]{Rognes}; see also \cite{BakerRichter})
shows that a $k$-local $G$-Galois extension is 
faithful if and only if the additive form of 
Hilbert's Theorem 90 holds:
$$ (E^{tG})_k \simeq \ast. $$
We will mostly 
consider faithful Galois extensions, because these are the
Galois extensions for which the fundamental theorem of Galois theory holds.
We refer the interested reader to
\cite{unfaithful} for an example, due to Ben Wieland, of a Galois
extension that is not faithful.
\end{rmk}

Let $G$ be a profinite group.
Following (and slightly modifying) Rognes's definition 
\cite[Def.~8.1.1]{Rognes} of a $k$-local pro-$G$-Galois extension,
we define a (profaithful) 
$k$-local 
profinite $G$-Galois extension $E$ of $A$ to be a colimit (in
the category of commutative $A$-algebras) of (faithful) 
$k$-local $G/U_\alpha$-Galois extensions $E_\alpha$ of $A$, for a cofinal
system of open normal subgroups $U_\alpha$ of $G$ (see
Definition~\ref{def:galois}).  
Since a colimit of $k$-local spectra need not
be $k$-local, the spectrum $E$ is not necessarily $k$-local.

In \cite{Davis}, the second author developed a category of discrete
$G$-spectra and defined their homotopy fixed points (see also
\cite{Thomason}, 
\cite{JardinePresheaves}, \cite{MitchellHyper}, \cite{Goerss}, 
\cite{Jardineetale}).
In this paper, we examine $k$-local profinite  
$G$-Galois extensions $E$ of $A$
as objects in the category of discrete $G$-spectra, and we study the spectra
of $A$-module maps between the various 
homotopy fixed point spectra of $E$.  Unfortunately, to say meaningful
things it seems that we must impose more hypotheses on our profinite Galois
extensions.

\begin{assumption}\label{assumption}
In this paper, we shall only concern ourselves with
localizations $(-)_k$ which are given as a composite of two localization
functors $((-)_T)_M$, where $(-)_T$ is a smashing localization and $(-)_M$ is a
localization with respect to a finite spectrum $M$.  The spectra
$S$, $H\FF_p$, $E(n)$, and $K(n)$ are all examples of such 
localizations $k$ (see \cite{Bousfield}, \cite{Hovey}).
\end{assumption}

For a cofibrant commutative $S$-algebra $B$ and a cofibrant
commutative $B$-algebra $C$, the $k$-local Amitsur
derived completion $B^\wedge_{k,C}$ is the 
homotopy limit of the 
cosimplicial spectrum
$$ C_k \Rightarrow (C \wedge_B C)_k \Rrightarrow (C \wedge_B C
\wedge_B C)_k \cdots $$
(see, for example, \cite[Def.~8.2.1]{Rognes}).

\begin{defn} Let $E$ be a $k$-local profinite $G$-Galois extension of $A$.
\begin{enumerate}

\item
The extension $E$ is \emph{consistent} if the 
coaugmentation of the $k$-local 
Amitsur derived completion
$$ A \rightarrow A^\wedge_{k,E} $$
is an equivalence.  

\item
The extension $E$ is of \emph{finite virtual cohomological
dimension} (finite vcd) 
if the profinite group $G$ has finite vcd (i.e., $G$ has an open subgroup $U$ of finite
cohomological dimension:
there exists a $d$ such that $H^s_c(U;M) = 0$ for 
each $s > d$ and each discrete $U$-module $M$).

\end{enumerate}
\end{defn}

Assumption~\ref{assumption} ensures that 
if $E$ has finite vcd, then
the condition of $E$ being consistent is equivalent to requiring that the
map
$$ A \rightarrow (E^{hG})_k $$ 
is an equivalence.  This is proven as
Corollary~\ref{cor:consistent}. 

It then follows that the maps 
$$ E_\alpha \rightarrow (E^{hU_\alpha})_k $$
are equivalences (Lemma~\ref{lem:consistency}). 
The consistency hypothesis may be
unnecessary, since we do not know of any profinite Galois
extensions which are not consistent.

The main concern of this paper is the study of the intermediate homotopy
fixed point spectra $E^{hH}$ with respect to \emph{closed} subgroups $H$ of
$G$.  We prove the ``forward'' direction of the Galois correspondence.

\begin{thm*}[\ref{thm:Galois}]
Suppose that $E$ is a consistent profaithful $k$-local profinite $G$-Galois extension of
$A$ of finite vcd, and 
that $H$ is a closed subgroup of $G$.  
\begin{enumerate}

\item
The spectrum $E$ is $k$-locally $H$-equivariantly equivalent to a
consistent profaithful $k$-local  
$H$-Galois extension of $(E^{hH})_k$ of finite
vcd.

\item
If $H$ is a normal subgroup of $G$, then the spectrum $E^{hH}$ is
$k$-locally equivalent to a
profaithful $k$-local $G/H$-Galois extension of $A$.  
If the quotient $G/H$
has finite vcd, then this extension is consistent
(and of finite vcd) over $A$.

\end{enumerate}
\end{thm*}

\begin{rmk}
Note that the open subgroups of a profinite group $G$ 
are precisely the closed subgroups of
finite index.  Also, if $G$ has finite vcd, then it easily follows from
\cite[I.3.3]{Serre} that every closed subgroup also has finite vcd.
\end{rmk}

We also identify the function spectrum of $A$-module maps between any two
such homotopy fixed point spectra.

\begin{thm*}[\ref{thm:funcspec}]
Let $E$ be a consistent profaithful 
$k$-local profinite $G$-Galois extension of finite
vcd, and let $H$ and $K$ be closed subgroups 
of $G$.  Then 
there is an equivalence
\begin{equation}\label{eqn:funcspec}
F_A((E^{hH})_k, (E^{hK})_k) \simeq ((E[[G/H]])^{hK})_k.
\end{equation}
\end{thm*}

The spectrum $E[[G/H]]$ that appears on the right-hand side of
(\ref{eqn:funcspec}) 
is the \emph{continuous} $G$-spectrum with the diagonal action.
The case where $K = H = \{ e \}$ was
handled by Rognes \cite[(8.1.3)]{Rognes}.

In the context of Morava $E$-theory, 
(\ref{eqn:funcspec}) was proven in \cite{GHMR} 
under the additional
assumption that $K$ is finite,
and
it was suggested by the authors of \cite{GHMR} that (\ref{eqn:funcspec}) 
should be true
with this extra assumption removed.
Another source of motivation for this work arises from the fact that a 
special case of (\ref{eqn:funcspec}) (Corollary~\ref{cor:trivial}) 
was needed in an essential way 
by the first author in \cite{Behrens} 
(see \cite[Thm. 2.3.2, Cor. 2.3.3]{Behrens}).

One important example of a profinite Galois extension 
is given by Morava $E$-theory.
Let $k = K(n)$ be the $n$th Morava $K$-theory spectrum and let $A =
S_{K(n)}$ be the $K(n)$-local sphere spectrum.
Let $G = \GG_n$ be the $n$th extended Morava stabilizer group 
$\MS_n \rtimes \mathrm{Gal}(\mathbb{F}_{p^n}/\mathbb{F}_p)$.  
Let 
$E_n$ be the $n$th Morava $E$-theory spectrum, where 
$(E_n)_\ast=W(\mathbb{F}_{p^n})[[u_1, ..., u_{n-1}]][u^{\pm 1}]$.
Goerss and Hopkins \cite{GoerssHopkins}, 
building on work of Hopkins 
and Miller \cite{RezkHMT}, 
have shown that $\GG_n$ acts on $E_n$ by maps of commutative $S$-algebras.  
Devinatz and
Hopkins \cite{DevinatzHopkins} have given constructions of 
homotopy fixed point spectra $E_n^{dhH}$ for closed subgroups $H$
of $\GG_n$.  
In particular, they show that there is an equivalence
$$ E_n^{dh\GG_n} \simeq S_{K(n)}. $$
Thus, the homotopy fixed point spectra of $E_n$ are intimately related to the 
$n$th chromatic layer of the sphere spectrum.

Rognes \cite[Thm.~5.4.4, Prop.~5.4.9]{Rognes} 
proved for $U$ an open normal subgroup of $\GG_n$, that
the work of Devinatz and Hopkins 
\cite{LHS, DevinatzHopkins} shows that $E_n^{dhU}$ is a faithful 
$K(n)$-local
$\GG_n/U$-Galois extension of $S_{K(n)}$.
Therefore, the discrete $\GG_n$-spectrum
$$ F_n = \colim_{U \trianglelefteq_o \GG_n} E_n^{dhU} $$
is a profaithful $K(n)$-local profinite 
$\GG_n$-Galois extension of $S_{K(n)}$. 
Additionally, the profinite extension $F_n$ of $S_{K(n)}$ is 
consistent and has finite vcd (Proposition~\ref{prop:EGalois}).
The spectrum $E_n$ is recovered by the equivalence \cite{DevinatzHopkins}
$$ E_n \simeq (F_n)_{K(n)}. $$

As mentioned above, for any closed subgroup $H$ of $\GG_n$, Devinatz and 
Hopkins \cite{DevinatzHopkins} constructed the commutative $S$-algebra 
$E_n^{dhH}$. Further, they showed that $E_n^{dhH}$ behaves like a homotopy fixed point spectrum with respect to a continuous action of $H$. In more detail, 
\cite{DevinatzHopkins} showed that $E_n^{dhH}$ has the following properties: (a) there is a $K(n)$-local 
$E_n$-Adams spectral sequence 
\[H^s_c(H;\pi_t(E_n)) \Rightarrow \pi_{t-s}(E_n^{dhH}),\] where the $E_2$-term 
is the continuous cohomology of $H$, with coefficients in the profinite 
$H$-module $\pi_t(E_n)$, and this spectral sequence has the form of a 
descent spectral sequence; (b) when $H$ is finite, there is a weak equivalence 
$E_n^{dhH} \rightarrow E_n^{hH}$, and the descent spectral sequence for 
$E_n^{hH}$ is isomorphic to the spectral sequence in (a); and 
(c) $E_n^{dhH}$ is an $(N(H)/H)$-spectrum, where $N(H)$ is the normalizer of $H$ in $\GG_n$. 
\par
On the other hand, when $H$ is not finite, $E_n^{dhH}$ 
is not known to actually be the $H$-homotopy 
fixed point spectrum of $E_n$, because (a) it is not constructed with respect to a continuous $H$-action, and (b) it is not obtained by taking the total right derived 
functor of fixed points (and homotopy fixed points are, by definition, the total right derived 
functor of fixed points, in some sense - see \cite[Remark 8.4]{Davis} for the precise definition in the case of a continuous $H$-spectrum that arises from a tower of discrete 
$H$-spectra). To address this situation, in \cite{Davis}, the second author showed that $H$ does act continuously on $E_n$ and there is an actual $H$-homotopy fixed point spectrum $E_n^{hH}$, with a descent spectral sequence 
\[H^s_c(H;\pi_t(E_n)) \Rightarrow \pi_{t-s}(E_n^{hH}).\] 
 \par
From the above discussion,
we see from the properties of $E_n^{dhH}$ and $E_n^{hH}$ that
they should be equivalent to each other, and by a result in
the second author's thesis \cite{Davisthesis}, they are.
However, since this part of \cite{Davisthesis} 
was never published, 
we use the machinery of this paper to prove the equivalence of these two 
spectra. In more detail, we give proofs 
of the following two results (which originally appeared in
\cite{Davisthesis}). 
\begin{thm*}[\ref{dhGhG}]
For every closed subgroup $H$ of $\GG_n$, there is an equivalence
$$ E_n^{dhH} \simeq E_n^{hH} $$
between the Devinatz-Hopkins construction and the homotopy 
fixed points that are defined with respect to the continuous action of $H$.
\end{thm*}
The above theorem shows that $E_n^{dhH}$ can be referred to as a 
homotopy fixed point spectrum, whereas, previously, $E_n^{dhH}$ 
was only known to behave like a homotopy fixed point spectrum. 
\begin{thm*}[\ref{finite}, \ref{finite2}, \ref{ss}] 
Let $H$ be a closed subgroup of $\GG_n$ and let $X$ be a finite spectrum. Then 
there is an equivalence 
\[ E_n^{dhH} \wedge X \simeq (E_n \wedge X)^{hH} \] and 
the $K(n)$-local $E_n$-Adams spectral sequence for $\pi_\ast(E_n^{dhH} \wedge X)$ is 
isomorphic to the descent spectral sequence for $\pi_\ast((E_n \wedge X)^{hH})$ 
from the $E_2$-terms onward. 
In particular, 
\[(X)_{K(n)} \simeq (E_n \wedge X)^{h\GG_n}.\]
\end{thm*}

The paper is organized as follows.
Our notion of homotopy fixed point spectra 
uses the framework of equivariant spectra
(with respect to a profinite group) as developed
by the second author \cite{Davis}. 
The foundations in \cite{Davis} use Bousfield-Friedlander spectra.  
Since we need to work with structured ring spectra to do Galois theory, 
it is essential for this paper that we 
reformulate portions of \cite{Davis} in the context
of symmetric spectra.  A concise summary of these foundations 
appears in Section~\ref{sec:Gspt}.
In Section~\ref{sec:hfp}, we describe properties of the homotopy fixed point
functor.
In Section~\ref{sec:cont}, we describe continuous
$G$-spectra, generalizing somewhat the setting of \cite{Davis}.
In Section~\ref{sec:alg}, we explain how to extend
our constructions to categories of modules and commutative algebras of
spectra.
In Section~\ref{sec:progalois}, we explain how profinite Galois extensions give
rise to discrete $G$-spectra, and we show that the homotopy fixed points with
respect to open subgroups of the Galois group give rise to intermediate
finite Galois extensions.
In Section~\ref{sec:closed}, we prove our results concerning the homotopy
fixed point spectra with respect to closed subgroups of the Galois group.
In Section~\ref{sec:Etheory}, we show that the hypotheses on profinite
Galois extensions which we require are satisfied by Morava $E$-theory.  We
then apply our machinery to show that the Devinatz-Hopkins homotopy fixed
points agree with the second author's homotopy fixed points, and deduce some
corollaries.
\vspace{10pt}

\noindent
\textbf{Acknowledgments.} 
The first author benefited from the input of Halvard Fausk, Paul Goerss,
and Daniel Isaksen.
The second author thanks Paul Goerss for
helpful discussions, when he was a Ph.D. student, 
regarding the results in
Sections~\ref{sec:EGalois} and \ref{sec:comparison}. 
Also, the
second author is grateful to Paul for later helpful
conversations, to Mark Hovey for providing some intuition
related to Theorem~\ref{thm:sigmaspmodel}, and John
Rognes for a helpful discussion
regarding group actions.
The second author spent
several weeks in the summer of 2008
in the Department of Math at Rice
University working on this paper, and
he thanks the department for its
hospitality.  The authors would also like to express their thanks to the
referee, for suggesting numerous improvements to this paper.

\section{Discrete symmetric $G$-spectra}\label{sec:Gspt}

Let $G$ be a profinite group.
We begin this section by describing the basic categories of discrete
$G$-objects that will be used
in this paper.  We then describe and compare 
the model structures on the categories of discrete $G$-objects in
Bousfield-Friedlander and symmetric spectra.  We end this section with
descriptions of some basic constructions in the category of discrete
$G$-spectra.  More detailed accounts of some of 
these model categories and constructions can be found in \cite{Davis}. 

\subsection{Simplicial discrete $G$-sets}\label{sec:Set_G}

A $G$-set $Z$ is said to be discrete if, for every element $z \in Z$, the
stabilizer $\Sta_G(z)$ is open in $G$.  We may express this condition by
saying that $Z$ is the colimit of its fixed points:
$$ Z = \colim_{U \le_o G} Z^U, $$
where the colimit is taken over all open subgroups. These conditions are 
equivalent to the condition that the action map $G \times Z \rightarrow Z$ 
is continuous, when $Z$ is given the discrete topology.
A simplicial discrete $G$-set is a simplicial object in the category of
discrete $G$-sets.

Goerss showed that the category $s\Set_G$ of simplicial discrete $G$-sets
admits a model category structure \cite{Goerss}.

\begin{thm}[Goerss]
The category $s\Set_G$ admits a model category structure where
\begin{itemize}
\item the cofibrations are the monomorphisms,
\item the weak equivalences are those morphisms which are weak equivalences
on underlying simplicial sets,
\item the fibrations are determined.
\end{itemize}
\end{thm}

\begin{lem}
The model structure on $s\Set_G$ is left proper and cellular.
\end{lem}

\begin{proof}
The model structure is left proper because the cofibrations and weak
equivalences are precisely the cofibrations and weak equivalences on the
underlying simplicial sets, and the model category structure on simplicial
sets is left proper.  The model structure in \cite{Goerss} is cofibrantly
generated, with generating cofibrations $I$ and generating trivial
cofibrations $J$, where:
\begin{align*}
I & = \{ G/U \times \partial \Delta^n \hookrightarrow G/U \times \Delta^n \:
: \: U \le_o G, \: n \ge 0 \}, \\
J & = \left\{ A \xrightarrow{j} B \: : \: 
\begin{array}{l}
\text{$j$ is a trivial cofibration,} \\
\# B \le \alpha.
\end{array}
\right\}.
\end{align*} 
Here, $\alpha$ is a fixed infinite cardinal greater than the cardinality of $G$
and $\# B$ denotes the cardinality of the set of non-degenerate simplices
of $B$ (see the proof of Lemma~1.13 of \cite{Goerss}).  The axioms of being
cellular are immediately verified from this description of the generating
cofibrations.
\end{proof}

The category $(s\Set_G)_{*}$ of pointed simplicial discrete $G$-sets, being
an under-category, inherits a model structure from $s\Set_G$.  The
cofibrations, weak equivalences, and fibrations are detected on the level
of underlying simplicial discrete $G$-sets.  If $K$ and $L$ are pointed 
simplicial
discrete $G$-sets, then their smash product
$$ K \wedge L $$
is easily seen to be a simplicial discrete $G$-set.  The smash product
gives the category $(s\Set_G)_*$ a symmetric monoidal structure.  (It does
\emph{not} extend to a closed symmetric monoidal structure.)

\begin{lem}\label{lem:cellular}
The model category structure on $(s\Set_G)_*$ is left proper and cellular.
With respect to the symmetric monoidal structure given by the smash product,
the model category $(s\Set_G)_*$ is a symmetric monoidal model
category.
\end{lem}

\begin{proof}
Left properness follows from the fact that $s\Set_G$ is left proper.  The
model structure on $(s\Set_G)_*$ is cofibrantly generated with generating
cofibrations (respectively generating trivial cofibrations) $I_+$
(respectively $J_+$).  Here, $I_+$ and $J_+$ are the sets of maps obtained
from $I$ and $J$ by adding a disjoint basepoint on which $G$ acts
trivially.  The axioms of being a symmetric monoidal
model category are easily verified.
\end{proof}

\subsection{Discrete $G$-spectra}\label{sec:Sp_G}

Define the category of discrete $G$-spectra $\Sp_G$ 
to be the category of 
Bousfield-Friedlander spectra of simplicial discrete
$G$-sets.  An object $X \in \Sp_G$ consists of a sequence
$\{X_i\}_{i \ge 0}$, where each $X_i$ is a pointed 
simplicial discrete $G$-set,
together with $G$-equivariant maps
$$ \sigma_i : S^1 \wedge X_i \rightarrow X_{i+1}. $$
Here, $S^1$ 
is given the trivial $G$-action.

A map $f :X \rightarrow Y$ of
discrete $G$-spectra is a sequence of $G$-equivariant maps of pointed 
simplicial
sets $f_i : X_i \rightarrow Y_i$ which are compatible with the 
spectrum structure maps.

In \cite{Davis}, the second author studied the 
following model structure.

\begin{thm}\label{thm:spmodel}
The category $\Sp_G$ admits a model structure where
\begin{itemize}
\item the cofibrations are the cofibrations of underlying Bousfield-Friedlander
spectra,
\item the weak equivalences are the stable weak equivalences of the underlying
Bousfield-Friedlander spectra,
\item the fibrations are determined.
\end{itemize}
\end{thm}

The method used in \cite{Davis} was to transport a Jardine model structure
on presheaves of spectra on an appropriate site.  However, an alternative
approach is given below using the machinery of M.~Hovey \cite{Hoveyspec}.

\begin{proof}
Observe that $(s\Set_G)_*$ satisfies the conditions of Definition~3.3 of
\cite{Hoveyspec}.  Therefore, the category $\Sp_G$ of spectra of
simplicial discrete $G$-sets admits a stable model category structure where
\begin{itemize}
\item the cofibrations are those morphisms $A \rightarrow B$ where the
induced maps
\begin{align*}
A_0 & \rightarrow B_0 \\
A_i \cup_{S^1 \wedge A_{i-1}} S^1 \wedge B_{i-1} & \rightarrow B_i \qquad n
\ge 1
\end{align*}
are cofibrations.
\item the fibrant objects $X$ are those spectra for which
\begin{enumerate}
\item the spaces $X_i$ are fibrant as simplicial discrete $G$-sets, 
\item the maps $X_i \rightarrow \Omega X_{i+1}$ are weak equivalences.
\end{enumerate}
\item the weak equivalences $f: X \rightarrow Y$ between fibrant objects
are those $f$ for which the maps $f_i: X_i \rightarrow Y_i $ are all weak
equivalences.
\end{itemize}
Clearly the cofibrations of $\Sp_G$ are the maps which are cofibrations of
underlying Bousfield-Friedlander spectra.  We are left with verifying that
the weak equivalences in $\Sp_G$ 
are precisely the stable equivalences of underlying
Bousfield-Friedlander spectra.

The forgetful functor
$$ \mathcal{U}: s\Set_G \rightarrow s\Set $$
from simplicial discrete $G$-sets to simplicial sets is a left Quillen
functor (it preserves cofibrations and trivial cofibrations, and is left
adjoint to the functor $\CoInd_1^G$ of Section~\ref{sec:homo}).
By Proposition~5.5 of \cite{Hoveyspec}, the induced forgetful functor
$$ \mathcal{U}: \Sp_G \rightarrow \Sp $$
is a left Quillen functor.  
Let $(-)_{fG}$ denote the functorial fibrant replacement in $\Sp_G$, so
that there are natural trivial cofibrations
$$ \alpha_{G,X} : X \rightarrow X_{fG}. $$

Suppose that $\phi: X \rightarrow Y$ is a morphism in $\Sp_G$, and consider
the following diagram:
$$
\xymatrix{
X \ar[r]^{\phi} \ar[d]_{\alpha_{G,X}} &
Y \ar[d]^{\alpha_{G,Y}} 
\\
X_{fG} \ar[r]_{\phi_{fG}} &
Y_{fG}
}
$$
We claim that $\phi$ is a stable equivalence in $\Sp_G$ if and only if the
induced morphism $\mathcal{U}\phi$ is a stable equivalence of underlying
Bousfield-Friedlander spectra.  Because $\mathcal{U}$ is a left Quillen functor, and the morphisms
$\alpha_{G,-}$ are trivial cofibrations, we may conclude that the morphisms
$\mathcal{U}\alpha_{G,-}$ induce stable equivalences of underlying
Bousfield-Friedlander spectra.  

Since the underlying spectrum of a fibrant object in $\Sp_G$ is fibrant in
$\Sp$, and since the stable equivalences between fibrant objects in $\Sp_G$
and $\Sp$ are precisely the levelwise equivalences, we see that $\phi_{fG}$
is a stable equivalence if and only if $\mathcal{U}\phi_{fG}$ is a stable
equivalence.  Therefore, we may deduce that $\phi$ is a stable equivalence
if and only if $\mathcal{U}\phi$ is a stable equivalence.
\end{proof}

\subsection{Discrete symmetric $G$-spectra}\label{sec:SigmaSp_G}

Let $\Sigma\Sp$ denote the category of symmetric spectra
(see \cite{HSS}, \cite{MMSS} 
for accounts of symmetric spectra).
Define the category of discrete symmetric $G$-spectra $\Sigma\Sp_G$ 
to be the category of 
symmetric spectra of simplicial discrete
$G$-sets.  
Let $\Sigma_i$ denote the $i$th symmetric group.
An object $X \in \Sigma\Sp_G$ consists of a sequence
$\{X_i\}_{i \ge 0}$, where each $X_i$ is a pointed 
simplicial discrete $G \times \Sigma_i$-set,
together with suitably compatible
$G \times \Sigma_{i} \times \Sigma_j$-equivariant maps
$$ \sigma_{i,j} : S^i \wedge X_j \rightarrow X_{i+j}. $$
Here, $S^i = (S^1)^{\wedge i}$ 
is given the trivial $G$-action, and $\Sigma_i$ permutes the
factors of the smash product $(S^1)^{\wedge i}$. 
When $G$ is finite, a discrete symmetric $G$-spectrum is 
simply a na\"\i ve symmetric $G$-spectrum, and not a genuine equivariant 
symmetric $G$-spectrum in the sense of \cite{Mandell}.

Maps $f :X \rightarrow Y$ of
discrete symmetric $G$-spectra are sequences of 
$G \times \Sigma_i$-equivariant maps of pointed simplicial sets 
$f_i : X_i \rightarrow Y_i$
which are compatible with the spectrum structure maps.

For a cofibrantly generated model category $\mathcal{C}$, let
$\mathcal{C}^{\Sigma_j}$ denote the diagram category of
$\Sigma_j$-equivariant objects in $\mathcal{C}$ with the projective model
structure (\cite[Thm.~11.6.1]{Hirschhorn}).  

\begin{lem}\label{lem:diagmodelstructure}
In the projective model category structure on $(s\Set_G)_*^{\Sigma_j}$:
\begin{itemize}
\item the cofibrations are those maps that are projective cofibrations in the underlying
category $s\Set_*^{\Sigma_j}$,
\item the weak equivalences are those maps that are weak equivalences in the underlying
category $s\Set_*^{\Sigma_j}$,
\item the fibrations are determined.
\end{itemize}
\end{lem}

\begin{proof}
The statement concerning weak equivalences follows immediately from the
definition of the weak equivalences in the projective model structure.
The projective cofibrations in $(s\Set_G)_*^{\Sigma_j}$ are generated by
the set
$$ I_+^{\Sigma_j} = 
\{ (\Sigma_j \times G/U \times \partial \Delta^n)_+ \hookrightarrow 
(\Sigma_j \times G/U \times \Delta^n)_+ \: : \: U \le_o G, \: n \ge 0 \}. $$
Using the relative skeletal filtration, it is easy to see that the
class of cofibrations generated by the set $I_+^{\Sigma_j}$ are the
monomorphisms which are relative free $\Sigma_j$-complexes.  However, these
are precisely the projective cofibrations in $s\Set_*^{\Sigma_j}$.
\end{proof}

\begin{thm}\label{thm:sigmaspmodel}
The category $\Sigma\Sp_G$ admits a left proper cellular model structure where
\begin{itemize}
\item the cofibrations are the cofibrations of underlying symmetric spectra,
\item the weak equivalences are the stable weak equivalences of underlying
symmetric spectra,
\item the fibrations are determined.
\end{itemize}
\end{thm}

\begin{proof}
Observe that $(s\Set_G)_*$ satisfies the conditions of Definition~8.7 of
\cite{Hoveyspec}.  Therefore, the category $\Sigma\Sp_G$ of symmetric 
spectra of
simplicial discrete $G$-sets admits a stable model category structure where
\begin{itemize}
\item the cofibrations are those morphisms $A \rightarrow B$ where the
induced maps
\begin{align*}
A_0 & \rightarrow B_0 \\
A_i \cup_{L_i A} L_i B & \rightarrow B_i \qquad i
\ge 1
\end{align*}
are projective cofibrations in $(s\Set_G)^{\Sigma_j}_*$, where $L_n$ is the
latching object of \cite[Def.~8.4]{Hoveyspec}.
\item the fibrant objects $X$ are those spectra for which
\begin{enumerate}
\item the spaces $X_i$ are fibrant as simplicial discrete $G$-sets, 
\item the maps $X_i \rightarrow \Omega X_{i+1}$ are weak equivalences.
\end{enumerate}
\item the weak equivalences $f: X \rightarrow Y$ between fibrant objects
are those $f$ for which the maps $f_i: X_i \rightarrow Y_i $ are all weak
equivalences of underlying simplicial sets.
\end{itemize}
The cofibrations are immediately seen to be the cofibrations of underlying
symmetric spectra, using Lemma~\ref{lem:diagmodelstructure}.  
The verification that the weak equivalences are
precisely the
stable equivalences of underlying symmetric spectra is identical to
the argument given in the proof of Theorem~\ref{thm:spmodel}.
\end{proof}

We have the following proposition, which helps
to translate results in
the category $\Sp_G$, to the category $\Sigma\Sp_G$.

\begin{prop}
There is a Quillen equivalence
$$
\xymatrix{
\mathbb{V}_G : \Sp_G \rightleftarrows \Sigma\Sp_G : \mathbb{U}_G  
}
$$
where $\mathbb{U}_G$ is the forgetful functor.
\end{prop}

\begin{proof}
The functor $\mathbb{V}_G$ 
is the left adjoint of $\mathbb{U}_G$: it is explicitly given by
(see \cite[Sec.~4.3]{HSS})
$$ \mathbb{V}_G(X) = S \otimes_{T(\mathbb{G}_1S^1)} \mathbb{G}X $$
where $\mathbb{G}X$ is the symmetric sequence given by
$$ (\mathbb{G}X)_i = (\Sigma_i)_+ \wedge X_i $$
(where $G$ acts through its action on $X_i$), $\mathbb{G}_1S^1$ is the
symmetric sequence 
$$ (*, S^1, *, *, \cdots ) $$ 
(with trivial $G$ action), 
$T(\mathbb{G}_1S^1)$ is the
free monoid on $\mathbb{G}_1S^1$ with respect to $\otimes$ (which
gives the symmetric
monoidal structure on symmetric sequences), 
and $S$ is the usual symmetric
sequence $(S^0, S^1, S^2, \cdots)$.
We have the following commutative diagram of functors
$$
\xymatrix{
\Sp_G \ar@<0.5ex>[r]^{\mathbb{V}_G} \ar[d]_{\mathcal{U}} & 
\Sigma\Sp_G \ar@<0.5ex>[l]^{\mathbb{U}_G} \ar[d]^{\mathcal{U}_\Sigma}
\\
\Sp \ar@<0.5ex>[r]^{\mathbb{V}} & 
\Sigma\Sp \ar@<0.5ex>[l]^{\mathbb{U}} 
}
$$
where the bottom row is the Quillen equivalence of \cite[Sec.~4]{HSS}.   

The functor $\mathbb{V}_G$ preserves cofibrations and trivial cofibrations
because the functors $\mathcal{U}$ and $\mathcal{U}_\Sigma$ 
reflect and detect 
cofibrations and trivial cofibrations, and the functor $\mathbb{V}$ is a
left Quillen functor.  Therefore $(\mathbb{V}_G, \mathbb{U}_G)$ forms a
Quillen pair.

To show that $(\mathbb{V}_G, \mathbb{U}_G)$ is a Quillen equivalence, we
must show that for all cofibrant $X$ in $\Sp_G$ and all fibrant $Y$ in
$\Sigma\Sp_G$, a morphism
$$ f: \mathbb{V}_GX \rightarrow Y $$
is a weak equivalence if and only if its adjoint
$$ \td{f}: X \rightarrow \mathbb{U}_GY $$
is a weak equivalence.  However, since the functors $\mathcal{U}$ and
$\mathcal{U}_\Sigma$ reflect
and detect weak equivalences, it suffices to show that:
$$
\left\{
\begin{array}{c}
\mathcal{U}_\Sigma f : \mathbb{V} \mathcal{U}X \rightarrow
\mathcal{U}_\Sigma Y \\
\text{is a weak equivalence}
\end{array}
\right\}
\quad \text{if and only if}
\quad
\left\{ 
\begin{array}{c}
\mathcal{U}\td{f} : \mathcal{U}X \rightarrow
\mathbb{U} \mathcal{U}_\Sigma Y \\
\text{is a weak equivalence}
\end{array}
\right\}.
$$
This follows from the fact that $\mathcal{U}$ preserves cofibrations
(Theorem~\ref{thm:spmodel}), $\mathcal{U}_\Sigma$ preserves fibrant objects
(proof of Theorem~\ref{thm:sigmaspmodel}), and $(\mathbb{V}, \mathbb{U})$
form a Quillen equivalence \cite[Thm.~4.2.5]{HSS}.
\end{proof}

For the rest of this paper, we shall be working in the world of symmetric 
spectra, and shall refer to a symmetric spectrum as simply a spectrum.
For a spectrum $X$, we shall always use $\pi_*X$ to refer to its
\emph{true} homotopy groups (the maps $[S^t, X]$ in the stable homotopy
category) and not the \emph{na\"ive} homotopy groups (the colimit of the
homotopy groups of $X_i$) \cite{Schwedepi}. 

\subsection{Mapping spectra}\label{sec:Map}

Let $K$ and $L$ be discrete $G$-sets.  Then the set
of (non-equivariant) functions $\Map(K,L)$ is a $G$-set with $G$ acting by
conjugation. This action is given as follows: 
for $g \in G$ and $f \in \Map(K,L)$, $g\cdot
f$ is the map
$$ (g\cdot f)(z) = gf(g^{-1}z). $$
Observe that $\Map(K,L)$ is not in general a discrete $G$-set, but it is 
if $K$ is finite.

For a finite set $K$ and a spectrum $X$, we define the mapping spectrum
$\Map(K,X)$ to be the spectrum whose $m$th space is
given by
$$ \Map(K,X)_m = \Map(K,X_m), $$ where the $n$-simplices of $\Map(K,X_m)$
is the set $\Map(K,(X_m)_n).$ If $X$ is a discrete $G$-spectrum, and $K$ is
a finite discrete $G$-set, then the above definitions combine to give
that
$\Map(K,X)$ is a discrete $G$-spectrum.
 
If $K = \lim_\alpha K_\alpha$ is a profinite set and $X$ is a spectrum, then 
the spectrum of
continuous maps is the spectrum
$$ \Map^c(K,X) = \colim_\alpha \Map(K_\alpha,X). $$  
If $K$ is a continuous $G$-space, with each $K_\alpha$ a discrete $G$-set,
and $X$ is a discrete $G$-spectrum, then $\Map^c(K,X)$ is a 
discrete $G$-spectrum.

\begin{lem}\label{lem:Mapexact}
Let $K = \lim_\alpha K_\alpha$ be a profinite set, where each of the
$K_\alpha$ is
finite, and each of the maps $K_\alpha \rightarrow K_\beta$ 
in the pro-system is a
surjection.  The functor
$$ \Map^c(K,-): \Sigma\Sp \rightarrow \Sigma\Sp $$
preserves stable equivalences.
\end{lem}

\begin{proof}
In \cite{Schwede}, it is shown that stable equivalences are preserved under
finite products.  The argument goes as follows: 
the canonical map from a finite wedge to
a finite product is a $\pi_*$-isomorphism, hence a stable equivalence, and
stable equivalences are preserved under finite wedges (this is easily
checked from the definition of stable equivalences, 
by mapping into injective $\Omega$-spectra).  Therefore, for each $\alpha$, the
functor
$$ X \mapsto \Map(K_\alpha, X) \cong \prod_{K_\alpha} X $$
preserves stable equivalences.
Since we have assumed that the morphisms $K_\alpha \rightarrow K_\beta$ are
surjections, the induced morphisms
$$ \Map(K_\beta, X) \rightarrow \Map(K_\alpha, X) $$
are levelwise monomorphisms for every $X$.  

The category of symmetric
spectra possesses an \emph{injective} stable model structure, where the
injective 
cofibrations are the levelwise monomorphisms and the weak equivalences are
the stable equivalences (see \cite[p.~199]{HSS}).  
The directed system
$$ \{ \Map(K_\alpha, X) \} $$
is a directed system of injective cofibrations between injectively
cofibrant objects (every object is cofibrant in the injective model
structure).
The colimit
$$ \Map^c(K, X) = \colim_\alpha \Map(K_\alpha, X) $$
may be computed on a cofinal $\lambda$-sequential subcategory of the indexing
category of the system $\{K_\alpha\}$, for some ordinal $\lambda$.  
We deduce, by 
\cite[Prop.~17.9.1]{Hirschhorn}, that the functor
$$ X \mapsto \Map^c(K, X) = \colim_\alpha \Map(K_\alpha, X) $$
preserves stable equivalences.
\end{proof}

If $A$ is an abelian group, endowed with the discrete topology, 
and $K = \lim_\alpha K_\alpha$ is a profinite set,
the continuous maps $\Map^c(K,A)$ are given by
$$ \Map^c(K,A) = \colim_\alpha \Map(K_\alpha,A). $$

\begin{lem}\label{lem:piMap}
Let $X$ be an object of $\Sigma\Sp$, and let 
$K$ be a profinite set satisfying the
hypotheses of Lemma~\ref{lem:Mapexact}. 
Then there is an isomorphism
$$ \pi_*\Map^c(K,X) \cong \Map^c(K, \pi_*(X)) $$
\end{lem}

\begin{proof}
By Lemma~\ref{lem:Mapexact} it suffices to assume that $X$ is fibrant.  The
result then follows from Corollary~\ref{cor:hocolim}.
\end{proof}

\subsection{Permutation spectra}

Let $K$ be a discrete $G$-set.  Then for $X$ a discrete
$G$-spectrum, we may define the permutation spectrum $X[K]$ to be
the spectrum whose $n$th space is given by
$$ X[K]_n = X_n \wedge K_+. $$
We let $G$ act on the spectrum $X[K]$ through the diagonal action.

\begin{lem}\label{lem:X[K]discrete}
The spectrum $X[K]$ is a discrete $G$-spectrum.
\end{lem}

\begin{proof}
Note that 
\begin{align*}
X_n \wedge K_+ & \cong (\colim_{N \trianglelefteq_o G} 
X_n^N) \wedge (\colim_{N' \trianglelefteq_o G} K^{N'}_+)
\\ & \cong \colim_{N,N' \trianglelefteq_o G} (X_n^N \wedge K^{N'}_+)
\end{align*} 
is a 
simplicial discrete $G$-set, with $G$ acting diagonally, 
since the simplicial set $X_n^N \wedge K^{N'}_+$ has a 
diagonal $G/(N \cap N')$-action and the group $G/(N \cap N')$ is finite. 
Thus, 
the spectrum $X[K]$ is a discrete $G$-spectrum.
\end{proof}

\subsection{Smash products}

Given discrete $G$-spectra $X$ and $Y$, we define their 
smash product
$$ X \wedge Y $$
to be the smash product of the underlying symmetric spectra with $G$ acting
diagonally.  
Since the smash product commutes 
with colimits, it follows, as in the proof of Lemma~\ref{lem:X[K]discrete}, 
that $X \wedge Y$ is a discrete $G$-spectrum. 
Also, 
if $K$ is a discrete $G$-set, 
then there is a $G$-equivariant isomorphism 
$$ X \wedge S[K] \cong X[K], $$
where the sphere spectrum $S$ has trivial $G$-action.

\section{Homotopy fixed points of discrete $G$-spectra}\label{sec:hfp}

Much of the material in this section is assembled from  
\cite{Thomason}, 
\cite{JardineSS}, \cite{Goerss}, \cite{Jardineetale},
\cite{MitchellHyper}, 
\cite{JardineNATO}, and \cite{Davis}.
Let $G$ be a profinite group.
We begin this section with an account of the model category theoretic
definition of $G$-homotopy fixed points.  We
then describe the comparison with hypercohomology spectra.  Finite index
restriction and induction functors, as well as iterated homotopy fixed
points for finite index subgroups are then discussed.  
We explain how continuous homomorphisms of groups induce various 
``change of group functors,''
of which induction,
coinduction, fixed points, and restriction functors are all special cases.
We then describe the various technical 
difficulties related to the homotopy fixed point
construction for closed subgroups of $G$.  The technical difficulties are
observed to vanish if $G$ has finite cohomological dimension.

As alluded to above, Sections \ref{sec:iterate}, \ref{sec:fcditerate}, and 
\ref{sec:difficulties} discuss the construction of iterated homotopy fixed 
points. Much of this material overlaps with portions of \cite{iterated}: 
it was necessary 
to repeat some of the material from \cite{iterated}, so that 
certain issues are clear and to give a 
context for the results of Section \ref{sec:Galoisiterate}. 

We note that, as explained in
Section 2.3, ``spectrum'' means ``symmetric
spectrum,'' so that, for example, a ``discrete
$G$-spectrum'' is a ``discrete $G$-symmetric
spectrum.'' 

\subsection{The homotopy fixed point 
spectrum}\label{sec:fixedpoint}

For a discrete $G$-spectrum $X$, we define the fixed point spectrum by
taking the fixed points levelwise:
$$ (X^G)_i = (X_i)^G. $$
The $G$-fixed points functor is right adjoint to the functor
${\it triv}$,
which
associates to a spectrum $X$ the discrete $G$-spectrum $X$,
where $X$ now has the trivial
$G$-action:
$$ {\it triv}: \Sigma\Sp \rightleftarrows \Sigma\Sp_G : (-)^G. $$

\begin{lem}\label{lem:Quillenfixedpoint}
The adjoint functors $({\it triv}, (-)^G)$ form a Quillen pair.
\end{lem}

\begin{proof}
The functor ${\it triv}$ preserves cofibrations and weak equivalences.
\end{proof}

Let $\alpha_{G,X}: X \rightarrow X_{fG}$ denote a functorial 
fibrant replacement functor
for the model category $\Sigma\Sp_G$, where $\alpha_{G,X}$ is a trivial
cofibration of discrete $G$-spectra.
The homotopy fixed point functor $(-)^{hG}$
is the Quillen right derived functor of $(-)^G$, and is thus given by
$$ X^{hG} = (X_{fG})^G. $$

\subsection{Hypercohomology spectra}

The functor $\Gamma_G = \Map^c(G, -)$ is a coaugmented comonad on the 
category of spectra,
with coproduct
$$ \psi: \Gamma_G = \Map^c(G,-) \rightarrow 
\Map^c(G\times G, -) \cong \Gamma_G \circ \Gamma_G $$
induced from the product on $G$, counit
$$ \Gamma_G = \Map^c(G,-) \rightarrow \Map^c({\it pt}, -) \cong \mathrm{Id} $$
induced from the unit on $G$, and coaugmentation
$$ \mathrm{Id} \rightarrow \Map^c(G, -)$$
given by the inclusion of the constant maps.

Discrete $G$-spectra are coalgebras over
the comonad $\Gamma_G$ (this follows from
considering the map of spectra
\begin{align*}
X & \rightarrow \Gamma_G(X), \\
x & \mapsto (g \mapsto g\cdot x),
\end{align*}
for any discrete $G$-spectrum $X$).

Let $\mathcal{C}$ and $\mathcal{D}$ be categories, and suppose that
$\Gamma$ is a comonad in $\mathcal{C}$.  Dualizing Definition~9.4 of
\cite{May}, there is a notion of a $\Gamma$-functor
$$ F : \mathcal{C} \rightarrow \mathcal{D}. $$
Let $Y$ be a $\Gamma$-coalgebra.
Dualizing Construction~9.6 of \cite{May}, one may associate to $(F,
\Gamma, Y)$ a cosimplicial object $C^\bullet(F, \Gamma, Y)$ 
in $\mathcal{D}$ 
(the comonadic cobar construction), given by
$$ C^s(F, \Gamma, Y) = F \Gamma^s Y. $$
If $\Gamma$ is a coaugmented comonad, then the identity functor
$\mathrm{Id}_{\mathcal{C}}$ is a $\Gamma$-functor.  We will let
$\Gamma^\bullet Y$ denote the cosimplicial object
$$ \Gamma^\bullet Y = C^\bullet(\mathrm{Id}_{\mathcal{C}}, \Gamma, Y) $$
in $\mathcal{C}$.

In \cite{Davis},
the homotopy fixed point spectrum was shown to have the
following alternate description, provided $G$ is sufficiently nice (see
also \cite{MitchellHyper}, \cite{Goerss},
\cite{JardinePresheaves}).

\begin{thm}\label{thm:cochaincomplex}
Suppose that $G$ has finite vcd, and that $X$ is a discrete $G$-spectrum.  
Then there is an
equivalence
\begin{align*}
X^{hG}
& \simeq \holim_\Delta \Gamma^\bullet_G X \\
& = \HH_c(G; X),
\end{align*}
where $\HH_c(G; X)$ is the hypercohomology spectrum.
\end{thm}

\begin{proof}
In \cite[Thm.~7.4]{Davis} it is proven that there is an
equivalence
$$ 
X^{hG}
\simeq \holim_\Delta \Gamma^\bullet_G X_{fG}.
$$
(The
cosimplicial object defining the hypercohomology spectrum is different, but
isomorphic to that appearing in \cite{Davis}.)  
The result follows once we establish that the 
map induced from fibrant replacement
$$ \holim_\Delta \Gamma^\bullet_G X 
\rightarrow \holim_\Delta \Gamma^\bullet_G X_{fG} $$
is an equivalence.
This map is deduced to be an equivalence from the following
facts: (a) 
the fibrant replacement map $X \rightarrow
X_{fG}$ is an equivalence; (b) the functor $\Gamma_G$ preserves
equivalences, by Lemma~\ref{lem:Mapexact}; (c) 
the homotopy limit
construction sends levelwise equivalences to equivalences, since 
it is a Quillen derived functor.
\end{proof}

\subsection{Iterated homotopy fixed points}\label{sec:iterate}

Let $U$ be an open subgroup of $G$, so that $G/U$ is finite.

\begin{prop}\label{prop:iterate}
Let $X$ be a discrete $G$-spectrum.
\begin{enumerate}
\item If $U$ is normal in $G$, the 
$U$-fixed point spectrum $(X_{fG})^{U}$ is fibrant
as a discrete $G/U$-spectrum.

\item The fibrant discrete $G$-spectrum $X_{fG}$ is fibrant as a discrete
$U$-spectrum.

\item If $U$ is normal in $G$, the homotopy fixed point spectrum $X^{hU}$ is a
$G/U$-spectrum.

\item If $U$ is normal in $G$, 
there is an equivalence $X^{hG} \simeq (X^{hU})^{hG/U}$.
\end{enumerate}
\end{prop}

\begin{proof}

To prove (1), observe that since $U$ is normal, for any discrete
$G$-spectrum $Y$, the $U$-fixed point spectrum $Y^U$ is naturally a 
$G/U$-spectrum.  
There is an adjoint pair of functors $(\Res_{G/U}^G, (-)^U)$
$$ \Res_{G/U}^G : \Sigma\Sp_{G/U} \rightleftarrows \Sigma\Sp_{G} : (-)^{U}, $$
where $\Res^G_{G/U}$ is defined by 
restriction along the quotient homomorphism 
$G \rightarrow G/U$.
Since $\Res_{G/U}^G$ preserves cofibrations and weak equivalences, the
functor $(-)^U$ preserves fibrant objects.

We verify (2) in a similar way 
(compare with \cite[Rmk.~6.26]{Jardineetale}). 
Define the 
induction
functor on a discrete $U$-spectrum $Y$ to be
$$ \Ind_U^G Y = G_+ \wedge_U Y. $$
Here, $G_+ \wedge_U Y$ is formed by regarding $G$ and $U$ as
discrete groups, but this is easily seen to produce a discrete
$G$-spectrum,
since $U$ is a subgroup of finite index.
The induction functor is the left adjoint of an adjunction
$$ \Ind_U^G: \Sigma\Sp_{U} \rightleftarrows \Sigma\Sp_{G} : \Res_G^U, $$
where $\Res_G^U$ is restriction along the inclusion $U \hookrightarrow G$.
Since non-equivariantly there is an isomorphism
$$ \Ind_U^G Y \cong G/U_+ \wedge Y, $$
we see that $\Ind_U^G$ preserves cofibrations and weak equivalences, from
which it follows that $\Res_G^U$ preserves fibrant objects.

By (2), $X_{fG}$ is a fibrant 
discrete $U$-spectrum. Also, $X \rightarrow X_{fG}$ is a trivial cofibration 
of spectra and it is $U$-equivariant, so it is a trivial cofibration in 
$\Sigma\mathrm{Sp}_U$. Thus, 
\[X^{hU} = (X_{fG})^U,\] 
which is a $G/U$-spectrum.
This proves (3).

(4) is proven using our fibrancy results.  There are equivalences:
$$
X^{hG} \simeq X_{fG}^G = (X_{fG}^U)^{G/U} \simeq (X^{hU})^{hG/U}.
$$
\end{proof}

\subsection{Homomorphisms of groups}\label{sec:homo}

If $f: H \rightarrow G$ is a continuous homomorphism of profinite groups,
we may regard discrete $G$-sets as discrete $H$-sets.
For a discrete $H$-set $Z$, we define the coinduced discrete $G$-set by
$$ 
f_*Z = \CoInd_H^G Z = \Map_H^c(G, Z) = \colim_{U \trianglelefteq_o G}
\Map_H(G/U, Z), $$
where the $G$-action is defined by
the formula
$$ (g \cdot \alpha)(g'U) = \alpha(g'gU), $$
for $g \in G$ and $\alpha \in \Map_H(G/U, Z)$. 
This construction extends to simplicial discrete $G$-sets and discrete
$G$-spectra in the obvious manner to give a functor
$$ f_*: \Sigma\Sp_H \rightarrow \Sigma\Sp_G. $$
The functor $f_*$ is the right adjoint of an adjoint pair $(f^*, f_*)$,
where
$$ f^* = \Res_G^H : \Sigma\Sp_G \rightarrow \Sigma\Sp_H $$
is the restriction functor along the homomorphism $f$.
Since $f^*$ clearly preserves cofibrations and weak equivalences, we have
the following lemma.

\begin{lem}\label{lem:homopair}
The adjoint functors $(f^*, f_*)$ form a Quillen pair.  In particular,
$f_*$ preserves fibrations and weak equivalences between fibrant objects.
\end{lem}

We make the following observations.
\begin{enumerate}
\item The Quillen pair $(f^*, f_*)$ gives rise to a derived adjoint pair
$(Lf^*, Rf_*)$.

\item Since the functor $f^*$ preserves all weak equivalences, there are
equivalences $Lf^*X \simeq f^*X$ for all discrete $G$-spectra $X$.

\item If $j: H \hookrightarrow G$ 
is the inclusion of a closed subgroup, then for a
discrete $H$-spectrum $X$, we have a \emph{non-equivariant} isomorphism
$$ j_*X = \Map_H^c(G, X) \cong \Map^c(G/H, X). $$
By Lemma~\ref{lem:Mapexact}, we see that
$j_*$ preserves weak equivalences, and therefore there is an
equivalence $j_* X \simeq Rj_*X$.

\item The adjoint pair $({\it triv}, (-)^G)$ of Section~\ref{sec:fixedpoint}
agrees with the adjoint pair $(r^*, r_*)$ when $r: G \rightarrow
\{e\}$ 
is the
homomorphism to the trivial group.  Therefore, the homotopy fixed
point functor is given by $(-)^{hG} = Rr_*$.

\item Given continuous homomorphisms $H \xrightarrow{f} G
\xrightarrow{g} K$, there are natural isomorphisms
$(g \circ f)_* \cong g_* \circ f_*$ and $(g
\circ f)^* \cong f^* \circ g^*$.  
We get similar formulas on the level of
derived functors.

\item If $i : U \hookrightarrow G$ is the inclusion of an open subgroup,
then the induction functor $i_! = \Ind_U^G$ (Proposition~\ref{prop:iterate})
is the left adjoint of the Quillen pair $(i_!, i^*)$.
\end{enumerate}

We use these derived functors to prove a 
version of Shapiro's Lemma.

\begin{lem}\label{lem:Shapiro}
Let $X$ be a discrete $G$-spectrum, and suppose that 
$H$ is a closed subgroup of $G$.  Then there is an equivalence
$$ \Map^c(G/H, X)^{hG} \xrightarrow{\simeq} X^{hH}. $$
\end{lem}

\begin{proof}
Consider the following diagram of groups.
$$
\xymatrix{
H \ar[rr]^{j} \ar[dr]_{s} && G \ar[dl]^{r} \\
& \{ e \}
}
$$
If $Z$ is a discrete $G$-set, there is a $G$-equivariant bijection
$$ \delta : j_* j^* Z = \Map^c_H(G, Z) \xrightarrow{\cong} 
\Map^c(G/H, Z). $$
The map $\delta$ sends a map $\alpha$ in $\Map^c_H(G,Z)$ to the map 
$$ \delta(\alpha): gH \mapsto g\alpha(g^{-1}). $$
The inverse $\delta^{-1}$ sends a map $\beta$ in $\Map^c(G/H, Z)$ to the
map
$$ \delta^{-1}(\beta): g \mapsto g \beta(g^{-1}H). $$
The isomorphism $\delta$ induces for a discrete $G$-spectrum $Y$ 
an isomorphism
$$ \delta: j_*j^* Y \xrightarrow{\cong} \Map^c(G/H, Y), $$
in $\Sigma\Sp_G$.
By Lemma~\ref{lem:homopair}, the functor $j_*$ sends $H$-fibrant 
objects to $G$-fibrant
objects.
Therefore we have equivalences:
\begin{align*}
\Map^c(G/H, X)^{hG} 
& \cong Rr_* j_* j^* X \\
& \simeq Rr_* Rj_* j^* X \\
& \simeq Rs_* j^* X \\
& = X^{hH}.
\end{align*}
\end{proof}

\subsection{Iterated fixed points for closed
subgroups}\label{sec:fcditerate}

We wish to extend the results of Section~\ref{sec:iterate} to closed
subgroups.  The following proposition may be compared to
\cite[Lem.~6.35]{Jardineetale}.

\begin{prop}\label{prop:closediterate}
Let $N$ be a closed normal subgroup of $G$, and let $X$ be a discrete
$G$-spectrum.  Then there is an equivalence
$$ ((X_{fG})^{N})^{hG/N} \simeq X^{hG}. $$
\end{prop}

\begin{proof}
Consider the following diagram.
$$
\xymatrix{
G \ar[rr]^q \ar[dr]_r && G/N \ar[dl]^{s} 
\\
& \{ e \}
}
$$
There is an equivalence
$$ ((X_{fG})^{N})^{hG/N} = Rs_* Rq_* X \simeq Rr_* X = X^{hG}. $$
\end{proof}

Let $H$ be a closed subgroup of $G$.
The reader might wonder if $X_{fG}$ is fibrant as a discrete $H$-spectrum,
but this does not appear to hold for nontrivial
closed subgroups $H$ that are not open.  
We discuss these difficulties in Section~\ref{sec:difficulties}.
Though $(X_{fG})^H$ is not known to always equal $X^{hH}$, the 
following 
result identifies $(X_{fG})^H$ with a canonical colimit that always maps 
to $X^{hH}$.

\begin{cor}\label{cor:mysterious}
Let $H$ be a closed subgroup of $G$.  There is an equivalence
$$ (X_{fG})^H \simeq \colim_{H \le U \le_o G} X^{hU}. $$
\end{cor}

\begin{proof}
Since $H$ acts discretely on $X_{fG}$, 
we have a canonical isomorphism
$$ (X_{fG})^H \cong \colim_{H \le U \le_o G} (X_{fG})^U. $$
By Proposition~\ref{prop:iterate}, the spectrum $X_{fG}$ is fibrant as a
discrete $U$-spectrum, so there are equivalences $(X_{fG})^U \simeq
X^{hU}$.
\end{proof}

By Corollary~\ref{cor:mysterious}, 
given a discrete $G$-spectrum $X$ and a closed subgroup $H$, 
there is a natural map 
\[(X_{fG})^H \simeq \colim_{H \leq U \leq_o G} X^{hU} \rightarrow X^{hH}.\]
If we restrict ourselves to the
case where $G$ has finite cohomological dimension, then, as shown
below, the iterated
homotopy fixed point spectrum behaves in a more satisfactory way.

\begin{prop}\label{prop:fcdHfp}
Suppose that $G$ has finite cohomological dimension, and suppose
that $X$ is a discrete $G$-spectrum.  Let $H$ be a closed subgroup of $G$.
Then the natural map
$$ \colim_{H \le U \le_o G} X^{hU} \rightarrow X^{hH} $$
is an equivalence.
\end{prop}

\begin{proof}
For $K$ a profinite group of finite vcd, 
and $Y$ a discrete $K$-spectrum, let $E_r(K; Y)$ denote the 
conditionally convergent descent
spectral sequence
$$ E_2(K;Y) = H^*_c(K; \pi_*(Y)) \Rightarrow \pi_*(Y^{hK}). $$
There is a map of spectral sequences
$$ E_r'(H;X) := \colim_{H \le U \le_o G} E_r(U;X) \rightarrow E_r(H;
X), $$
which is an isomorphism on the level of $E_2$-terms
by \cite[Thm.~9.7.2]{Wilson}.  The proposition now follows from
\cite[Prop.~3.3]{MitchellHyper}.
\end{proof}

In Section~\ref{sec:Galoisiterate}, we will see that we may extend
Proposition~\ref{prop:fcdHfp}
to groups of finite \emph{virtual} cohomological dimension provided that 
we are taking homotopy fixed points of 
a consistent profaithful $k$-local profinite Galois extension.

\begin{cor}\label{cor:fcdHfp}
Let $G$ be of finite cohomological dimension, and let $X$ be a discrete
$G$-spectrum.  Suppose that $H$ is a closed subgroup of $G$.  
Then there is an equivalence $ (X_{fG})^H \simeq X^{hH}. $
\end{cor}

\begin{thm}\label{thm:fcdHfp}
Let $G$ be of finite cohomological dimension, and let $X$ be a fibrant discrete
$G$-spectrum.  Suppose that $H$ is a closed subgroup of $G$.  Then $X$ is
fibrant as a discrete $H$-spectrum.
\end{thm}

\begin{proof}
By the proof of Theorem~\ref{thm:sigmaspmodel}, the fibrant objects $Y$ of
$\Sigma\Sp_H$ are precisely the objects for which $Y_i$ are
fibrant as simplicial discrete $H$-sets, and the maps
$$ Y_i \rightarrow \Omega Y_{i+1} $$ 
are weak equivalences.

Since $X$ is fibrant as a discrete $G$-spectrum, each $X_i$ is fibrant as a
simplicial discrete $G$-set, and the maps
$$ X_i \rightarrow \Omega X_{i+1} $$
are weak equivalences.  The only thing remaining to check is that each
space $X_i$ is fibrant when regarded as a simplicial discrete $H$-set.  

A criterion for fibrancy is established in \cite[10.2.7]{taf}.  
We remark that \cite{taf} deals with the more general class of locally 
compact totally disconnected groups $G$, 
acting on ``simplicial smooth $G$-sets.''  However,
in the case where $G$ is a profinite group, the category of smooth $G$-sets
is the category of discrete $G$-sets.
We
need to check
\begin{enumerate}
\item the $V$-fixed points $X_i^V$ is a Kan complex for every open subgroup
$V \le H$,
\item for every open subgroup $V \le H$ and every hypercover
$\{H/V_{\alpha, \bullet}\}_{\alpha \in I_\bullet}$ of $H/V$, the induced map
$$ X_i^V \rightarrow \holim_{\Delta} 
\prod_{\alpha \in I_\bullet} X_i^{V_{\alpha,\bullet}} $$
is a weak equivalence.
\end{enumerate}

Let $V$ be an open subgroup of $H$.
We have 
$$ X_i^{V} = \colim_{V \le U \le_o G} X_i^{U}. $$
Now $X_i^U$ is a Kan complex, since $X_i$ is fibrant as a simplicial 
discrete $G$-set.  Since filtered 
colimits of Kan complexes are Kan complexes, we
deduce that $X_i^V$ is a Kan complex.  This verifies condition $(1)$.  
We
furthermore point out that for each open subgroup $U$ in $G$ containing
$V$, Proposition~\ref{prop:iterate} implies that the fixed point spectrum
$X^U$
is fibrant.
Therefore, the maps
$$ X_i^U \rightarrow \Omega X_{i+1}^U $$
are weak equivalences.  Since
the functor $\Omega$ commutes with
filtered colimits, and because filtered colimits of weak equivalences are weak
equivalences, the maps
$$ X^V_i \rightarrow \Omega X^V_{i+1} $$ 
are weak equivalences, and we deduce
that the spectrum $X^V$ is fibrant (as a symmetric spectrum).
By Proposition~\ref{prop:iterate}, we have
$$ (X_{fH})^V \simeq X^{hV}. $$
Corollary~\ref{cor:fcdHfp} therefore implies that 
the map
$$ X^V \rightarrow (X_{fH})^V $$
is a weak equivalence.  Since weak equivalences between fibrant symmetric
spectra are levelwise equivalences, we deduce that each map
$$ X_i^{V} \rightarrow (X_{fH})^V_i $$
is a weak equivalence.  

We now verify condition~(2).
Suppose that $\{ H/V_{\alpha, \bullet} \}_{\alpha \in 
I_\bullet}$ is a hypercover of $H/V$.  Consider the following diagram.
$$
\xymatrix{
X_i^V \ar[r] \ar[d] &
\holim_{\Delta} \prod_{\alpha \in I_\bullet} X_i^{V_{\alpha,\bullet}}
\ar[d]
\\
(X_{fH})^V_i \ar[r] &
\holim_{\Delta} \prod_{\alpha \in I_\bullet}
(X_{fH})^{V_{\alpha,\bullet}}_i
}
$$
The bottom map is a weak equivalence since $X_{fH}$ is fibrant as a
discrete $H$-spectrum.  We have shown that the vertical maps are weak
equivalences.  We deduce that the top map is a weak equivalence.
\end{proof}

\begin{cor}\label{cor:fcditerate}
Let $X$ be a discrete $G$-spectrum.
Suppose that $G$ has finite cohomological dimension, and 
suppose that $N$ is a closed normal subgroup of $G$.  Then the homotopy
fixed point spectrum $X^{hN}$ is a discrete $G/N$-spectrum, and there 
is an equivalence 
$$ (X^{hN})^{hG/N} \simeq X^{hG}. $$
\end{cor}

\begin{proof}
Since $X_{fG}$ is fibrant as a discrete $N$-spectrum, we have
$$ X^{hN} = (X_{fG})^N $$
and $(X_{fG})^N$ is a discrete $G/N$-spectrum.  Furthermore, by
Proposition~\ref{prop:closediterate},
we have
$$ X^{hG} \simeq ((X_{fG})^N)^{hG/N} \simeq (X^{hN})^{hG/N}. $$ 
\end{proof}

\subsection{The difficulties concerning arbitrary closed fixed
points}\label{sec:difficulties}

Let $H$ be a closed subgroup of an arbitrary profinite group $G$.  
We would be able to remove
the finite cohomological dimension hypothesis in
Section~\ref{sec:fcditerate}, if we knew that the restriction functor
$$ \Res_G^H : \Sigma\Sp_G \rightarrow \Sigma\Sp_H $$
sends $G$-fibrant objects to $H$-fibrant objects.  While we know of no
counterexamples to this assertion, we also doubt that it 
this is true in general.

We saw in Proposition~\ref{prop:iterate} that for $U$ an open subgroup of
$G$, the presence of an induction functor $\Ind_U^G$ which was a 
left Quillen 
adjoint to
$\Res_G^U$ allowed us to prove that $\Res_G^U$ preserves fibrant objects.

However, as pointed out to the second author 
by Jeff Smith, $\Res_G^H$ cannot possess a left adjoint, in general, 
since it 
does not preserve limits. This can be seen as follows. 

For a profinite 
group $K$ and a diagram $\{X_\alpha\}$ in the category $\Sigma\Sp_K$, 
let
$$ {\lim_\alpha}^K X_\alpha $$
denote the limit computed in the category $\Sigma\Sp_K$.
This limit is given by the following formula:
$$
{\lim_\alpha}^K X_\alpha = 
\colim_{U \trianglelefteq_o K} ({\lim_\alpha}^{\Sp} X_\alpha)^U. $$ 
Here, the limit $\lim^{\Sp}$ is the limit computed in the underlying 
category of symmetric spectra.

Thus, given a diagram $\{ X_\alpha\}$ in $\Sigma\Sp_G$, 
the restriction of the limit is given by
$$
\Res^H_G {\lim_\alpha}^G X_\alpha = 
\colim_{U \trianglelefteq_o G} ({\lim_\alpha}^{\Sp} X_\alpha)^U. 
$$ 
However, the limit of the restriction is computed to be
$$ {\lim_\alpha}^H \Res_G^H X_\alpha = 
\colim_{V \trianglelefteq_o H} 
({\lim_\alpha}^\Sp X_\alpha)^{V}.$$ 
When $H$ is not open in $G$, the lack of cofinality implies that
these subspectra of $\lim_\alpha^\Sp X_\alpha$ in general do not agree.  

One might suspect that one could still prove that the map
$$ \colim_{H \le U \le_o G} X^{hU} \rightarrow X^{hH} $$
is an equivalence if $G$ has
finite \emph{virtual} cohomological dimension, by a comparison 
of descent spectral sequences.  This approach, however, also presents
difficulties.
As in the proof of Proposition~\ref{prop:fcdHfp},
there is a map of spectral sequences
$$ E_r'(H;X) = \colim_{H \le U \le_o G} E_r(U;X) \rightarrow E_r(H; X) $$
which is an isomorphism on the level of $E_2$-terms
(see \cite[Thm.~9.7.2]{Wilson}).  The
problem is that the colimit of the spectral sequences does not converge to
the colimit of the abutments in general.  

\section{Continuous $G$-spectra}\label{sec:cont}

In this paper, a {\em continuous $G$-spectrum} is a pro-object in the category of
discrete $G$-spectra.  
In this section, we extend some of our constructions 
for $\Sigma \mathrm{Sp}_G$ to the category 
of continuous $G$-spectra.  For continuous 
$G$-spectra that are indexed over $\{0 \leftarrow 1 \leftarrow 2 \leftarrow 
\cdots \}$, part of this material appears in more detail in \cite{Davis}.

\subsection{Pro-objects in discrete $G$-spectra}

Following standard usage, a pro-object in a category $\mathcal{C}$ 
is a cofiltered diagram in $\mathcal{C}$.  We define the category of 
continuous $G$-spectra $\Sigma \mathrm{Sp}^c_G$ to be the category of 
pro-objects in $\Sigma \mathrm{Sp}_G$.  Thus, a continuous $G$-spectrum 
is a cofiltered diagram $\mathbf{X} = \{X_i\}_{i\in I}$ of discrete 
$G$-spectra.  Maps in the category of continuous $G$-spectra are given 
by $$\Sigma \mathrm{Sp}^c_G(\mathbf{X},\mathbf{Y}) = 
\lim_j \colim_i \Sigma \mathrm{Sp}_G(X_i,Y_j).$$  
\par
Any pro-spectrum 
$\mathbf{X} = \{X_i\}$ gives rise to a spectrum $X$ via the homotopy 
limit functor: $$X = \holim_i X_i.$$  We shall always denote our pro-spectra 
by boldface type and their homotopy limits with non-boldface type.  

\begin{rmk}
A more general theory of pro-spectra,
including a model category structure, has
been developed by Isaksen (see \cite{Isaksen} and
\cite[Section~1.1]{tmodel}).  Fausk has developed
a category of continuous \emph{genuine} $G$-spectra where $G$ is a 
compact Hausdorff topological group \cite{Fausk}.
The notion of continuous $G$-spectrum in this paper 
(that is, a pro-discrete $G$-spectrum) corresponds roughly, in \cite{Fausk}, 
to a pro-$G$-spectrum that is 
in the
full subcategory of cofibrant objects in the
Postnikov $\br{\mathrm{Lie}(G)}$-model structure on
pro-$\mathcal{M}_S$. For more detail, we refer
the reader to \cite[Section 11.3]{Fausk},
especially the discussion centered around
[op. cit., Eq. (11-15)]. 
\end{rmk}

\subsection{Continuous mapping spectra}

Let $K = \lim_i K_i$ be a 
profinite $G$-set.  Given a continuous $G$-spectrum $\mathbf{X}$, the 
continuous mapping spectrum $\mathbf{Map}^c(K,\mathbf{X})$ is defined to 
be the continuous $G$-spectrum $$\{\Map^c(K,X_j)\}_j.$$  We denote the 
homotopy limit of $\mathbf{Map}^c(K,\mathbf{X})$ by $\Map^c(K,\mathbf{X})$.
If $K$ satisfies the hypotheses of Lemma~\ref{lem:Mapexact}, and 
the derived functors 
$\lim^s_j \Map^c(K, \pi_t(X_j)) = 0,$ for all $s > 0$ and all $t \in 
\mathbb{Z}$, then the Bousfield-Kan spectral sequence $$\textstyle
\lim^s_j 
\Map^c(K, \pi_t(X_j)) \Rightarrow \pi_\ast(\Map^c(K,\mathbf{X}))$$ collapses, 
and thus, $$\pi_\ast(\Map^c(K,\mathbf{X})) \cong \Map^c(K, \lim_j 
\pi_\ast(X_j)). $$

\subsection{Continuous permutation spectra}

Let $K = \lim_i K_i$ be a 
profinite $G$-set, 
and let each finite set $K_j$, for each $j$ in the indexing set for $K$, be a discrete $G$-set.  
Also, let 
$\mathbf{X} = \{ X_i \}_i$ be a continuous $G$-spectrum.  Define the permutation 
spectrum $\mathbf{X}[[K]]$ to be the continuous $G$-spectrum given 
by 
$$\{X_i[K_j]\}_{i,j}.$$  
We denote the homotopy 
limit of $\mathbf{X}[[K]]$ by 
$$ X[[K]] = \holim_{i,j} X_i[K_j]. $$  
\par
Note that if $\lim^s_{i,j} \pi_t(X_i)[K_j] = 0$, for all $s > 0$ 
and all $t \in \mathbb{Z}$ (where $\pi_t(X_i)[K_j]$ is an abelian 
group), then $$\pi_\ast(X[[K]]) \cong \lim_{i,j} \pi_\ast(X_i)[K_j].$$
If $E$ is a 
discrete $G$-spectrum, we use $\mathbf{E}[[K]]$ to denote 
the continuous $G$-spectrum $\{E[K_j]\}_j$.  

\subsection{Continuous homotopy fixed points}

For a continuous $G$-spectrum $\mathbf{X}$, we define the homotopy fixed
point pro-spectrum 
$\mathbf{X}^{hG}$ 
to be
$$ \{ X_i^{hG} \}_i. $$
We denote the homotopy limit of $\mathbf{X}^{hG}$ by
$X^{hG}$, and we refer to $X^{hG}$ as
the \emph{homotopy fixed point spectrum}.

\subsection{Continuous hypercohomology spectra}

We define 
$$ \mathbf{\Gamma}_G : \mathrm{(pro}\negthinspace-\negthinspace \Sigma\Sp\mathrm{)}
\rightarrow \mathrm{(pro}\negthinspace-\negthinspace
\Sigma\Sp\mathrm{)} $$
to be the coaugmented comonad given
by
$$ \mathbf{\Gamma}_G(\mathbf{X}) = \MAP^c(G,\mathbf{X}). $$
Let $\Gamma_G(\mathbf{X})$ be the homotopy limit of 
$\mathbf{\Gamma}_G(\mathbf{X})$.

If $\mathbf{X}$ is a continuous $G$-spectrum, then it is a coalgebra over 
$\mathbf{\Gamma}_G$.  
Let $\HH_c(G; \mathbf{X})$ denote the pro-spectrum obtained by taking
hypercohomology levelwise:
$$ 
\HH_c(G; \mathbf{X}) = \{ \HH_c(G;X_i) \}_i.
$$
Let $\HH_c(G; X)$ denote the homotopy limit of the pro-spectrum $\HH_c(G;
\mathbf{X})$.
The following result follows immediately from
Theorem~\ref{thm:cochaincomplex}.

\begin{thm}\label{Davis}
Suppose that $\mathbf{X} = \{X_i \}$ is a continuous $G$-spectrum.
If $G$ has finite vcd,
then there is an equivalence
$$X^{hG} \simeq \HH_c(G; X). $$
\end{thm}

\subsection{Homotopy fixed point spectral sequence}

Let $G$ have finite vcd.   
Then Theorem~\ref{Davis} implies that
$$ X^{hG} \simeq \holim_\Delta \holim_i \Gamma_G^\bullet X_i, $$  
and, hence, the associated
Bousfield-Kan spectral sequence 
has the form 
$$ E_2^{s,t}(G;X) = \pi^s \pi_t(\holim_i \Gamma_G^\bullet X_i) 
\Rightarrow \pi_{t-s}(X^{hG}), $$  
giving the conditionally convergent
homotopy fixed point spectral sequence for $X^{hG}$.

Observe that there is a natural isomorphism
$$ \Gamma_G^k(-) \cong
\Map^c(G^k,-). $$
If 
$\lim^s_i \Map^c(G^k, \pi_t(X_i)) = 0$ for all $s > 0$, all $k \geq 0$, 
and all $t \in \mathbb{Z}$, then for each $k \ge 0$,
the Bousfield-Kan spectral sequence
$$\textstyle\lim^s_i \Map^c(G^k, \pi_t(X_i)) \Rightarrow \pi_\ast
(\holim_i \Map^c(G^k, X_i))$$
collapses, and thus, 
\begin{align*}
E_2^{s,t}(G;X) & \cong \pi^s(\lim_i
\Map^c(G^\bullet, \pi_t(X_i))) \\  
& \cong H^s(\Map^c(G^\bullet, \lim_i \pi_t(X_i))) \\
& \cong H^s(\Map^c(G^\bullet, \pi_t(X))) \\
& \cong H^s_c(G; \pi_t(X)).
\end{align*}
Here, $H^s_c(G; \pi_t(X))$
denotes the continuous cohomology of continuous cochains, with 
coefficients in the topological $G$-module 
$\pi_t(X) \cong \lim_i \pi_t(X_i).$

\subsection{Completed smash product}

If $\mathbf{X}$ and $\mathbf{Y}$ are continuous $G$-spectra,
we define the completed smash product $\mathbf{X} \wedge_c \mathbf{Y}$ 
to be the continuous 
$G$-spectrum 
$$ \{ X_i \wedge Y_j \}_{i,j}. $$
The completed
smash product gives $\Sigma\Sp^c_G$ a
symmetric monoidal product, where the unit is
$\{S^0\}$ (the sphere spectrum regarded as a
diagram indexed by a single element).

\section{Modules and commutative algebras of discrete
$G$-spectra}\label{sec:alg}

Let $A$ be a commutative symmetric ring spectrum and let $G$ be a
profinite group.
In this section, we describe the model categories of discrete
$G$-$A$-modules and discrete commutative $G$-$A$-algebras.  We show that 
the 
homotopy fixed points of a discrete $G$-$A$-module is an $A$-module and
the homotopy fixed points of a discrete commutative $G$-$A$-algebra is a
commutative $A$-algebra.
These structured homotopy
fixed point constructions are shown to agree, in the stable homotopy
category, with the usual homotopy fixed points of the underlying discrete
$G$-spectrum.  
We then make
comparisons between filtered homotopy colimits and filtered 
colimits of modules and commutative algebras, and conclude that, when
properly interpreted, they all coincide in the stable homotopy category.
We conclude this section by describing how to make the hypercohomology spectra
of discrete commutative $G$-$A$-algebras 
take values in the category of commutative $A$-algebras.

\subsection{Modules of discrete $G$-spectra}

Let $A$ be a commutative symmetric ring spectrum.  By a discrete 
$G$-$A$-module, we shall mean a discrete $G$-spectrum $X$ that also
possesses the structure of an $A$-module.  We require these structures to be
compatible in the following sense: for every element $g \in G$, the
following diagram must commute.
$$
\xymatrix{
A \wedge X \ar[r]^{\xi} \ar[d]_{A \wedge g} 
& X \ar[d]^{g} 
\\
A \wedge X \ar[r]^{\xi} 
& X
}
$$
Here, $\xi$ is the $A$-module structure map.  Let $\Mod_{G,A}$ denote the
category of discrete $G$-$A$-modules, 
with morphisms being the $G$-equivariant maps
that are also maps of $A$-modules.
Note that, given discrete $G$-$A$-modules $X$ and $Y$, 
their smash product $X \wedge_A
Y$ is easily seen to be a discrete $G$-$A$-module with the diagonal
action.

The following simplified 
variant of D.M.~Kan's
``lifting theorem'' will 
be used repeatedly 
to provide the desired model structures on structured 
categories like $\Mod_{G,A}$.

\begin{lem}\label{lem:pullback}
Suppose that $\mathcal{M}$ is a cofibrantly generated model category with
generating cofibrations $I$ and generating trivial cofibrations $J$.
Furthermore, assume that the domains of $I$ and $J$ are $\alpha$-small for
some cardinal $\alpha$.  Suppose
that we are given a complete and cocomplete category $\mathcal{N}$ and an 
adjoint pair $(F,G)$
$$ F: \mathcal{M} \rightleftarrows \mathcal{N} : G, $$
where:
\begin{enumerate}
\item $G$ commutes with filtered colimits; and
\item $G$ takes relative $FJ$-cell complexes to weak
equivalences.
\end{enumerate}
Then $\mathcal{N}$ admits an induced model category structure where the
fibrations and weak equivalences are those morphisms which get sent to
fibrations and weak equivalences by $G$, and the cofibrations are
determined.  This model category structure is cofibrantly generated with
generating cofibrations $FI$ and generating trivial cofibrations $FJ$.  The
domains of $FI$ and $FJ$ are $\alpha$-small in $\mathcal{N}$.
\end{lem}

\begin{proof}
This lemma is a special case of Theorem~11.3.2 of \cite{Hirschhorn}.  To
see that the hypotheses of this theorem are met in our situation, we must
verify that
the domains of $FI$ and $FJ$ are $\alpha$-small with
respect to relative $FI$- and $FJ$-cell complexes,
respectively.
However, our hypotheses imply that $F$ preserves all $\alpha$-small objects:
given an $\alpha$-small object $X \in \mathcal{M}$, and a $\lambda$-sequence
($\lambda \ge \alpha$)
$$ Y_1 \rightarrow Y_2 \rightarrow Y_3 \rightarrow \cdots \rightarrow
Y_\beta \rightarrow \cdots \quad (\beta < \lambda) $$
in $\mathcal{N}$, we have:
\begin{align*}
\colim_i \Hom_{\mathcal{N}}(FX, Y_i) & \cong 
\colim_i \Hom_{\mathcal{M}}(X, GY_i) \\
& \cong \Hom_{\mathcal{M}}(X, \colim_i GY_i) \\
& \cong \Hom_{\mathcal{M}}(X, G \colim_i Y_i) \\
& \cong \Hom_{\mathcal{N}}(FX, \colim_i Y_i).
\end{align*}
\end{proof}

In \cite{MMSS}, a model category structure is defined on the category of
$A$-modules.  The fibrations and weak equivalences of this model structure
are the fibrations and weak equivalences of underlying symmetric spectra,
and the cofibrations are determined.

\begin{prop}\label{prop:modulemodel}
The category $\Mod_{G,A}$ is a model category, where the
fibrations and weak equivalences are the fibrations and weak equivalences 
of the underlying discrete $G$-spectra, and the cofibrations are the
cofibrations of the underlying $A$-modules.
\end{prop}

\begin{proof}
We apply Lemma~\ref{lem:pullback} to the adjoint pair
$$ A \wedge - : \Sigma\Sp_G \rightleftarrows \Mod_{G,A} : \mathcal{U}, $$
where $\mathcal{U}$ is the forgetful functor.  

The category $\Sigma\Sp_G$ is cofibrantly generated with generating
cofibrations
$$ I = \{ F_n (G/U \times \partial \Delta^n)_+ \hookrightarrow 
F_n (G/U \times \Delta^n)_+ \: : \: U \le_o G, \: n \ge 0 \}.
$$
(Here, $F_n$ is the left adjoint to the functor which returns the $n$th
space of a symmetric spectrum.)  Thus, the domains of the maps in $I$ are 
$\omega$-small.  Since the model structure on $\Sigma\Sp_G$ is 
left proper and cellular, the Bousfield-Smith cardinality argument
\cite[Prop.~4.5.1]{Hirschhorn} implies that 
there exists a cardinal $\alpha$, such that the collection $J$ of inclusions
of $I$-cell complexes of size at most $\alpha$, which are weak
equivalences,
generates the trivial cofibrations of $\Sigma\Sp_G$.  In particular, since
the $I$-cells are finite, the domains of $J$ are $\alpha$-small.

By
Lemma~\ref{lem:algebracolimits} below, the functor $\mathcal{U}$ preserves
filtered colimits.  We need to verify that $\mathcal{U}$ takes
relative $(A \wedge J)$-cell complexes to weak equivalences.  However, the
domains and codomains of the maps in $J$ are cofibrant objects in
$\Sigma\Sp$, and $A \wedge -$
preserves stable equivalences between cofibrant symmetric spectra.

The model structure on $\Mod_{G,A}$ given by Lemma~\ref{lem:pullback} has
the desired fibrations and weak equivalences.  We need to check that the
cofibrations of this model structure are precisely the cofibrations of
underlying non-equivariant $A$-modules.
The cofibrations are
generated by $A \wedge I$.  Since all of the morphisms in $I$ are
cofibrations of underlying non-equivariant symmetric spectra, we deduce
that all of the cofibrations of $G$-$A$-modules are cofibrations of
underlying $A$-modules.  The converse is an argument similar to
Corollary~1.10 of \cite{Goerss}.
\end{proof}

Given a discrete $G$-$A$-module $X$, let $X_{fGA}$ denote a functorial fibrant
replacement in $\Mod_{G,A}$, so that the fibrant replacement map
$$ X \rightarrow X_{fGA} $$
is a trivial cofibration of discrete $G$-$A$-modules.
We define the homotopy fixed point $A$-module to be the fixed
point spectrum
$$ X^{h_AG} = (X_{fGA})^G. $$
The properties of our model category structures immediately give the following
lemma.

\begin{lem}
The spectrum $X_{fGA}$ is fibrant as a discrete $G$-spectrum, and there
exists a weak equivalence
$$ X_{fG} \xrightarrow{\simeq} X_{fGA} $$
in the category $\Sigma\Sp_G$.
\end{lem}

Since $(-)^G$ preserves weak equivalences between fibrant spectra
in $\Sigma\Sp_G$,
we have the following corollary.

\begin{cor}\label{cor:5.1.4}
There is an equivalence
$$ X^{hG} \xrightarrow{\simeq} X^{h_AG}. $$ 
\end{cor}

Since colimits of discrete $G$-$A$-modules
are formed on the level of underlying
symmetric spectra, the comonad
$\Gamma_G(-) = \Map^c(G, -)$ restricts to the category of 
$A$-modules, where, given an $A$-module $X$, the $A$-module structure
is given by the composition
$$
A \wedge \Map^c(G,X) \rightarrow  \Map^c(G, A \wedge X) \xrightarrow{\xi_*}
\Map^c(G, X).
$$
In this composition, the first
map is given by the composite
\begin{align*} 
A \wedge \Map^c(G,X) & = A \wedge 
\colim_{N \trianglelefteq_o G} \Map(G/N, X) \\
& \xrightarrow{\cong} 
\colim_{N \trianglelefteq_o G} A \wedge \Map(G/N, X) \\
& \rightarrow \colim_{N \trianglelefteq_o G} 
\Map(G/N, A \wedge X) \\
& = \Map^c(G,A \wedge X). 
\end{align*}

Since fibrant discrete $G$-$A$-modules are fibrant in $\Sigma\Sp_G$, 
the following proposition is an immediate consequence of 
Theorem~\ref{thm:cochaincomplex}.

\begin{prop}\label{prop:Amodcochaincomplex}
Suppose $G$ has finite vcd, and let $X$ be a discrete
$G$-$A$-module.
Then there is an equivalence of $A$-modules
$$ X^{h_AG} \simeq \holim_\Delta \Gamma_G^\bullet X = \HH_c(G; X). $$
\end{prop}

Define the category $\Mod^c_{G,A}$ of \emph{continuous $G$-$A$-modules} 
to be the category
of pro-objects in $\Mod_{G,A}$.  Given a continuous $G$-$A$-module
$\mathbf{X} = \{ X_i \}_i$, we define
\begin{align*}
X^{h_AG} & := \holim_i X^{h_A G}_i, \\
\HH_c(G; X) & := \holim_i \HH_c(G; X_i).
\end{align*}
Proposition~\ref{prop:Amodcochaincomplex} has the following
corollary.

\begin{cor}
Suppose $G$ has finite vcd, and let $\mathbf{X}$ be a continuous
$G$-$A$-module.
Then there is an equivalence of $A$-modules
$$ X^{h_AG} \simeq \HH_c(G; X). $$
\end{cor}

We shall henceforth drop the distinction between $(-)^{h_AG}$ and
$(-)^{hG}$: all homotopy
fixed points of discrete $G$-$A$-modules will 
implicitly be taken in the category of
discrete $G$-$A$-modules.

\subsection{Commutative algebras of discrete $G$-spectra}

By a \emph{discrete commutative $G$-$A$-algebra}, we shall mean a 
discrete $G$-$A$-module
$E$ together with a commutative $A$-algebra multiplication
$$ \mu : E \wedge_A E \rightarrow E, $$
such that $G$ acts on $E$ through maps of commutative 
$A$-algebras.  Let $\Alg_{A,G}$ denote
the category of discrete commutative $G$-$A$-algebras, with morphisms
being those morphisms in $\mathrm{Mod}_{G,A}$
that are also maps of commutative $A$-algebras.
Following
\cite{MMSS}, to place a
model structure on $\Alg_{A,G}$, we need to replace the model structure on
$\Sigma\Sp_G$ with a Quillen equivalent ``positive'' model structure.

\begin{lem}
The category of discrete $G$-spectra admits a positive stable 
model structure,
which we denote by $\Sigma\Sp_G^+$,
where the cofibrations are the positive cofibrations of underlying
symmetric spectra, the weak equivalences are the stable equivalences of
underlying symmetric spectra, and the positive fibrations are determined.
\end{lem}

\begin{proof}
We follow \cite[Sec.~14]{MMSS},
in its construction of the positive
stable model structure on symmetric spectra,
from start to finish, with some mild alterations.
The category $\Sigma\Sp_G$ admits a positive level model structure
that is defined as follows:
\begin{itemize}
\item the positive level fibrations are those maps $f: X \rightarrow Y$ where
the maps $f_i: X_i \rightarrow Y_i$ are fibrations of simplicial discrete
$G$-sets for $i \ge 1$;

\item the positive level weak equivalences are those maps $f: X \rightarrow
Y$ where the maps $f_i : X_i \rightarrow Y_i$ are weak equivalences of
underlying simplicial sets, for all $i \ge 1$; and

\item the positive cofibrations are those morphisms $X \rightarrow Y$ where the
induced map
$$ X_0 \rightarrow Y_0 $$
is an isomorphism and each of the induced maps
\begin{align*}
X_i & \rightarrow Y_i, \qquad i = 1,\\
X_i \cup_{L_i X} L_i Y & \rightarrow Y_i, \qquad i
\ge 2,
\end{align*}
is a projective cofibration in $(s\Set_G)^{\Sigma_i}_*$, where $L_i$ is the
latching object of \cite[Def.~8.4]{Hoveyspec} (that is to say, they are
precisely the positive cofibrations of underlying symmetric
spectra).
\end{itemize}
The positive level model structure on $\Sigma\Sp_G$ is left proper and
cellular, and hence, it 
admits a localization with respect to the set of maps
$$ F_n S^1 \wedge C \rightarrow F_{n+1} C, \quad n \ge 1, $$
where, as usual, $F_n$ is the left adjoint to the $n$th space evaluation
functor and $C$ runs through the cofibrant domains and codomains of the
generating cofibrations of $s\Set_G$.  This localized model structure is
the \emph{positive stable model structure}:
\begin{itemize}
\item the cofibrations are those morphisms which are positive cofibrations
on the underlying symmetric spectra;
\item the fibrant objects are the discrete $G$-spectra $X$ for which
\begin{enumerate}
\item the spaces $X_i$ are fibrant as simplicial discrete $G$-sets for $i
\ge 1$;
\item the spectrum structure maps $X_i \rightarrow \Omega X_{i+1}$ are weak
equivalences for $i \ge 1$,
\end{enumerate}
\item the weak equivalences between fibrant objects are those morphisms
that are positive level weak equivalences of discrete $G$-spectra.
\end{itemize}
We are left with verifying that the weak equivalences of the positive
stable model structure are precisely the stable equivalences of 
underlying symmetric
spectra.  It suffices to check this for morphisms between positive stable
fibrant objects.  
Let $\phi: X \rightarrow Y$ be a morphism between positive stable fibrant
objects in $\Sigma\Sp_G$, 
and consider the
functorial fibrant replacement (in the non-positive stable model structure) 
of $\phi$ in $\Sigma\Sp_G$.
$$ 
\xymatrix{
X \ar[r]^{\alpha_{G,X}} \ar[d]_\phi & 
X_{fG} \ar[d]^{\phi_{fG}}
\\
Y \ar[r]_{\alpha_{G,Y}} \ar[r] &
Y_{fG} 
}
$$
Note that the stably fibrant objects of $\Sigma\Sp_G$ are positive stable
fibrant.
The two morphisms $\alpha_{G,-}$ are therefore 
positive stable equivalences between positive stable fibrant discrete 
$G$-spectra,
and therefore are positive level weak equivalences.  It follows that
they induce isomorphisms on na\"ive homotopy groups, and hence are stable
equivalences \cite[Thm.~3.1.11]{HSS}.  The same argument shows that if
$\phi$ is a positive stable equivalence, then $\phi$ is a stable
equivalence.  Conversely, suppose that $\phi$ is a stable equivalence.
Then, by the two out of three axiom, the morphism $\phi_{fG}$ is a stable
equivalence.  Since $\phi_{fG}$ is a stable equivalence between stable
fibrant discrete $G$-spectra, we deduce that it is a levelwise
equivalence.  In particular, it is a positive level weak equivalence.
This allows us to
deduce that $\phi$ is a positive level weak equivalence, and
thus, it is a positive stable equivalence. 
\end{proof}

\begin{lem}
The identity functor from discrete $G$-spectra with the positive
stable model structure to discrete $G$-spectra with the stable model structure is
the left adjoint of a Quillen equivalence.
\end{lem}

\begin{proof}
This follows from the fact that the weak equivalences in both model
structures are the same, and every positive cofibration is a cofibration.
\end{proof}

The category $\Alg_S$ of commutative symmetric ring spectra 
admits a model category
structure where the weak equivalences and fibrations are detected on the
level of underlying symmetric spectra (with the positive model structure)
\cite{MMSS}.  The category $\Alg_A$ may be regarded as the category of
commutative symmetric ring spectra under $A$, and hence inherits a model
structure with the same cofibrations, fibrations, and weak equivalences.

\begin{thm}
The category of discrete commutative $G$-$A$-algebras admits a positive model
structure where the cofibrations are the cofibrations of underlying 
commutative
symmetric ring spectra, the weak equivalences are the stable equivalences
of underlying symmetric spectra, and the fibrations are the positive
fibrations of discrete $G$-spectra.
\end{thm}

\begin{proof}
The proof follows \cite[Thm~15.2(i)]{MMSS}.
The category of discrete commutative $G$-$A$-algebras is 
the
undercategory
$(A \downarrow \mathrm{Alg}_G)$, where
$A$ is regarded as a discrete $G$-spectrum
with the trivial action and
$\mathrm{Alg}_G$ denotes the category
$\mathrm{Alg}_{S,G}$ of discrete commutative
$G$-algebras. Therefore, $\mathrm{Alg}_{A,G}$
inherits a model structure from $\mathrm{Alg}_G$,
by \cite[Thm~7.6.5(1)]{Hirschhorn}. Hence,
it suffices to prove the theorem when $A=S$.
We apply Lemma~\ref{lem:pullback} to the adjoint pair
$$ \mathbb{P} : \Sigma\Sp_G^+ \rightleftarrows \Alg_G : \mathbb{U}, $$
where $\mathbb{U}$ is the forgetful functor, and $\mathbb{P}$ is the free
commutative algebra functor:
$$ \mathbb{P}(X) = \bigvee_{i \ge 0} (X^{\wedge i})_{\Sigma_i}. $$

The category $\Sigma\Sp_G^+$ is cofibrantly generated with generating
cofibrations
$$ I^+ = \{ F_n (G/U \times \partial \Delta^n)_+ \hookrightarrow 
F_n (G/U \times \Delta^n)_+ \: : \: U \le_o G, \: n \ge 1 \}.
$$
The domains of the maps in $I^+$ are thus
$\omega$-small.  The same Bousfield-Smith cardinality argument used in the
proof of Proposition~\ref{prop:modulemodel} 
shows that there exists a cardinal $\alpha$ and a set $J^+$ of generating
trivial cofibrations such that 
the domains of $J^+$ are $\alpha$-small.

By Lemma~\ref{lem:algebracolimits} below, the functor $\mathbb{U}$ preserves
filtered colimits.  We just need to verify that for any stable equivalence
$$ \phi : X \rightarrow Y, $$
the induced morphism
$$ \mathbb{P}(\phi): \mathbb{P}(X) \rightarrow \mathbb{P}(Y) $$ is a stable
equivalence.  It suffices to verify that the map on coinvariants
$$ \phi_{i}: (X^{\wedge i})_{\Sigma_i} \rightarrow (Y^{\wedge
i})_{\Sigma_i} $$
is a stable equivalence for every $i$.  Consider the following diagram
of canonical maps.
$$
\xymatrix{
(X^{\wedge i})_{h\Sigma_i} \ar[r]^{q_{X,i}} \ar[d]_{(\phi_{i})_h} &
(X^{\wedge i})_{\Sigma_i} \ar[d]^{\phi_i}
\\
(Y^{\wedge i})_{h\Sigma_i} \ar[r]_{q_{Y,i}} &
(Y^{\wedge i})_{\Sigma_i}
}
$$
Since homotopy colimits preserve weak equivalences, we deduce that the maps
$(\phi_i)_h$ are stable equivalences.  The morphisms $q_{-,i}$ 
induce isomorphisms of na\"ive homotopy groups: this is seen by applying the
geometric realization functor levelwise and by using Lemma~15.5 of
\cite{MMSS} (the geometric realization of a positive cofibrant symmetric
spectrum of simplicial sets is easily seen to be a positive cofibrant
symmetric spectrum of topological spaces).  Therefore, by
\cite[Thm~3.1.11]{HSS}, the morphisms $q_{-,i}$ are stable equivalences.
We deduce that each $\phi_i$ is a stable equivalence, as desired.
\end{proof}

We have the following corollary.

\begin{cor}\label{cor:5.2.4}
The following pairs of adjoint functors are Quillen adjoints:
\begin{gather*}
{\it triv} : \Sigma\Sp^+ \rightleftarrows \Sigma\Sp^+_G : (-)^G, \\
{\it triv} : \Alg_{A} \rightleftarrows \Alg_{A,G} : (-)^G.
\end{gather*}
\end{cor}

For a discrete $G$-spectrum $X$, we use $X_{fG^+}$ to denote the functorial
fibrant replacement in the positive model structure, and we denote the
corresponding homotopy fixed point spectrum by
$$ X^{h^+G} = (X_{f^+G})^G. $$
We have the following result.

\begin{lem}\label{lem:5.2.5}
If $X$ is a discrete $G$-spectrum,
then there is an
equivalence
$$ X^{h^+G} \xrightarrow{\simeq} X^{hG}. $$
\end{lem}

\begin{proof}
Since fibrant discrete $G$-spectra are positive fibrant discrete
$G$-spectra, there is a stable equivalence
$$ \alpha: X_{f^+G} \rightarrow X_{fG} $$
in the category $\Sigma\Sp_G$.
Since $(-)^G$ preserves stable equivalences between positive fibrant
discrete $G$-spectra, $\alpha$ induces an equivalence
$$ X^{h^+G} = (X_{f^+G})^G \xrightarrow{\alpha_*} (X_{fG})^G = X^{hG}. $$
\end{proof}

Given a discrete commutative $G$-$A$-algebra $E$, we shall denote the 
functorial
fibrant replacement by 
$E_{fGA-\Alg}$.  We define the homotopy fixed point
commutative $A$-algebra to be the fixed point spectrum
$$ E^{h_\Alg G} = (E_{fGA-\Alg})^G. $$
Since
$E_{fGA-\mathrm{Alg}}$ is a positive fibrant
discrete $G$-spectrum, a slight modification
of the proof of Lemma~\ref{lem:5.2.5} implies the following lemma.

\begin{lem}\label{lem:5.2.6}
If $E$ is a discrete commutative $G$-$A$-algebra, then
there is an equivalence
$$ E^{h^+G} \xrightarrow{\simeq} E^{h_\Alg G}. $$
\end{lem}

If $E$ is a discrete commutative $G$-$A$-algebra,
then Corollary~\ref{cor:5.1.4} and Lemmas~\ref{lem:5.2.5} and
\ref{lem:5.2.6}
imply that there is the following zig-zag of
equivalences:
\[E^{h_A G} \overset{\simeq}{\leftarrow} E^{hG}
\overset{\simeq}{\leftarrow} E^{h^+G}
\overset{\simeq}{\rightarrow} E^{h_\mathrm{Alg} G}.\]
Therefore, we shall henceforth not distinguish between the four 
equivalent types of
homotopy fixed points that appear in the above zig-zag.  
Also, homotopy fixed points of discrete
commutative $G$-$A$-algebras will always implicitly be taken in the category of
discrete commutative $G$-$A$-algebras.

\subsection{Filtered colimits}

We will make frequent use of filtered colimits.  
In this section, we show that
filtered colimits of spectra are rather
well-behaved. 
We begin with the
following lemma, whose
proof follows the proofs of \cite[Lemma~5.5]{Thomason}
and \cite[Proposition~3.2]{MitchellHyper} (written in
the context of Bousfield-Friedlander spectra). 
We remind
the reader that a fibrant
spectrum is positive fibrant.

\begin{lem}
In the category of symmetric spectra, filtered 
colimits preserve cofibrations,
fibrations, and positive fibrations.  
Filtered colimits preserve weak equivalences between 
fibrant spectra, and weak equivalences between positive fibrant spectra.
\end{lem}

\begin{cor}\label{cor:hocolim}
Given a filtered diagram $\{ X_\alpha \}_{\alpha \in I}$ of 
(positive) fibrant spectra, there is a stable equivalence
$$ \hocolim_\alpha X_\alpha \xrightarrow{\simeq} \colim_\alpha X_\alpha. $$
\end{cor}

\begin{proof}
Let $\{\td{X}_\alpha\} \xrightarrow{\phi} \{X_\alpha\}$ be a cofibrant
replacement in the projective model 
category of $I$-shaped diagrams of spectra, so that
$\phi$ is a levelwise acyclic fibration.  
Then each spectrum $\td{X}_\alpha$ is
(positive) fibrant, and we have
$$ \hocolim_\alpha X_\alpha = \colim_\alpha \td{X}_\alpha
\xrightarrow[\simeq]{\phi} \colim_\alpha X_\alpha. $$
\end{proof}

\begin{cor}\label{cor:5.3.3}
(Positive) fibrant discrete $G$-spectra are (positive)
fibrant as non-equivariant spectra.
\end{cor}

\begin{proof}
Let $X$ be a (positive) fibrant discrete $G$-spectrum.  
Let $U$ be an open normal
subgroup of $G$.  By Proposition~\ref{prop:iterate}(2) (and the obvious
analog in the positive fibrant case), $X$ is (positive)
fibrant as a
discrete $U$-spectrum.  Therefore, by Lemma~\ref{lem:Quillenfixedpoint}
(Corollary~\ref{cor:5.2.4}), 
the $U$-fixed
points $X^{U}$ is (positive) 
fibrant as a non-equivariant spectrum.  The formula
$$ X = \colim_{U \trianglelefteq_o G} X^U $$
shows that $X$ is (positive) fibrant as a non-equivariant spectrum.
\end{proof}

\begin{lem}\label{lem:algebracolimits}
Filtered colimits in both the category of $A$-modules and in the
category of commutative $A$-algebras are formed in the category of 
spectra.
\end{lem}

\begin{proof}
We treat the case of commutative $A$-algebras --- the case of $A$-modules
is similar.
Let $\{E_\alpha\}$ be a filtered diagram of commutative $A$-algebras.  Then
the colimit in the category of spectra 
is easily seen to have the structure of a commutative
$A$-algebra with multiplication given by
\begin{align*}
(\colim_\alpha E_\alpha) \wedge_A (\colim_\beta E_\beta) 
& \cong
\colim_\alpha \colim_\beta E_\alpha \wedge_A E_\beta \\
& \cong
\colim_\alpha E_\alpha \wedge_A E_\alpha \\
& \rightarrow \colim_\alpha E_\alpha.
\end{align*}
This filtered colimit is easily seen to satisfy the universal property.
\end{proof}

We shall henceforth 
always form
filtered colimits of spectra,
with
the understanding that we implicitly take functorial (positive)
fibrant
replacements before computing the filtered colimit, if the terms in the
colimit are not already (positive) fibrant.  
Therefore,
Corollary~\ref{cor:hocolim} implies that our filtered colimits will always
have the desired homotopy invariance, so we will
never need to take a filtered homotopy colimit of spectra.

If $\{X_\alpha\}$ is a filtered 
diagram of discrete commutative $G$-$A$-algebras, 
then whenever we are taking the filtered colimit, we shall 
implicitly be taking the filtered
colimit 
$\colim_\alpha (X_\alpha)_{fGA-\Alg}$
of the functorial fibrant replacements.  Note that 
the underlying spectrum of
each $(X_\alpha)_{fGA-\Alg}$ is positive fibrant.

Since
we often take filtered colimits of commutative
$A$-algebras, in the next result (whose proof
is similar to that of Corollary~\ref{cor:hocolim}),
we point out a nice relationship between this
colimit and the homotopy colimit in $\mathrm{Alg}_A$.
This result is also useful in relating the filtered colimits
of Section \ref{sec:Etheory} to the homotopy colimits
of commutative $S$-algebras that appear in
\cite[Def. 1.5, Sec. 6]{DevinatzHopkins}.

\begin{lem}
Suppose that $\{E_\alpha\}$ is a filtered diagram of fibrant commutative
$A$-algebras.
Then
there is an equivalence
$$ \colim_\alpha E_\alpha \simeq
{\hocolim_\alpha}^{\Alg} E_\alpha,
$$
where the homotopy colimit $\hocolim^\Alg$ 
is taken in the category of commutative
$A$-algebras.
\end{lem}

\subsection{Commutative hypercohomology algebras}

Let $E$ be a discrete commutative $G$-$A$-algebra.  
For any finite set $K$, the 
mapping spectrum $\Map(K,E)$ is naturally a commutative $A$-algebra,
by using the diagonal on $K$.  
Therefore, by Lemma~\ref{lem:algebracolimits},
the continuous mapping spectrum
$$ \Map^c(G, E) = \colim_{U \trianglelefteq_o G} \Map(G/U, E) $$
is a commutative $A$-algebra.

Since the category of spectra with the positive model structure is
Quillen equivalent to the category of spectra with the stable model 
structure \cite[Prop. 14.6]{MMSS},
there is an equivalence
$$ \HH^+_c(G; E) := {\holim_\Delta}^+ \Map^c(G^\bullet, E) \simeq
\holim_\Delta \Map^c(G^\bullet, E) = \HH_c(G; E) $$
between the homotopy limits computed in the positive and stable model
structures.  Since the homotopy limit of commutative $A$-algebras is
computed in the underlying category of spectra with the 
positive model structure, we have the following
lemma.

\begin{lem}
The hypercohomology spectrum $\HH_c(G; E)$ is equivalent to a commutative
$A$-algebra $\HH_c^+(G; E)$.
\end{lem}

When we take hypercohomology spectra of discrete commutative 
$G$-$A$-algebras, we shall
always implicitly be taking the homotopy limit with respect to the positive
model structure.

\section{Profinite Galois extensions}\label{sec:progalois}

Although the homotopy limit of $k$-local objects is $k$-local, it is not
true in general that $k$-localization commutes with homotopy limits.
We begin this section by explaining how, under a certain hypothesis,
Assumption~\ref{assumption} allows us to commute these two functors.
We then explain how a profinite Galois extension of a commutative
symmetric ring spectrum $A$ naturally gives rise to a
discrete commutative $G$-$A$ algebra, 
and we show that our consistency hypothesis allows us to
recover the intermediate finite Galois extensions using the 
homotopy fixed
point construction.

\subsection{Properties of $k$-localization}

Recall that we 
assume that the $k$-localization functor is given by $((-)_T)_M$, where
localization with respect to $T$ is smashing 
and $M$ is a finite spectrum
(Assumption~\ref{assumption}).  
These localizations are the functorial fibrant replacements in
appropriately localized model categories.
In this subsection we shall establish some lemmas
concerning such $k$-localizations.

The following lemma is immediate.

\begin{lem}
If $X$ is a $k$-local spectrum, then it is a $T$-local spectrum.
\end{lem}

\begin{lem}\label{lem:Mholim}
Let $\{X_i\}$ be a diagram of spectra.  Then there is an equivalence
$$ (\holim_i X_i)_M \simeq \holim_i (X_i)_M. $$
\end{lem}

\begin{proof}
The homotopy limit $\holim_i (X_i)_M$ is $M$-local, so there is a map
$$ f: (\holim_i X_i)_M \rightarrow \holim_i (X_i)_M. $$
Smashing with $M$, and using the fact that $M$ is a finite complex, we have
the following commutative diagram of equivalences
$$
\xymatrix{
M \wedge \holim_i X_i \ar[r]^\simeq \ar[d]_\simeq 
& \holim_i (M \wedge X_i) \ar[d]^\simeq
\\
M \wedge (\holim_i X_i)_M \ar[r]_{M \wedge f} 
& \holim_i (M \wedge (X_i)_M)
}
$$
from which we deduce that $f$ is an $M$-local equivalence.
Since $f$ is a map between $M$-local spectra, the map $f$ is an
equivalence.
\end{proof}

Arbitrary localizations do not commute with homotopy limits.
Our reason for making Assumption~\ref{assumption} on $k$-localization is
that it allows us to deduce the following corollary.

\begin{cor}\label{cor:kholim}
Let $\{X_i\}$ be a diagram of $T$-local spectra.  
Then there is an equivalence
$$ (\holim_i X_i)_k \simeq \holim_i (X_i)_k. $$
\end{cor}

Since $T$-localization is smashing, it possesses the following pleasant
properties.

\begin{lem}\label{lem:smashing} $\quad$
\begin{enumerate}
\item Colimits of $T$-local spectra are $T$-local.
\item If $X$ is a $T$-local spectrum and $Y$ is any spectrum, then 
$X \wedge Y$
is $T$-local.
\item If $X$ is $T$-local, then $\Map^c(G,X)$ is $T$-local.
\end{enumerate}
\end{lem}

\begin{lem}\label{lem:localfixedpoints}
Suppose that $f: X \rightarrow Y$ is a $k$-local equivalence of $T$-local
discrete $G$-spectra.  Then the induced map
$$ f_*: \HH_c(G; X)_k \rightarrow \HH_c(G; Y)_k $$
is an equivalence.
\end{lem}

\begin{proof}
Using Lemma~\ref{lem:smashing}, 
we see that the hypercohomology functor
$$ \HH_c(G;-) = \holim_\Delta \Map^c(G^{\bullet}, -) $$
sends $T$-local spectra to $T$-local spectra.
Therefore, we
just need to check that the map
$$ f_* : \HH_c(G; X) \rightarrow \HH_c(G; Y) $$
is an $M$-local equivalence.  Since $M$ is finite and $f$ is an $M$-local
equivalence, we have 
\begin{eqnarray*}
M \wedge \HH_c(G; X)
& \simeq & \HH_c(G; M \wedge X) \\
& \xrightarrow[\simeq]{(M \wedge f)_*} & \HH_c(G; M \wedge Y) \\
& \simeq & M \wedge \HH_c(G; Y),
\end{eqnarray*}
which implies that $M \wedge f_\ast$ is an
equivalence.
\end{proof}

Theorem~\ref{thm:cochaincomplex} implies the following corollary.

\begin{cor}\label{cor:localfixedpoints}
Suppose that $G$ has finite vcd and that 
$f: X \rightarrow Y$ is a $k$-local equivalence of $T$-local
discrete $G$-spectra.  Then the induced map
$$ f_*: X^{hG} \rightarrow Y^{hG} $$
is a $k$-local equivalence.
\end{cor}

We end this section by briefly 
explaining the relationship between localized fixed
points, and the right derived functor of fixed points taken with respect
to a localized model structure on discrete $G$-spectra.

Let $\mathcal{C}_k$ be the class of morphisms in $\Sigma\Sp_G$ which are
$k$-local equivalences on the underlying symmetric spectra.
Let $(\Sigma\Sp_G)_k$ be the model category obtained by localizing 
$\Sigma\Sp_G$ at $\mathcal{C}_k$.  Such a localized model category
exists: the arguments of Bousfield \cite{Bousfield} carry over in our
setting to show that the $k$-local equivalences in $\Sigma\Sp_G$ 
is the class of $f$-local
equivalences for some map $f$, and the localized model structure exists
since the model structure on $\Sigma\Sp_G$ is left proper and cellular
(see, for instance, \cite[Thm.~4.1.1]{Hirschhorn}).  The $k$-local model
structure $(\Sigma\Sp_G)_k$ is characterized by:
\begin{itemize}
\item the cofibrations are the underlying cofibrations of symmetric
spectra,
\item the weak equivalences are the underlying $k$-local equivalences of
symmetric spectra,
\item the fibrations are determined.
\end{itemize}
We shall let
$$ \alpha_{G,k,X}: X \rightarrow X_{f_kG} $$
denote a fibrant replacement functor in $(\Sigma\Sp_G)_k$.

The pair of functors $(\mathit{triv}, (-)^G)$ forms a Quillen pair on the
localized model categories
$$ \mathit{triv} : \Sigma\Sp_k \rightleftarrows (\Sigma\Sp_G)_k : (-)^G $$
since $\mathit{triv}$ preserves cofibrations and $k$-local equivalences.
We shall define $(-)^{h_kG}$ to be the Quillen right derived functor of
$(-)^G$ with respect to the $k$-local model structure:
$$ X^{h_kG} = (X_{f_kG})^G. $$

\begin{prop}\label{prop:hkG}
Let $X$ be a discrete $G$-spectrum.
\begin{enumerate}
\item For each open subgroup $U$ contained in $G$, the spectrum $X^{h_kU}$ is $k$-local.
\item The spectrum $X_{f_kG}$ is $T$-local.
\item Suppose $G$ has finite vcd and $X$ is $T$-local.  Then the map
$$ (X^{hG})_k \rightarrow X^{h_kG} $$
is an equivalence.
\end{enumerate}
\end{prop}

\begin{proof}
Observe that the same arguments used in the proof of
Proposition~\ref{prop:iterate} apply to show that the adjunction
$$ \Ind_U^G : (\Sigma\Sp_U)_k \rightleftarrows (\Sigma\Sp_G)_k: \Res_G^U $$
is a Quillen pair with respect to the $k$-local model structures.  In
particular, $\Res_G^U$ preserves $k$-locally fibrant objects, so $X_{f_kG}$
is $k$-locally fibrant as a discrete $U$-spectrum.  Since the fixed point
functor
$$ (-)^U: (\Sigma\Sp_U)_k \rightarrow \Sigma\Sp_k $$
is a right Quillen functor, it follows that
$$ (X_{f_kG})^U = X^{h_kU} $$
is $k$-local.  This establishes (1).
(2) then follows immediately from Lemma~\ref{lem:smashing} since there is an
isomorphism
$$ X_{f_kG} \cong \colim_{U \le_o G} (X_{f_kG})^U. $$
The map in (3) arises from the fact that $X_{f_kG}$ is fibrant as a
discrete $G$-spectrum, hence there is a $k$-local equivalence
$$ X_{fG} \rightarrow X_{f_kG}. $$
Note that if $X$ is assumed to be $T$-local, then $X_{fG}$ must also be
$T$-local, since it is equivalent to $X$.
Consider the following diagram.
$$
\xymatrix{
X^{hG} \ar@{=}[r] \ar[d] &
(X_{fG})^G \ar[r]^{\simeq} \ar[d] &
(X_{fG})^{hG} \ar[d]
\\
X^{h_kG} \ar@{=}[r] &
(X_{f_kG})^G \ar[r]_{\simeq} &
(X_{f_kG})^{hG} 
}
$$
The rightmost vertical arrow is a $k$-local equivalence by
Corollary~\ref{cor:localfixedpoints}.  Therefore the leftmost vertical 
arrow is a
$k$-local equivalence, and the induced morphism
$$ (X^{hG})_k \rightarrow X^{h_kG} $$
is an equivalence.
\end{proof}

\begin{rmk}\label{rmk:klocal}
Although the underlying non-equivariant spectrum of $X_{f_kG}$ is $T$-local,
it is
\emph{not}, in general, $k$-local.  
As pointed out in the proof of
Proposition~\ref{prop:hkG}, the spectrum $X_{f_kG}$ is a filtered colimit of
$k$-local spectra.
\end{rmk}

We record the following proposition, whose proof is identical to the proof
of Proposition~\ref{prop:iterate}.

\begin{prop}\label{prop:localiterate}
Let $X$ be a discrete $G$-spectrum, and suppose that $U$ is an open
subgroup of $G$.
\begin{enumerate}
\item If $U$ is normal in $G$, the 
$U$-fixed point spectrum $(X_{f_kG})^{U}$ is $k$-locally fibrant
as a discrete $G/U$-spectrum.

\item The $k$-locally 
fibrant discrete $G$-spectrum $X_{f_kG}$ is $k$-locally fibrant as a discrete
$U$-spectrum.

\item If $U$ is normal in $G$, the homotopy fixed point spectrum $X^{h_kU}$ is a
$G/U$-spectrum.

\item If $U$ is normal in $G$, 
there is an equivalence $X^{h_kG} \simeq (X^{h_kU})^{h_kG/U}$.
\end{enumerate}
\end{prop}

\subsection{Profinite Galois extensions as discrete $G$-spectra}

We first give the definition of a
profinite Galois extension, which is a slight
modification of the notion of a pro-$G$-Galois
extension, due to John Rognes (see
\cite[Section 8.1]{Rognes}).
Let $A$ be a $k$-local cofibrant commutative symmetric ring spectrum, 
let $E$ be a commutative
$A$-algebra, and let $G$ be a profinite group.

\begin{defn}[Profinite Galois extension]\label{def:galois}
The spectrum $E$ is a \emph{(profaithful) $k$-local $G$-Galois extension} of 
$A$ if
\begin{enumerate}
\item there is a directed system of (faithful) 
finite $k$-local $G/U_\alpha$-Galois extensions
$E_\alpha$ of $A$, for $\{U_\alpha\}$ a cofinal 
system of open normal subgroups of $G$;

\item all of the maps $E_\alpha \rightarrow E_\beta$
are $G$-equivariant and are cofibrations of underlying 
commutative $A$-algebras;

\item for $\alpha \le \beta$, letting $K_{\alpha, \beta}$ denote the
quotient $U_\alpha/U_\beta$,
the natural maps $E_\alpha \rightarrow E_\beta^{hK_{\alpha,\beta}}$ are
equivalences; and

\item the spectrum $E$ is the filtered colimit $\colim_\alpha E_\alpha$.
\end{enumerate}
\end{defn}

\begin{rmk}
The spectra $E_\alpha$ are $k$-local, but the spectrum $E$ need not be
$k$-local.  
However, Assumption \ref{assumption} does imply
that
$E$ is
$T$-local.
\end{rmk}

\begin{prop}
The spectrum $E$ in Definition \ref{def:galois} is a 
discrete commutative $G$-$A$-algebra.
\end{prop}

\begin{proof}
Clearly, $E_\alpha$ is a discrete commutative $G$-$A$-algebra.
Discrete commutative
$G$-$A$-algebras are closed under filtered colimits taken in the
category of commutative $A$-algebras.
\end{proof}

\begin{prop}[Rognes {\cite[Sec.~8.1]{Rognes}}]\label{prop:Rognesfuncspec}  
If $E$ is a $k$-local $G$-Galois extension of $A$, 
then there are natural equivalences:
\begin{gather*}
(E \wedge_A E)_k \xrightarrow{\simeq} (\Map^c(G, E))_k, \\
(E[[G]])_k \simeq F_A(E_k,E_k).
\end{gather*}
\end{prop}

\subsection{The consistent hypothesis}

In this subsection, we assume that
$E$ is a $k$-local profinite $G$-Galois extension of
$A$. We recall from the Introduction some terminology:
\begin{itemize}
\item $E$ is \emph{consistent} over $A$ if the map $A
\rightarrow A^\wedge_{k,E}$ is an equivalence; and

\item $E$ has \emph{finite vcd} if the
profinite group $G$ has finite virtual cohomological dimension.
\end{itemize}

\begin{prop}\label{prop:amitsurfixedpoint}
Let $E$ be a $k$-local profinite $G$-Galois extension of $A$ 
of finite vcd.
Then there is a natural equivalence
$$ A^\wedge_{k,E} \simeq (E^{hG})_k, $$
between the $k$-local Amitsur derived completion and the $k$-localization
of the homotopy fixed point
spectrum.
\end{prop}

\begin{proof}
By iterating Proposition~\ref{prop:Rognesfuncspec}, the natural map
$$ (\underbrace{E \wedge_A E \wedge_A \cdots \wedge_A E}_{n+1})_k \rightarrow 
(\Map^c(G^{n}, E))_k $$
is an equivalence.  
Totalizing the associated cosimplicial spectra and 
using Corollary~\ref{cor:kholim} and Theorem~\ref{thm:cochaincomplex}, 
we have:
\begin{eqnarray*}
A^\wedge_{k,E} & = & \holim_\Delta (E^{\wedge_A \bullet + 1})_k \\
& \xrightarrow{\simeq} &
\holim_{\Delta} (\Map^c(G^\bullet, E))_k \\
& \simeq & 
(\holim_{\Delta} \Map^c(G^\bullet, E))_k \\
& \simeq & (E^{hG})_k.
\end{eqnarray*}
\end{proof}

\begin{cor}\label{cor:consistent}
Let $E$ be a $k$-local profinite $G$-Galois extension of $A$
of finite vcd.  Then the extension is consistent
if and only if the $A$-algebra unit map
$$ A \rightarrow (E^{hG})_k $$
is an equivalence.
\end{cor}

We shall say that a $k$-local $A$-module $X$ is 
\emph{$k$-locally dualizable} if the map
$$ (D_A(X) \wedge_A X)_k \rightarrow F_A(X,X) $$
is an equivalence.  Here, $D_A(-) = F_A(-,A)$ is the Spanier-Whitehead dual
in the category of $A$-modules.  The following standard properties of
$k$-local dualizability are contained in \cite[Lem.~3.3.2(a),(b)]{Rognes}.

\begin{lem}\label{lem:rognesdualizable}
$\quad$
\begin{enumerate}
\item For $k$-local $A$-modules $X$, $Y$, and $Z$, the natural map
$$ (F_A(X, Y) \wedge_A Z)_k \rightarrow F_A(X, (Y \wedge_A Z)_k) $$
is an equivalence if either $X$ or $Z$ is $k$-locally dualizable.
\item If the $k$-local $A$-module $X$ is $k$-locally dualizable, then
$D_A(X)$ is also $k$-locally dualizable, and the natural map
$$ X \rightarrow D_A(D_A(X)) $$
is an equivalence.
\end{enumerate}
\end{lem}

We note the following useful consequence of $k$-local dualizability
which makes use of Assumption~\ref{assumption}.

\begin{lem}\label{lem:holimdualizable}
Suppose that $X$ is a $k$-local $A$-module which is $k$-locally dualizable,
and that $\{ Y_i \}$ is a diagram of $T$-local $A$-modules.  Then the
natural map
$$ (X \wedge_A \holim_i Y_i)_k \rightarrow (\holim_i X \wedge_A Y_i)_k $$
is an equivalence.
\end{lem}

\begin{proof}
The result follows from the following chain of equivalences:
\begin{align*}
(X \wedge_A \holim_i Y_i)_k
& \simeq (D_A(D_A(X)) \wedge_A \holim_i Y_i)_k \\
& \simeq (D_A(D_A(X)) \wedge_A (\holim_i Y_i)_k)_k \\
& \simeq F_A(D_A(X), (\holim_i Y_i)_k) \\
& \simeq F_A(D_A(X), \holim_i (Y_i)_k) \\
& \simeq \holim_i F_A(D_A(X), (Y_i)_k) \\
& \simeq \holim_i (D_A(D_A(X)) \wedge_A (Y_i)_k)_k \\
& \simeq \holim_i (D_A(D_A(X)) \wedge_A Y_i)_k \\
& \simeq (\holim_i X \wedge_A Y_i)_k,
\end{align*}
where the fourth and 
last equivalences follow from Corollary~\ref{cor:kholim}.
\end{proof}

We shall repeatedly use the following dualizability result
\cite[Props.~6.2.1, 6.4.7]{Rognes}.

\begin{prop}\label{prop:dualizable}
If $E$ is a \emph{finite} $k$-local Galois extension of $A$ (not
required to be faithful), then $E$ is a
$k$-locally dualizable $A$-module.  Also,
there is a natural \emph{discriminant} map (in the stable homotopy
category)
$$ E \rightarrow D_A(E), $$
which is an equivalence.
\end{prop}

Given a $k$-local profinite $G$-Galois extension 
$E = \colim_\alpha E_\alpha$ of $A$, each of the spectra $E_\alpha$
carries a
$G$-action, where the subgroup $U_\alpha$ acts trivially on $E_\alpha$.
Since the maps
$$ E_\alpha \rightarrow E_\beta $$
are $G$-equivariant, each of the maps
$$ E_\alpha \rightarrow \colim_\beta E_\beta = E $$
is $G$-equivariant.  Since the subgroup $U_\alpha$ acts trivially on
$E_\alpha$, we get an induced $G/U_\alpha$-equivariant map
$$ E_\alpha \rightarrow E^{U_\alpha} \rightarrow (E_{fGA-\mathrm{Alg}})^{U_\alpha}
\simeq E^{hU_\alpha}, $$
where the last equivalence follows from Proposition~\ref{prop:iterate}.

Being consistent implies the following consistency result.

\begin{lem}\label{lem:consistency}
Suppose that $E = \colim_\alpha E_\alpha$ 
is a consistent $k$-local profinite $G$-Galois extension of finite vcd.  
Then for each $\alpha$, the
natural $G/U_\alpha$-equivariant map
$$ E_\alpha \rightarrow ((E_{fGA-\mathrm{Alg}})^{U_\alpha})_k \simeq (E^{hU_\alpha})_k $$
is an equivalence.
\end{lem}

\begin{proof}
Since $E_\alpha$ is a $k$-local $G/U_\alpha$-Galois extension of $A$,
we have a chain of equivalences:
\begin{align*}
(E_\alpha \wedge_A E)_k 
& \simeq (E_\alpha \wedge_A \colim_{\beta \ge \alpha} E_\beta)_k \\
& \simeq (\colim_{\beta \ge \alpha} 
((E_{\alpha} \wedge_A E_{\alpha})\wedge_{E_\alpha} E_\beta))_k \\
& \simeq (\colim_{\beta \ge \alpha} 
(\Map(G/U_\alpha, E_{\alpha})\wedge_{E_\alpha} E_\beta))_k \\
& \simeq (\colim_{\beta \ge \alpha} \Map(G/U_\alpha, E_{\beta}))_k \\
& \simeq (\Map(G/U_\alpha, E))_k,
\end{align*}
where the $G$-action on the factor $E$
in $(E_\alpha \wedge_A E)_k$ corresponds
to the conjugation action on
$(\mathrm{Map}(G/U_\alpha,E))_k$.
By Corollary~\ref{cor:consistent}, the natural map
\begin{equation}\label{eq:consistency}
A \rightarrow (E^{hG})_k
\end{equation}
is an equivalence.
Smashing (\ref{eq:consistency}) over $A$ with $E_\alpha$, using
Theorem~\ref{thm:cochaincomplex},
employing the fact that $E_\alpha$ is a $k$-locally dualizable $A$-module
(Proposition~\ref{prop:dualizable} and Lemma~\ref{lem:holimdualizable}), 
and applying Corollary~\ref{cor:kholim} and Shapiro's
Lemma (\ref{lem:Shapiro}), we have the following equivalences:
\begin{align*}
E_\alpha 
& \simeq (E_\alpha \wedge_A A)_k \\
& \simeq (E_\alpha \wedge_A E^{hG})_k \\
& \simeq (E_\alpha \wedge_A \holim_\Delta \Map^c(G^\bullet, E))_k \\
& \simeq (\holim_\Delta \Map^c(G^\bullet, E_\alpha \wedge_A E))_k \\
& \simeq \holim_\Delta \Map^c(G^\bullet, E_\alpha \wedge_A E)_k \\
& \simeq \holim_\Delta \Map^c(G^\bullet, \Map(G/U_\alpha, E))_k \\
& \simeq (\holim_\Delta \Map^c(G^\bullet, \Map(G/U_\alpha, E)))_k \\
& \simeq (\Map(G/U_\alpha, E)^{hG})_k \\
& \simeq (E^{hU_\alpha})_k.
\end{align*}
\end{proof}

Adding the profaithful hypothesis allows us to expand the particular system
$\{ U_\alpha \}$ of open normal subgroups of $G$ to the collection of
all open normal subgroups of $G$.

\begin{prop}\label{prop:independent}
Let $E$ be a consistent profaithful $k$-local profinite $G$-Galois 
extension of $A$ of finite vcd.
\begin{enumerate}
\item For each
open normal subgroup $U$ of $G$, 
$(E^{hU})_k$ is a faithful $k$-local $G/U$-Galois extension of $A$.

\item
If $U \le V$ are a pair of open subgroups of $G$ with $U$ normal in $V$, then
$(E^{hU})_k$ is a faithful $k$-local $V/U$-Galois extension of $(E^{hV})_k$.
\end{enumerate}
\end{prop}

\begin{proof}
For both parts, we 
repeatedly use the fundamental theorem of Galois theory
\cite[Thm.~1.2]{Rognes}.

(1) Choose
$\alpha$ so that $U_\alpha$ is contained in $U$.  Then by
Lemma~\ref{lem:consistency}, there is an equivalence
$E_\alpha \simeq
(E^{hU_\alpha})_k$, so $(E^{hU_\alpha})_k$ 
is a faithful $k$-local $G/U_\alpha$-Galois extension of $A$.
Proposition~\ref{prop:iterate} implies that there are equivalences
$$ (E^{hU})_k \simeq ((E^{hU_\alpha})^{hU/U_\alpha})_k \simeq 
((E^{hU_\alpha})_k)^{hU/U_\alpha}. $$
Therefore, the fundamental theorem of 
Galois theory implies that $(E^{hU})_k$ is a faithful $k$-local
$G/U$-Galois extension of $A$.

(2) 
Let $N$ be an open normal subgroup of $G$ contained in $U$.
By (1), we know that $(E^{hN})_k$ is a faithful $k$-local
$G/N$-Galois extension of
$A$. 
By Proposition~\ref{prop:iterate}, we have
$$
(E^{hV})_k \simeq ((E^{hN})^{hV/N})_k \simeq ((E^{hN})_k)^{hV/N}.
$$
Thus, Galois theory implies that
$(E^{hN})_k$ is a faithful $k$-local $V/N$-Galois extension of $(E^{hV})_k$.
As before, we have
$$
(E^{hU})_k \simeq ((E^{hN})_k)^{hU/N}.
$$
Since $U/N$ is normal in $V/N$, with quotient $V/U$, 
Galois theory implies that $(E^{hU})_k$ is a faithful $k$-local 
$V/U$-Galois extension of $(E^{hV})_k$.
\end{proof}

\section{Closed homotopy fixed points of profinite Galois
extensions}\label{sec:closed}

Let $E$ be a consistent 
profaithful $k$-local profinite $G$-Galois extension of $A$ of finite
vcd.  We begin by showing that under these
hypotheses, 
the $H$-homotopy fixed points functor is well-behaved
when $H$ is 
any closed subgroup of $G$.  We then prove the forward
part of the profinite Galois correspondence, and we 
compute the homotopy type
of the function spectrum between arbitrary $k$-local closed homotopy 
fixed point spectra of $E$.

\subsection{Iterated Galois homotopy fixed points}\label{sec:Galoisiterate}

In this subsection 
we will extend the results of Section~\ref{sec:fcditerate} to
all closed subgroups $H$ of $G$.  Let $j: H \hookrightarrow G$ be the
inclusion
of the closed
subgroup $H$. Also, recall that, by hypothesis, $G$ has finite vcd.
Then we prove the following theorem and derive
consequences from it.

\begin{thm}\label{thm:Galoiscolim}
The natural map
$$ (\colim_{H \le U \le_o G} E^{hU})_k \rightarrow (E^{hH})_k $$
is an equivalence.
\end{thm}

In order to prove Theorem~\ref{thm:Galoiscolim}, we need to
introduce a spectrum $E'$ which is equivalent to $E$, but which
has better point-set level properties.
Observe that by Proposition~\ref{prop:independent}, the collection of
homotopy fixed point spectra 
$\{ (E^{hU})_k \}_{U \trianglelefteq_o G}$ gives rise to a $k$-local
profinite $G$-Galois extension 
$$ E' = \colim_{U \trianglelefteq_o G} (E^{hU})_k. $$
of $A$.  

Strictly speaking, given an open subgroup $V$ of $G$, 
the spectrum $E$ is \emph{not} an $(E^{hV})_k$-algebra.
Since the spectrum $(E^{hU})_k$ is an $(E^{hV})_k$-algebra for
every open normal subgroup $U$ of $G$ contained in $V$, 
the spectrum $E'$ is an $(E^{hV})_k$-algebra.
Furthermore, by Lemma~\ref{lem:consistency}, the map
$$ E = \colim_\alpha E_\alpha \rightarrow \colim_\alpha (E^{hU_\alpha})_k 
\cong E' $$
is an equivalence of discrete commutative $G$-$A$-algebras.

We shall need the following fundamental lemma.

\begin{lem}\label{lem:HcapV}
Let $V$ be an open normal subgroup of $G$. 
Then the natural map
$$ ((E^{hV})_k \wedge_{(E^{hHV})_k} (E')^{hH})_k \rightarrow ((E')^{h(H\cap
V)})_k, $$
induced from the commutative diagram 
$$
\xymatrix{
(E^{hHV})_k \ar[r] \ar[d] & (\colim_{U \trianglelefteq_o G}
(E^{hU})_k)^H \ar[r] &
(E')^{hH} \ar[d]
\\
(E^{hV})_k \ar[r] &
(\colim_{U \trianglelefteq_o G}
(E^{hU})_k)^{H \cap V} \ar[r] &
(E')^{hH \cap V}
}
$$
of commutative symmetric ring spectra,
is an equivalence.
\end{lem}

\begin{proof}
The lemma will be proven by showing that there exists a zig-zag of
$k$-local equivalences between
$$ (E^{hV})_k \wedge_{(E^{hHV})_k} (E')^{hH} \qquad \text{and} \qquad
(E')^{hH \cap V} $$
which are maps both of commutative $(E^{hV})_k$-algebras and
commutative $(E')^{hH}$-algebras.
Let $Q$ be the finite
group $HV/V \cong H/(H \cap V)$.
Note that, by 
Proposition~\ref{prop:independent}, the map
$$ (E^{hHV})_k \rightarrow (E^{hV})_k $$
is a $k$-local $Q$-Galois extension.

Observe that there is a zig-zag of maps:
\begin{eqnarray*}
(E^{hV})_k \wedge_{(E^{hHV})_k} (E')^{hH}
& = & (E^{hV})_k \wedge_{(E^{hHV})_k} (E'_{fH})^{H} \\
& \xrightarrow{w_1} & ((E^{hV})_k \wedge_{(E^{hHV})_k} E'_{fH})^{H} \\
& \xrightarrow{w_2} & ((E^{hV})_k \wedge_{(E^{hHV})_k} E'_{fH})^{hH} \\
& \xleftarrow{\simeq} & ((E^{hV})_k \wedge_{(E^{hHV})_k} E')^{hH}.
\end{eqnarray*}
The map $w_2$ above is the inclusion of fixed points into homotopy
fixed points.
Each of the maps above is a map of commutative $(E^{hV})_k$-algebras and
commutative $(E')^{hH}$-algebras.  Furthermore, the composite 
$w = w_2 \circ w_1$ is a
$k$-local equivalence, since we have a commutative diagram (in the
$k$-local stable
homotopy category)
$$
\xymatrix{
(E^{hV})_k \wedge_{(E^{hHV})_k} (E'_{fH})^H \ar[r]^-{u} \ar[d]_{w} &
(E^{hV})_k \wedge_{(E^{hHV})_k} \HH_c(H; E'_{fH}) \ar[d]^{w'}
\\
((E^{hV})_k \wedge_{(E^{hHV})_k} E'_{fH})^{hH} \ar[r]_-{u'} &
\HH_c(H; (E^{hV})_k \wedge_{(E^{hHV})_k} E'_{fH} )
}
$$
where the maps $u$ and $u'$ are equivalences by
Theorem~\ref{thm:cochaincomplex}, and the map $w'$ is seen to be a $k$-local
equivalence by using
the fact that $(E^{hV})_k$ is
a $k$-locally dualizable $(E^{hHV})_k$-module 
(Proposition~\ref{prop:dualizable} and 
Lemma~\ref{lem:holimdualizable}). 

The composite
\begin{equation}\label{line1-3}
\begin{split}
(E^{hV})_k \wedge_{(E^{hHV})_k} E'
& \cong 
((E^{hV})_k \wedge_{(E^{hHV})_k} (E^{hV})_k) \wedge_{(E^{hV})_k} E' 
\\
& \rightarrow
\Map(Q, (E^{hV})_k) \wedge_{(E^{hV})_k} E' 
\\
& \xrightarrow{\simeq} \Map(Q, E')
\end{split}
\end{equation}
is a $k$-local equivalence.
Here, the smash product $\wedge_{(E^{hV})_k}$ on the right-hand side of the
first line of
(\ref{line1-3}) uses the left $(E^{hV})_k$-module
structure on
$(E^{hV})_k \wedge_{(E^{hHV})_k} (E^{hV})_k$.
Under the isomorphism
$$ E' \cong (E^{hV})_k \wedge_{(E^{hV})_k} E' $$
the $G$-action on $E'$ is transformed to the diagonal action on 
$(E^{hV})_k \wedge_{(E^{hV})_k} E'$.
Therefore, under (\ref{line1-3}),
the $H$-action on $E'$ is transformed to the conjugation action on 
$\Map(Q, E')$.
Furthermore, under (\ref{line1-3}):
\begin{enumerate}
\item the
$E'$-algebra structure on $(E^{hV})_k \wedge_{(E^{hHV})_k} E'$ 
is sent to that given by
the inclusion
$$ E' \rightarrow \Map(Q, E') $$
of the constant maps, and  
\item
the $(E^{hV})_k$-module structure on $(E^{hV})_k \wedge_{(E^{hHV})_k} E'$
is sent to that given by the composite
$$ (E^{hV})_k \xrightarrow{\xi} \Map(Q, (E^{hV})_k) \rightarrow
\Map(Q, E'), $$
where $\xi$ is the adjoint of the $Q$-action map.
\end{enumerate}
Taking $H$-homotopy fixed points of (\ref{line1-3}) gives, by
Corollary~\ref{cor:localfixedpoints}, a $k$-local equivalence
$$ ((E^{hV})_k \wedge_{(E^{hHV})_k} E')^{hH} \rightarrow 
\Map(Q, E')^{hH}
$$
which is a map of commutative $(E^{hV})_k$-algebras and of commutative 
$(E')^{hH}$-algebras.
The proof of the lemma is completed by observing that the equivalence 
given by Shapiro's Lemma (Lemma~\ref{lem:Shapiro}) 
$$ (E')^{hH\cap V} \xrightarrow{\simeq} \Map(H/(H \cap V), E')^{hH} \cong
\Map(Q, E')^{hH}
$$
is a map of commutative $(E^{hV})_k$-algebras and of commutative
$(E')^{hH}$-algebras.
\end{proof}

\begin{proof}[Proof of Theorem~\ref{thm:Galoiscolim}]
Choose an open normal subgroup $V$ of $G$ of finite cohomological
dimension. 
By Proposition~\ref{prop:fcdHfp},
we see that the map
\begin{equation}\label{eq:fcd}
\colim_{H \le U \le_o HV} E^{h(U\cap V)} \rightarrow E^{h(H \cap V)}
\end{equation}
is an equivalence.  
Let $Q = HV/V \cong H/(H \cap V)$ be the finite quotient group.  
For each open subgroup $U$ of $HV$ containing $H$, we have 
$UV = HV$.  Therefore, there is an isomorphism
$Q = UV/V \cong U/(U \cap V)$.
By Proposition~\ref{prop:independent}, the extensions
\begin{gather*}
(E^{hU})_k \rightarrow (E^{h(U \cap V)})_k, \\
(E^{hUV})_k \rightarrow (E^{hV})_k
\end{gather*}
are faithful $k$-local $Q$-Galois extensions.
Therefore, by Remark~\ref{rmk:Hilbert90}, the norm maps give the following
equivalences:
\begin{gather}
((E^{h(U \cap V)})_{hQ})_k \simeq ((E^{h(U \cap V)})^{hQ})_k, 
\label{norm1}
\\
((E^{hV})_{hQ})_k \simeq ((E^{hV})^{hQ})_k.
\label{norm2}
\end{gather}

We may now establish that the natural map 
$$ (\colim_{H \le U \le_o G} E^{hU})_k \rightarrow (  E^{hH}  )_k $$
is an equivalence
by establishing a chain of intermediate equivalences.
Using Proposition~\ref{prop:iterate}, we have
\begin{align*}
(\colim_{H \le U \le_o G} E^{hU})_k
& \simeq (\colim_{H \le U \le_o HV} E^{hU})_k \\
& \simeq  (\colim_{H \le U \le_o HV} (E^{h(U \cap V)})^{hU/(U \cap V)})_k \\
& \cong  (\colim_{H \le U \le_o HV} (E^{h(U \cap V)})^{hQ})_k.
\end{align*}
Using equivalence (\ref{norm1}), we have
\begin{align*}
(\colim_{H \le U \le_o HV} (E^{h(U \cap V)})^{hQ})_k
& \simeq (\colim_{H \le U \le_o HV} (E^{h(U \cap V)})_{hQ})_k \\
& \simeq ((\colim_{H \le U \le_o HV} E^{h(U \cap V)})_{hQ})_k. \\
\end{align*}
Since $H \cap V$ is a subgroup of $V$, and $V$ has finite cohomological
dimension, $H\cap V$ has finite cohomological dimension.  Therefore, we may
apply Proposition~\ref{prop:fcdHfp} to deduce that there is an equivalence 
$$
((\colim_{H \le U \le_o HV} E^{h(U \cap V)})_{hQ})_k 
\simeq ((E^{h(H \cap V)})_{hQ})_k.
$$
We may now use Lemma~\ref{lem:HcapV} to deduce that there are equivalences
\begin{align*}
((E^{h(H \cap V)})_{hQ})_k
& \simeq (((E')^{h(H \cap V)})_{hQ})_k \\
& \simeq ((  (E^{hV})_k \wedge_{(E^{hHV})_k} (E')^{hH}  )_{hQ})_k \\
& \simeq ((  ((E^{hV})_k)_{hQ})_k \wedge_{(E^{hHV})_k} (E')^{hH}  )_k. \\
\end{align*}
Applying equivalence (\ref{norm2}), and applying
Proposition~\ref{prop:iterate},
we have
\begin{align*}
((  ((E^{hV})_k)_{hQ})_k \wedge_{(E^{hHV})_k} (E')^{hH}  )_k
& \simeq ((  (E^{hV})^{hQ})_k \wedge_{(E^{hHV})_k} (E')^{hH}  )_k \\
& \cong ((  (E^{hV})^{h(HV/V)})_k \wedge_{(E^{hHV})_k} (E')^{hH}  )_k \\
& \simeq (  (E^{hHV})_k \wedge_{(E^{hHV})_k} (E')^{hH}  )_k \\
& \simeq (  (E')^{hH}  )_k \\
& \simeq (  E^{hH}  )_k,
\end{align*}
as desired.
\end{proof}

Using Corollary~\ref{cor:mysterious}, 
Theorem~\ref{thm:Galoiscolim} 
has the following corollary.

\begin{cor}\label{cor:GaloisHfixed}
There is an equivalence $((E_{fG})^H)_k \simeq (E^{hH})_k$.
\end{cor}

We wish to derive from Corollary~\ref{cor:GaloisHfixed} an ``iterated
homotopy fixed
point theorem'' analogous to Corollary~\ref{cor:fcditerate}. The presence
of the $k$-localizations in the statement of
Corollary~\ref{cor:GaloisHfixed} presents some difficulties.  One might
hope that these could be overcome by using the $k$-local model
structures, and ask if $X_{f_kG}$ is $k$-locally fibrant as a discrete
$H$-spectrum. This is typically untrue, however.  For instance if $H = e$,
this would imply that the underlying spectrum of $X_{f_kG}$ is $k$-local,
which is typically false (see Remark~\ref{rmk:klocal}).

However, we can still formulate an iterated homotopy fixed point
theorem on the level of homotopy categories.  We first observe that we have
the following $k$-local analog of Proposition~\ref{prop:closediterate}.

\begin{prop}\label{prop:Galoisiterate}
Suppose that $N$ is a closed normal subgroup of $G$.
Then the functor 
$$ (-)^{N} : (\Sigma\Sp_G)_k \rightarrow (\Sigma\Sp_{G/N})_k $$
is a right Quillen functor. 
Let $R(-)^N$ denote its right derived functor.
Then the following diagram commutes up to
natural isomorphism.  
$$
\xymatrix{
\mathrm{Ho}((\Sigma\Sp_G)_k) \ar[rr]^{(-)^{h_kG}} \ar[dr]_{R(-)^N}
&& \mathrm{Ho}(\Sigma\Sp_k) 
\\
& \mathrm{Ho}((\Sigma\Sp_{G/N})_k) \ar[ur]_{(-)^{h_kG/N}}
}
$$
\end{prop}

\begin{proof}
The adjoint pair $(\Res_{G/N}^{G}, (-)^N)$ forms a Quillen pair
$$ \Res^G_{G/N}: (\Sigma\Sp_{G/N})_k \rightleftarrows (\Sigma\Sp_G)_k :
(-)^N $$
since $\Res^{G}_{G/N}$ is easily seen to preserve cofibrations and
$k$-local equivalences.
The functor
$(-)^N$ has a Quillen right derived functor $R(-)^N$.   For any discrete
$G$-spectrum $X$, the functor $R(-)^N$ is computed to be
$$ R(X)^N = (X_{f_kG})^N. $$
Since $(-)^N$ is a right Quillen functor, $(X_{f_kG})^N$ is $k$-locally
fibrant as a discrete $G/N$-spectrum.  We therefore have
$$ X^{h_kG} = (X_{f_kG})^G = ((X_{f_kG})^N)^{G/N} \cong
((X_{f_kG})^N)^{h_kG/N} = (R(X)^N)^{h_kG/N} $$
in the $k$-local homotopy category.
\end{proof}

\begin{rmk}
Note that the previous proposition does not make use of the finite vcd
hypothesis on $G$.
\end{rmk}

As in Section~\ref{sec:fcditerate}, the difficulty with fully interpreting
Proposition~\ref{prop:Galoisiterate} as an iterated homotopy fixed
point theorem is that in general, 
the Quillen right derived functor of the functor
$$ (-)^N : (\Sigma\Sp_G)_k \rightarrow (\Sigma\Sp_{G/N})_k $$
may not agree, on the level of  
underlying non-equivariant spectra, 
with the Quillen right derived functor of the functor
$$ (-)^N: (\Sigma\Sp_N)_k \rightarrow \Sigma\Sp_k $$
in the $k$-local homotopy category.

However, the composite 
$$ \mathrm{Ho}((\Sigma\Sp_G)_k) \xrightarrow{R(-)^N}
\mathrm{Ho}((\Sigma\Sp_{G/N})_k)
\xrightarrow{\mathcal{U}} \mathrm{Ho}(\Sigma\Sp_k)
$$
is easily seen to be a total right derived functor (in the sense
of \cite[Def.~8.4.7]{Hirschhorn}) of the composite
$$ (\Sigma\Sp_G)_k \xrightarrow{(-)^N}  
(\Sigma\Sp_{G/N})_k \xrightarrow{\mathcal{U}} \Sigma\Sp_k.
$$
The universal property of total right derived functor gives a canonical
natural transformation in $\mathrm{Ho}(\Sigma\Sp_k)$
$$ \gamma_X: R(X)^N \rightarrow X^{h_kN}. $$

\begin{thm}\label{thm:galoisiterate}
The map
$$ \gamma_E: R(E)^N \rightarrow E^{h_kN} $$
is an isomorphism in $\mathrm{Ho}(\Sigma\Sp_k)$.
\end{thm}

\begin{proof}
Since $E_{f_kG}$ is a discrete $G$-spectrum, we have
$$ (E_{f_kG})^N = \colim_{N \le U \le_o G} (E_{f_kG})^{U}. $$
By Proposition~\ref{prop:localiterate}, $E_{f_kG}$ is $k$-locally fibrant
as a discrete $U$-spectrum.  We therefore have a $k$-local equivalence
$$ (E_{f_kG})^N \xrightarrow{\simeq_k} \colim_{N \le U \le_o G} E^{h_kU}. $$

We first observe that the map
$$ (E_{fG})^N \rightarrow (E_{f_kG})^N $$
is a $k$-local equivalence by examining the following diagram.
$$
\xymatrix{
(E_{fG})^N \ar[r] \ar[d]_{\simeq} &
(E_{f_kG})^N \ar[d]^{\simeq_k} 
\\
\colim_{N \le U \le_o G}E^{hU} \ar[r] &
\colim_{N \le U \le_o G}E^{h_kU} &
}
$$
The left map is the equivalence of Corollary~\ref{cor:mysterious}.  The
bottom map is a $k$-local equivalence by Proposition~\ref{prop:hkG}.  

For a discrete $G$-spectrum $X$,
the natural zig-zag
$$ X_{f_kG} \xrightarrow{\simeq_k} 
(X_{f_kG})_{f_kN} \xleftarrow{\simeq_k} X_{f_kN} $$
of discrete $N$-spectra induces the natural transformation
$$ \gamma_X: R(X)^N = (X_{f_{k}G})^N \rightarrow X^{h_kN} $$
in $\mathrm{Ho}(\Sigma\Sp_k)$.  
We see that $\gamma_E$ is a $k$-local equivalence by
studying the following diagram in $\mathrm{Ho}(\Sigma\Sp_k)$.
$$
\xymatrix{
(E_{fG})^N \ar[r] \ar[d] &
E^{hN} \ar[d] 
\\
(E_{f_kG})^{N} \ar[r] &
E^{h_kN}
}
$$
We showed that the left map is a $k$-local equivalence.  The top map is a
$k$-local equivalence by Corollary~\ref{cor:GaloisHfixed}, and the right
map is a $k$-local equivalence by Proposition~\ref{prop:hkG}.
\end{proof}

\subsection{Intermediate Galois extensions}\label{sec:Galois}

In this subsection we will prove the forward direction of the 
profinite Galois
correspondence.  

\begin{thm}\label{thm:Galois}
Suppose that $H$ is a closed subgroup of $G$.  
\begin{enumerate}

\item
The spectrum $E$ is $k$-locally $H$-equivariantly equivalent to a
consistent profaithful $k$-local  
$H$-Galois extension of $(E^{hH})_k$ of finite
vcd.

\item
If $H$ is a normal subgroup of $G$, then the spectrum $E^{hH}$ is
$k$-locally $G/H$-equivariantly equivalent to a
profaithful $k$-local 
$G/H$-Galois extension of $A$.  If the quotient $G/H$
has finite vcd, then this extension is consistent
(and of finite vcd) over $A$.

\end{enumerate}
\end{thm}

\begin{rmk} 
It is
useful to note that if $G$ is a compact $p$-adic
analytic group, then for any closed normal subgroup $H$ of $G$,  
the quotient
group $G/H$ is a compact $p$-adic analytic group, and therefore must also
have finite vcd
\cite[Thm.~9.6]{DSMS}, \cite{Symonds}.
\end{rmk}

\begin{rmk}
Theorem~\ref{thm:Galois}, when applied to the $K(n)$-local profinite 
Galois extension
$F_n$ of $S_{K(n)}$, provides an extension of \cite[Thm.~5.4.4]{Rognes}.
\end{rmk}

\begin{proof}[Proof of part (1)] 
We shall prove that $E$ is $k$-locally $H$-equivariantly 
equivalent, as a discrete commutative 
$H$-algebra, to a commutative $(E^{hH})_k$-algebra $L$.
We will prove that the map 
$$ (E^{hH})_k \rightarrow L $$
is a consistent profaithful $k$-local $H$-Galois extension by proving
that there is a commutative diagram of commutative symmetric ring spectra
\begin{equation}\label{diag:LL'}
\xymatrix{
L \ar[r]^\simeq  
& L' 
\\
(E^{hH})_k \ar[r]_\simeq \ar[u]
& ((E')^{hH})_k \ar[u]
}
\end{equation}
where $E'$ is the discrete commutative $G$-$A$-algebra, equivalent to
$E$, introduced
before Lemma~\ref{lem:HcapV}, and $L'$ is a profaithful $k$-local 
$H$-Galois extension of $((E')^{hH})_k$.  
The proof concludes by showing
directly that $L$ is consistent over $(E^{hH})_k$.

Since $H$ is a closed subgroup of $G$, a group of finite vcd, 
we may conclude that $H$ has finite vcd.
The system $\{ H \cap V \}_{V \trianglelefteq_o G}$
is cofinal in the system of open normal subgroups of $H$.  Let $U = H \cap
V$ be one of these open normal subgroups.

{\it $((E')^{hU})_k$ is $k$-locally $H/U$-Galois over $((E')^{hH})_k$.}
We must check that the last two conditions of
Definition~\ref{defn:finiteGalois} are satisfied.
By Proposition~\ref{prop:iterate} and Corollary~\ref{cor:kholim}, we have
\begin{align*}
(((E')^{hU})_k)^{hH/U}
& \simeq (((E')^{hU})^{hH/U})_k \\
& \simeq ((E')^{hH})_k,
\end{align*}
which verifies the second condition.  The third condition is verified
through the use of Lemma~\ref{lem:HcapV} and the fact that $(E^{hV})_k$ is
a faithful $k$-local $HV/V$-Galois extension of $(E^{hHV})_k$
(Proposition~\ref{prop:independent}):
\begin{align*}
& (((E')^{hU})_k \wedge_{((E')^{hH})_k} ((E')^{hU})_k)_k \\
& \quad \quad \simeq 
( ( ((E')^{hH})_k \wedge_{(E^{hHV})_k} (E^{hV})_k )_k 
\wedge_{((E')^{hH})_k} ( ((E')^{hH})_k \wedge_{(E^{hHV})_k} (E^{hV})_k  )_k )_k\\
& \quad \quad \simeq 
( ((E')^{hH})_k \wedge_{(E^{hHV})_k} (E^{hV})_k 
\wedge_{(E^{hHV})_k} (E^{hV})_k  )_k \\
& \quad \quad \simeq 
( ((E')^{hH})_k \wedge_{(E^{hHV})_k} \Map(HV/V, (E^{hV})_k)  )_k \\
& \quad \quad \cong 
( ((E')^{hH})_k \wedge_{(E^{hHV})_k} \Map(H/U, (E^{hV})_k)  )_k \\
& \quad \quad \simeq 
\Map(H/U, (((E')^{hH})_k \wedge_{(E^{hHV})_k} (E^{hV})_k)_k) \\
& \quad \quad 
\simeq \Map(H/U, ((E')^{hU})_k).
\end{align*}

{\it $((E')^{hU})_k$ is $k$-locally faithful over $((E')^{hH})_k$.}
Suppose $M$ is an $((E')^{hH})_k$-module
and that we have
$$ (((E')^{hU})_k \wedge_{((E')^{hH})_k} M)_k \simeq \ast. $$
We must show $M_k$ is null.
We use Lemma~\ref{lem:HcapV} to deduce
\begin{align*}
\ast
& \simeq 
(( (E^{hV})_k \wedge_{(E^{hHV})_k} ((E')^{hH})_k )_k \wedge_{((E')^{hH})_k} M)_k \\
& \simeq 
((E^{hV})_k \wedge_{(E^{hHV})_k} \wedge M)_k.
\end{align*}
By Proposition~\ref{prop:independent}, we deduce that 
$(E^{hV})_k$ is $k$-locally faithful over $(E^{hHV})_k$, so we may conclude
that $M_k$ is null.

Let $L$ and $L'$ be defined by the colimits 
\begin{align*}
L & := \colim_{V \trianglelefteq_o G} (E^{hH \cap V})_k, \\
L' & := \colim_{V \trianglelefteq_o G} ((E')^{hH \cap V})_k.
\end{align*}
We have shown that the spectrum $L'$ is a profaithful $k$-local $H$-Galois
extension of $((E')^{hH})_k$.  Furthermore, the equivalence
$$ E \rightarrow E' $$
of discrete commutative $G$-$A$-algebras gives rise to
Diagram~(\ref{diag:LL'}).

{\it $E$ is $k$-locally $H$-equivariantly equivalent to $L$.} 
By Corollary~\ref{cor:GaloisHfixed}, for each $V$ we
have
$$ ((E_{fG})^{H \cap V})_k \simeq (E^{hH \cap V})_k. $$
Since $E_{fG}$ is a discrete $H$-spectrum, we have:
\begin{align*}
(E_{fG})_k 
& = (\colim_{V \trianglelefteq_o G} (E_{fG})^{H \cap V})_k \\
& \simeq (\colim_{V \trianglelefteq_o G} (E^{hH \cap V})_k)_k \\
& = L_k.
\end{align*}
The fibrant replacement map $E \rightarrow E_{fG}$ is an $H$-equivariant
equivalence.

{\it $L$ is consistent over $(E^{hH})_k$.}
By Corollary~\ref{cor:consistent}, we just need to check that the map
\begin{equation}\label{eq:EtoL}
(E^{hH})_k \rightarrow (L^{hH})_k
\end{equation}
is an equivalence.  We have already seen that 
the map $E \rightarrow L$
is a
$k$-local equivalence.  By 
Corollary~\ref{cor:localfixedpoints}, we see that the map of
(\ref{eq:EtoL}) is an equivalence.
\end{proof}

\begin{proof}[Proof of part (2)]
Let $K$ be the colimit $\colim_{U \trianglelefteq_o G} (E^{hHU})_k$.
Since $H$ is normal in $G$, the groups $HU$ are open normal subgroups of
$G$.  By
Proposition~\ref{prop:independent}, the spectra $(E^{hHU})_k$ are $k$-local
faithful 
$G/HU$-Galois extensions of $A$.  Therefore, $K$ is a $k$-local 
profaithful
$G/H$-Galois extension of $A$.
The spectrum $K$ is $k$-locally equivalent to $E^{hH}$ by
Theorem~\ref{thm:Galoiscolim}.

Suppose that $G/H$ is of finite vcd.  We are
left with showing that $K$ is consistent over $A$.  By
Corollary~\ref{cor:consistent}, it suffices to check that the map
$$ A \rightarrow (K^{hG/H})_k $$
is an equivalence.
Using Corollary~\ref{cor:kholim}, Theorem~\ref{thm:Galoiscolim}, 
Corollary~\ref{cor:GaloisHfixed}, and the fact that $E$ is consistent over $A$,
we have:
\begin{align*}
(K^{hG/H})_k
& = ((\colim_{U \trianglelefteq_o G} (E^{hHU})_k)^{hG/H})_k \\
& \simeq ((\colim_{U \trianglelefteq_o G} E^{hHU})^{hG/H})_k \\
& \simeq (((E_{fG})^H)^{hG/H})_k \\
& \simeq (E^{hG})_k \\
& \simeq A.
\end{align*}
\end{proof}

\subsection{Function spectra}\label{sec:funcspec}

In this section, we prove the following theorem.

\begin{thm}\label{thm:funcspec}
Let $H$ and $K$ be closed subgroups of $G$.
Then there is an equivalence
$$ F_A((E^{hH})_k,(E^{hK})_k) \simeq ((E[[G/H]])^{hK})_k, $$
where $E[[G/H]]$ has the
diagonal $K$-action.
\end{thm}

\begin{cor}\label{cor:trivial}
If $H$ and $K$ are closed subgroups of $G$, and the left action of $K$ on
$G/H$ is trivial, then there is an equivalence
$$ F_A((E^{hH})_k, (E^{hK})_k) \simeq (E^{hK}[[G/H]])_k. $$
\end{cor}

\begin{proof}
We have the following sequence of equivalences:
\begin{align*}
F_A((E^{hH})_k, (E^{hK})_k) 
& \simeq ((E[[G/H]])^{hK})_k \\
& \simeq (\holim_{H \le U \le_o G} (E[G/U])^{hK})_k \\
& \simeq (\holim_{H \le U \le_o G} E^{hK}[G/U])_k \\
& \simeq (E^{hK}[[G/H]])_k.
\end{align*}
\end{proof}

\begin{rmk}
The conclusion of Corollary~\ref{cor:trivial} is typically far from true for
arbitrary $H$ and $K$.  For instance, 
let $n$ be odd.
It is shown in \cite[Prop.~16]{Strickland} that
in the case of $k = K(n)$, $A = S_{K(n)}$, $E = F_n$, $G = \GG_n$, 
$H = \{e \}$,
and $K = \GG_n$,
the $K(n)$-local Spanier-Whitehead dual of $E_n$ is given
by:
$$ F(E_n, E_n^{h\GG_n}) \simeq F(E_n, S_{K(n)}) \simeq \Sigma^{-n^2}E_n
\not \simeq E_n. $$  
\end{rmk}

The remainder of this section will be spent proving
Theorem~\ref{thm:funcspec}.
We first prove some technical lemmas.
Recall that an $A$-module $X$ is said to be $k$-\emph{locally}
\emph{$F$-small} if
the natural map
$$ \colim_i F_A(X,Y_i) \rightarrow F_A(X,(\colim_i Y_i)_k) $$
is a $k$-local equivalence
for every filtered diagram $\{Y_i\}_i$ of $k$-local $A$-modules.  
Observe that if $X$ is a $k$-locally dualizable $A$-module,
then it is $k$-locally $F$-small, since we have:
\begin{align*}
(\colim_i F_A(X,Y_i))_k
& \simeq (\colim_i (D_A(X) \wedge_A Y_i)_k)_k \\
& \simeq (\colim_i (D_A(X) \wedge_A Y_i))_k \\
& \simeq (D_A(X) \wedge_A \colim_i Y_i)_k \\
& \simeq F_A(X,(\colim_i Y_i)_k).
\end{align*}

\begin{lem}\label{lem:technicallemma1}
Suppose that $X$ is an $A$-module which is $k$-locally $F$-small, and
that $Y$ is a $k$-local $A$-module.
Let $T = \lim_i T_i$ be a profinite set.
Then the natural map
$$ \Map^c(T,F_A(X,Y)) \rightarrow F_A(X,\Map^c(T,Y)_k) $$
is a $k$-local equivalence.
\end{lem}

\begin{proof}
We have:
\begin{align*}
\Map^c(T,F_A(X,Y))_k
& = (\colim_i \Map(T_i, F_A(X, Y)))_k \\
& \cong (\colim_i F_A(X,\Map(T_i, Y)))_k \\
& \simeq F_A(X, (\colim_i \Map(T_i,Y))_k) \\ 
& = F_A(X, \Map^c(T,Y)_k). 
\end{align*}
\end{proof}

The following lemma is immediate from the definition of $\Map^c$.

\begin{lem}\label{lem:mapc}
Let $\{Y_j\}_j$ be a filtered diagram of spectra and let $T = \lim_i T_i$
be a profinite set.  Then the natural map
$$ \colim_j \Map^c(T, Y_j) \rightarrow \Map^c(T, \colim_j Y_j) $$
is an isomorphism.
\end{lem}

\begin{lem}\label{lem:xi}
Let $U$ be an open subgroup of $G$, and let $V$ be an open normal subgroup 
of $G$, with $V \le U$. Then 
there is a map of discrete $G$-$A$-modules 
$$ \xi :  A[G/U] \wedge_A (E^{hU})_k \rightarrow (E^{hV})_k, $$ 
where $G$ acts on the source of $\xi$ by acting only on $A[G/U]$.
\end{lem}

\begin{proof}
To produce the map $\xi$, it suffices to construct the adjoint map of
sets
$$ \td{\xi} : G/U \rightarrow \Mod_A((E^{hU})_k, (E^{hV})_k). $$
Observe that for $g \in G$ the $G$-action map
$$ g: E \rightarrow E $$
descends to a map
$$ \br{g}: E^{hU} \rightarrow E^{hV}, $$
which localizes to give
$$ \td{\xi}(gU) = (\br{g})_k: (E^{hU})_k \rightarrow (E^{hV})_k. $$
It is easy to check that this map 
is independent of choice of coset
representative.
To show that $\xi$ is $G$-equivariant we must show that $\td{\xi}$
is $G$-equivariant, where $G$ acts on the morphism set
$\Mod_A((E^{hU})_k,(E^{hV})_k)$ by postcomposition.  This is clear from the
definition of $\td{\xi}$.
\end{proof}

The map $\xi$ gives rise to a map
$$ \psi_V: (E^{hV})_k [G/U] \rightarrow F_A((E^{hU})_k, (E^{hV})_k) $$
as follows:
the adjoint $\td{\psi}_V$ is given by the composite
\begin{eqnarray*}
\td{\psi}_{V}: (E^{hV})_k \wedge_A (A[G/U] \wedge_A (E^{hU})_k)
& \xrightarrow{1 \wedge \xi} &
(E^{hV})_k \wedge_A (E^{hV})_k
\\
& \xrightarrow{\mu} & 
(E^{hV})_k,
\end{eqnarray*}
where $\mu : (E^{hV})_k \wedge_A (E^{hV})_k \rightarrow (E^{hV})_k$ 
is the multiplication map of the $A$-algebra $(E^{hV})_k.$
The map $\psi_V$ is easily checked to be $G$-equivariant, where $G$
acts diagonally on the source and acts on the target 
through its action on the term $(E^{hV})_k$. 

\begin{lem}\label{lem:weak}
Let $U$ be an open subgroup of $G$, and suppose that $V$ is an open normal
subgroup of $G$ contained in $U$.  Then the map 
$$ \psi_V: (E^{hV})_k [G/U] \xrightarrow{\simeq} F_A((E^{hU})_k, (E^{hV})_k) $$
is an equivalence.
\end{lem}

\begin{proof}
By Proposition~\ref{prop:independent}, the spectrum $(E^{hV})_k$ is
a $k$-local $G/V$-Galois extension of $A$, and by 
Proposition~\ref{prop:iterate}, there is an
equivalence
$$ E^{hU} \simeq (E^{hV})^{hU/V}. $$
By Proposition~\ref{prop:dualizable}, the $A$-module $(E^{hV})_k$ is
$k$-locally dualizable.  Making use of Corollary~\ref{cor:kholim} and 
Lemma~\ref{lem:holimdualizable}, we have:
\begin{align*}
((E^{hV})_k \wedge_A (E^{hU})_k)_k
& \simeq ((E^{hV})_k \wedge_A ((E^{hV})^{hU/V})_k)_k \\
& \simeq ((E^{hV})_k \wedge_A ((E^{hV})_k)^{hU/V})_k \\
& \simeq (((E^{hV})_k \wedge_A (E^{hV})_k)^{hU/V})_k \\
& \simeq \Map(G/V, (E^{hV})_k)^{hU/V} \\
& \simeq \Map(G/U, (E^{hV})_k),
\end{align*}
where the last equivalence follows from the fact that the right 
$U/V$-action on $G/V$ is free.
Applying $F_{(E^{hV})_k}(-,(E^{hV})_k)$ to both sides, we have:
\begin{align*}
F_A((E^{hU})_k, (E^{hV})_k)
& \cong F_{(E^{hV})_k}((E^{hV})_k \wedge_A (E^{hU})_k, (E^{hV})_k) \\
& \simeq F_{(E^{hV})_k}(((E^{hV})_k \wedge_A (E^{hU})_k)_k, (E^{hV})_k) \\
& \simeq F_{(E^{hV})_k}(\Map(G/U, (E^{hV})_k),(E^{hV})_k) \\
& \simeq F_{(E^{hV})_k}((E^{hV})_k,(E^{hV})_k[G/U]) \\
& \simeq (E^{hV})_k[G/U].
\end{align*}
This sequence of equivalences may be checked to be compatible with
the map $\psi_V$.
\end{proof}

\begin{lem}\label{lem:technicallemma2}
Let $U$ be an open subgroup of $G$.  There is an 
equivalence of discrete $G$-spectra
$$ \phi: 
E[G/U] \xrightarrow{\simeq}
\colim_{V \trianglelefteq_o G, V \le U} 
F_A((E^{hU})_k, (E^{hV})_k). $$
Here, $G$ is acting diagonally
on the left-hand side and acting only on each $(E^{hV})_k$ 
on the right-hand side.
\end{lem}

\begin{rmk}
Let $V$ be an open normal subgroup of $G$ and let $U$ be an open
subgroup of $G$. 
Since $E^{hV}$ is a $G/V$-spectrum, 
the function spectrum  
$$F_A((E^{hU})_k, (E^{hV})_k) $$ 
is a 
discrete $G$-spectrum, where $G$ is acting only on the spectrum $(E^{hV})_k$.  
The colimit 
$$ \colim_{V \trianglelefteq_o G, V \le U} 
F_A((E^{hU})_k, (E^{hV})_k) $$ 
is therefore a discrete $G$-spectrum.
\end{rmk}

\begin{proof}[Proof of Lemma~\ref{lem:technicallemma2}]
The map $\phi$ is given as the composite
\begin{eqnarray*}
E[G/U] 
& \xrightarrow{\simeq} & 
({\colim_{V \trianglelefteq_o G}} (E^{hV})_k)[G/U]
\\
& \xrightarrow{\simeq} & 
{\colim_{V \trianglelefteq_o G}} ((E^{hV})_k[G/U])
\\
& \xrightarrow{\psi} & 
\colim_{V \trianglelefteq_o G, V \le U} F_A((E^{hU})_k,(E^{hV})_k).
\end{eqnarray*}
The map $\psi$
is the colimit of the equivalences 
$$ \psi_V: (E^{hV})_k[G/U] \rightarrow 
F_A((E^{hU})_k, (E^{hV})_k) $$
of Lemma~\ref{lem:weak}.
Therefore, $\psi$
is an equivalence,
so $\phi$ is an equivalence.
\end{proof}

\begin{proof}[Proof of Theorem~\ref{thm:funcspec}]
Fix $U$ to be an open subgroup of $G$, and suppose that $s
\ge 0$.  Using Lemma~\ref{lem:technicallemma2}, we have
$$
\Map^c(K^s, E[G/U])_k 
\simeq 
\Map^c(K^s, \colim_{V \trianglelefteq_o G, V \le U} 
F_A((E^{hU})_k, (E^{hV})_k))_k.
$$
Lemma~\ref{lem:mapc} gives
\begin{multline*}
\Map^c(K^s, \colim_{V \trianglelefteq_o G, V \le U} 
F_A((E^{hU})_k, (E^{hV})_k))_k
\\
\simeq 
(\colim_{V \trianglelefteq_o G, V \le U}
\Map^c(K^s, 
F_A((E^{hU})_k, (E^{hV})_k)))_k.
\end{multline*}
Using Lemma~\ref{lem:technicallemma1}, we get
\begin{multline*}
(\colim_{V \trianglelefteq_o G, V \le U}
\Map^c(K^s, 
F_A((E^{hU})_k, (E^{hV})_k)))_k
\\
\simeq 
(\colim_{V \trianglelefteq_o G, V \le U} 
F_A((E^{hU})_k, \Map^c(K^s,(E^{hV})_k)_k))_k.
\end{multline*}
Now, by \cite[Lem.~7.2.5]{Rognes}, the spectrum $(E^{hU})_k$ is a $k$-locally
dualizable $A$-module.  Therefore, as discussed prior to
Lemma~\ref{lem:technicallemma1}, it is $k$-locally $F$-small.  Therefore,
there is an equivalence
\begin{multline*}
(\colim_{V \trianglelefteq_o G, V \le U} 
F_A((E^{hU})_k, \Map^c(K^s,(E^{hV})_k)_k))_k
\\
\simeq 
F_A((E^{hU})_k, (\colim_{V \trianglelefteq_o G, V \le U}
\Map^c(K^s,(E^{hV})_k))_k).
\end{multline*}
Applying Lemma~\ref{lem:mapc} again, we have
\begin{multline*}
F_A((E^{hU})_k, (\colim_{V \trianglelefteq_o G, V \le U}
\Map^c(K^s,(E^{hV})_k))_k)
\\
\simeq
F_A((E^{hU})_k, \Map^c(K^s, \colim_{V \trianglelefteq_o G, V \le U}
(E^{hV})_k)_k)
\\
\simeq
F_A((E^{hU})_k, \Map^c(K^s,E)_k).
\end{multline*}
We therefore have an equivalence of cosimplicial spectra
\begin{equation}\label{eq:cosimp1}
\begin{split}
(\Gamma_K^\bullet E[G/U])_k 
& \cong
\Map^c(K^\bullet, E[G/U])_k \\
& \simeq
F_A((E^{hU})_k, \Map^c(K^\bullet,E)_k) \\
& \cong
F_A((E^{hU})_k, (\Gamma^\bullet_K E)_k).
\end{split}
\end{equation}
We will now deduce the desired equivalence by producing a string of
intermediate equivalences.
By Theorem~\ref{Davis} and Corollary~\ref{cor:kholim}, we have
$$
((E[[G/H]])^{hK})_k
\simeq
\holim_{H \le U \le_o G} \holim_\Delta
(\Gamma_K^\bullet E[G/U])_k.
$$
Using equivalence~(\ref{eq:cosimp1}), we have
\begin{align*}
\holim_{H \le U \le_o G} \holim_\Delta
(\Gamma_K^\bullet E[G/U])_k
& \simeq
\holim_{H \le U \le_o G} \holim_\Delta  
F_A((E^{hU})_k, (\Gamma_K^\bullet E)_k)
\\
& \simeq 
\holim_{H \le U \le_o G}   
F_A((E^{hU})_k, \holim_\Delta (\Gamma_K^\bullet E)_k).
\end{align*}
By Corollary~\ref{cor:kholim}, we have
$$
\holim_{H \le U \le_o G}   
F_A((E^{hU})_k, \holim_\Delta (\Gamma_K^\bullet E)_k)
\simeq
\holim_{H \le U \le_o G}   
F_A((E^{hU})_k, (\holim_\Delta (\Gamma_K^\bullet E))_k).
$$
Using Theorem~\ref{thm:cochaincomplex}, we get
\begin{align*}
\holim_{H \le U \le_o G}   
F_A((E^{hU})_k, (\holim_\Delta (\Gamma_K^\bullet E))_k)
& \simeq
\holim_{H \le U \le_o G}   
F_A((E^{hU})_k, (E^{hK})_k)
\\
& \simeq 
F_A(\colim_{H \le U \le_o G} (E^{hU})_k, (E^{hK})_k).
\end{align*}
We may now apply Theorem~\ref{thm:Galoiscolim} to deduce
$$
F_A(\colim_{H \le U \le_o G} (E^{hU})_k, (E^{hK})_k)
\simeq
F_A((E^{hH})_k, (E^{hK})_k).
$$
\end{proof}

\section{Applications to Morava $E$-theory}\label{sec:Etheory}

\subsection{Morava $E$-theory as a profinite Galois
extension}\label{sec:EGalois}

The general theory developed in this paper applies to Morava $E$-theory.
In this setting, we have
\begin{align*}
k & = K(n), \\
A & = S_{K(n)}, \\
G & = \GG_n, 
\end{align*}
where $K(n)$ is the $n$th Morava $K$-theory spectrum, 
$S_{K(n)}$ is the $K(n)$-local sphere spectrum, and 
$\GG_n$ is the $n$th extended Morava stabilizer group: 
$$ \GG_n = \MS_n \rtimes \mathrm{Gal}(\mathbb{F}_{p^n}/\mathbb{F}_p). $$  

Let 
$E_n$ be the $n$th Morava $E$-theory spectrum, where 
$$ (E_n)_\ast=W(\mathbb{F}_{p^n})[[u_1, ..., u_{n-1}]][u^{\pm 1}]. $$ 
Here, the degree of $u$ is $-2$ and the complete power series ring 
is in degree zero.  
Goerss and Hopkins \cite{GoerssHopkins}, 
building on work of Hopkins 
and Miller \cite{RezkHMT}, 
showed that $\GG_n$ acts on $E_n$ by maps of commutative $S$-algebras.  

Devinatz and
Hopkins \cite{DevinatzHopkins} constructed 
homotopy fixed point spectra 
$$ E_n^{dhH} $$ 
for closed subgroups $H$
of $\GG_n$.  (Here we use the notation $E_n^{dhH}$ to distinguish the
Devinatz-Hopkins homotopy fixed point spectra from the homotopy fixed point
spectra constructed in
this paper.) We give a brief overview 
of the way that Devinatz and Hopkins defined $E_n^{dhH}$.
\begin{enumerate}
\item For $U$ an open subgroup of $\GG_n$, Devinatz and Hopkins form a
cosimplicial object in the stable homotopy category
$$ \td{\mathbf{C}}^\bullet_U: \Delta \rightarrow \mathrm{Ho}(\Sigma\Sp) $$
having the property that
$$ \td{\mathbf{C}}^s_U = (E_n^{\wedge s} \wedge \Map(\GG_n/U, E_n))_{K(n)}. $$
This would yield a $K(n)$-local $E_n$-based Adams
resolution of $E_n^{dhU}$, if this spectrum were to exist, given that one
expects
$$ (E_n \wedge E_n^{dhU})_{K(n)} \simeq \Map(\GG_n/U, E_n). $$
The cosimplicial objects taken together give a functor
$$ \td{\mathbf{C}}^\bullet_{(-)}:
(\mathcal{O}^{\mathrm{fin}}_{\GG_n})^{\mathrm{op}} \rightarrow
c\mathrm{Ho}(\Sigma\Sp), $$
from the opposite category of finite $\GG_n$-orbits to the category of
cosimplicial objects in the stable homotopy category.

\item
Using the $E_\infty$-mapping space calculations of \cite{GoerssHopkins},
together with the rectification machinery of \cite{DKS}, Devinatz and
Hopkins show that the diagram $\td{\mathbf{C}}^\bullet_U$ admits a
rectification to give a point-set level commutative diagram
$$ \mathbf{C}^\bullet_U: \Delta \rightarrow \Alg_S. $$
Furthermore, they show that the rectification can be made functorial in
$U$.

\item The spectrum $E_n^{dhU}$ is defined by
$$ E_n^{dhU} := \holim_\Delta \mathbf{C}^\bullet_U. $$
When $U$ is normal in $\GG_n$, the functoriality of
$\mathbf{C}^\bullet_U$ in $U$ induces a $\GG_n/U$-action on $E_n^{dhU}$ by maps
of $S$-algebras. 

\item For $H$ a closed subgroup of $\GG_n$, the spectrum $E_n^{dhH}$ is
defined by
$$ E_n^{dhH} := (\colim_{H \le U \le_o \GG_n} E_n^{dhU})_{K(n)}. $$
\end{enumerate}

Rognes \cite[Thm.~5.4.4, Lem.~4.3.7]{Rognes} 
observed, for $U$ an open normal subgroup of $\GG_n$, that
the work of Devinatz and Hopkins \cite{LHS, DevinatzHopkins} 
proves that $E_n^{dhU}$ is a faithful 
$K(n)$-local
$\GG_n/U$-Galois extension of $S_{K(n)}$.
Therefore, the discrete commutative $\GG_n$-$S_{K(n)}$-algebra
$$ F_n = \colim_{U \trianglelefteq_o \GG_n} E_n^{dhU} $$
is a profaithful $K(n)$-local profinite $\GG_n$-Galois extension of 
$S_{K(n)}$. 
The spectrum $E_n$ is recovered by the equivalence 
(see \cite[Def.~1.5, Thm.~3(i)]{DevinatzHopkins})
$$ E_n \simeq (F_n)_{K(n)}. $$
With this in mind, we make the following definition.

\begin{defn}\label{defn:EnhG}
For $H$ a closed subgroup of $\GG_n$, we define
$$ E_n^{hH} := (F_n^{hH})_{K(n)}, $$
which is a commutative
$S_{K(n)}$-algebra by Lemma~\ref{lem:5.2.6}.
\end{defn}

We note that the use of pro-spectra  
gives an alternative, but equivalent
approach to defining $E_n^{hH}$.
For a sequence of integers $I = (i_0, \ldots, i_{n-1})$, we shall let $M_I$
denote the generalized Moore spectrum with
$$ BP_* M_I \cong BP_*/(p^{i_0}, v_1^{i_1}, \ldots, v_{n-1}^{i_{n-1}}). $$
These spectra are inductively defined by cofiber sequences
$$ \Sigma^{2i_j(p^j-1)}M_{(i_0, \ldots, i_{j-1})} \xrightarrow{v_j^{i_j}}
M_{(i_0, \ldots, i_{j-1})} \rightarrow M_{(i_0, \ldots, i_j)} $$
where $v_j^{i_j}$ is a $v_j$ self-map.  While $M_I$ does not exist for
every $I$, the Hopkins-Smith Periodicity Theorem implies that the maps
$v_j^{i_j}$ exist for $i_j \gg 1$.  

As explained in 
\cite[Prop.~7.10]{HoveyStrickland}, one can form a pro-spectrum
$\{ M_I \}_I$ for a cofinal system of multi-indices $I$.  
Then the pro-spectrum
$$ {\bf E}_n = \{
F_n \wedge M_I \}_I $$ 
is a continuous $H$-spectrum.  Since
$F_n$ is $E(n)$-local and each $M_I$ is a finite spectrum, the homotopy
fixed points are identified by
\begin{align*}
E_n^{hH} & = \holim_I (F_n \wedge M_I)^{hH} \\
& \simeq \holim_I F_n^{hH} \wedge M_I \\
& \simeq (F_n^{hH})_{K(n)}.
\end{align*}
Thus, the homotopy fixed points of the continuous $H$-spectrum ${\bf E}_n$ 
coincide with the
$K(n)$-localization of the homotopy fixed points of the discrete
$H$-spectrum $F_n$.  In particular, Definition~\ref{defn:EnhG} is
equivalent to the definition of $E_n^{hH}$ given in \cite{Davis}.

\begin{prop}\label{prop:EGalois}
The profaithful $K(n)$-local profinite 
$\GG_n$-Galois extension $F_n$ of $S_{K(n)}$ is 
consistent and has finite vcd.
\end{prop}

\begin{proof}
There is a zig-zag
$$ (F_n^{\wedge_{S_{K(n)}} \bullet+1})_{K(n)} \xleftarrow{\simeq} 
(F_n^{\wedge \bullet + 1})_{K(n)} \xrightarrow{\simeq} 
((F_n)_{K(n)}^{\wedge \bullet + 1})_{K(n)}
\simeq (E_n^{\wedge \bullet + 1})_{K(n)}
$$
of levelwise equivalences of cosimplicial objects.  We
therefore have equivalences
\begin{align*}
(S_{K(n)})^\wedge_{K(n), F_n} & =
\holim_\Delta (F_n^{\wedge_{S_{K(n)}} \bullet+1})_{K(n)} \\
& \simeq 
\holim_\Delta (E_n^{\wedge \bullet + 1})_{K(n)} \\
& \simeq S_{K(n)},
\end{align*}
where the last equivalence follows from the fact that the cosimplicial
object 
$$ (E_n^{\wedge \bullet + 1})_{K(n)} $$ 
is the $K(n)$-local $E_n$-Adams
resolution for $S_{K(n)}$.

Since $\GG_n$ is a compact $p$-adic analytic group, 
it has finite virtual cohomological dimension.  
Therefore, the extension $F_n$ of $S_{K(n)}$ has finite vcd.
\end{proof}

\begin{cor}\label{neato}
There is a weak equivalence $E_n^{h\GG_n} \simeq S_{K(n)}$.
\end{cor}

\begin{proof}
This follows immediately from Corollary~\ref{cor:consistent}.
\end{proof}

\begin{rmk}
Because of the result of Devinatz and Hopkins that there is an equivalence
$E_n^{dh\GG_n} \simeq 
S_{K(n)}$ \cite[Theorem 1(iii)]{DevinatzHopkins}, 
it has been known for some time 
that the $K(n)$-local sphere behaves like a homotopy fixed point spectrum; 
Corollary~\ref{neato} makes this idea precise.  
\end{rmk}

\subsection{Comparison with the Devinatz-Hopkins homotopy fixed
points}\label{sec:comparison}

Let $H$ be a closed subgroup of $\GG_n$.
The following theorem relates the 
homotopy fixed point construction $E_n^{hH}$ of 
Definition~\ref{defn:EnhG}
to
the Devinatz-Hopkins homotopy fixed point construction $E_n^{dhH}$ of 
\cite{DevinatzHopkins}. 

\begin{thm}\label{dhGhG}
If $H$ is a closed subgroup of $\GG_n$, there is an equivalence
$$ E_n^{dhH} \simeq E_n^{hH}. $$
\end{thm}

\begin{proof}
As explained in Subsection~\ref{sec:EGalois}, 
$F_n = \colim_{U \trianglelefteq_o \GG_n} E_n^{dhU}$ is a 
consistent profaithful $K(n)$-local profinite $\GG_n$-Galois 
extension of $S_{K(n)}$ of finite 
vcd.  
Thus, by Lemma~\ref{lem:consistency}, 
for each $U \trianglelefteq_o \GG_n$,
there is a $\GG_n/U$-equivariant equivalence
\[E_n^{dhU} \simeq (F_n^{hU})_{K(n)}.  \] 
Therefore, given 
a generalized Moore spectrum $M_I$, there is an equivalence 
\begin{equation}\label{U}
E_n^{dhU} \wedge M_I \simeq F_n^{hU} \wedge M_I.
\end{equation} 
By Theorem~\ref{thm:Galoiscolim}, 
the natural map 
\[(\colim_{H \leq V \leq_o \GG_n} 
F_n^{hV})_{K(n)} \rightarrow (F_n^{hH})_{K(n)}\] 
is an equivalence. 

Let $V$ be any open subgroup of $\GG_n$. Then $V$ contains a subgroup $W$ such 
that $W$ is an open normal subgroup of $\GG_n$. 
We have the following chain of 
equivalences:
\begin{align*} 
E_n^{dhV} \wedge M_I & \simeq (E_n^{dhW})^{hV/W} \wedge M_I 
\simeq (E_n^{dhW} \wedge M_I)^{hV/W} \\
& \simeq (F_n^{hW} \wedge M_I)^{hV/W} \simeq 
(F_n^{hW})^{hV/W} \wedge M_I \\
& \simeq F_n^{hV} \wedge M_I,
\end{align*} 
where the first equivalence is \cite[Thm.~4]{DevinatzHopkins}, 
the second and fourth 
equivalences follow 
from the fact that $M_I$ is a finite spectrum, 
the third equivalence is because 
of (\ref{U}), and the last equivalence is due 
to Proposition~\ref{prop:iterate}.

Recall from Definition~\ref{defn:EnhG} that 
$E_n^{hH} = (F_n^{hH})_{K(n)}$. Then 
the above observations imply that
\begin{align*}
E_n^{hH} & \simeq 
(\colim_{H \leq V \leq_o \GG_n} F_n^{hV})_{K(n)} \\
& \simeq \holim_I \colim_{H \leq V \leq_o \GG_n} (F_n^{hV} \wedge M_I) \\
& \simeq \holim_I \colim_{H \leq V \leq_o \GG_n} (E_n^{dhV} \wedge M_I) \\
& \simeq (\colim_{H \leq V \leq_o \GG_n} E_n^{dhV})_{K(n)} \\
& \simeq E_n^{dhH},
\end{align*} 
where the last equivalence follows from
\cite[Def.~1.5]{DevinatzHopkins}.
\end{proof}

\begin{rmk}
As mentioned in the Introduction, 
Theorem~\ref{dhGhG} first appeared in the second author's thesis 
\cite{Davisthesis}. 
The 
arguments in \cite{Davisthesis} relied on a somewhat complicated analysis of a 
$K(n)$-local $E_n$-Adams resolution of $E_n^{dhH}$, 
whereas our proof makes use of Rognes's Galois theory and consequently, 
is more efficient.
\end{rmk}

\begin{cor}\label{finite}
If $H$ is a closed subgroup of $\GG_n$ and $X$ is a finite spectrum, then
there is an equivalence
\[E_n^{dhH} \wedge X \simeq (E_n \wedge X)^{hH}.\]
\end{cor}

\begin{proof}
By \cite[Thm.~1.3, Rmk.~9.3]{Davis}, 
$E_n \wedge X$ is a continuous $H$-spectrum (in the sense of
\cite{Davis}). 
Then, by Theorem~\ref{dhGhG} and \cite[Thm.~9.9]{Davis}, 
\[E_n^{dhH} \wedge X \simeq E_n^{hH} \wedge X \simeq (E_n \wedge X)^{hH}.\]
\end{proof}

\begin{cor}\label{finite2}
If $X$ is a finite spectrum, there is an equivalence
$$ X_{K(n)} \simeq (E_n \wedge X)^{h\GG_n}. $$
\end{cor}

\begin{proof}
By \cite[Thm.~9.9]{Davis} and Corollary~\ref{neato}, 
\[(E_n \wedge X)^{h\GG_n} \simeq E_n^{h\GG_n} \wedge X \simeq S_{K(n)} \wedge X 
\simeq X_{K(n)},\] 
where the last equivalence follows from the fact that 
$X$ is finite.
\end{proof}

Let $X$ be a finite spectrum. By \cite[Thm.~2(ii)]{DevinatzHopkins}, 
there is a strongly convergent $K(n)$-local $E_n$-Adams 
spectral sequence that has the form 
\begin{equation}\label{ssone}
H^s_c(H;\pi_t(E_n \wedge X)) \Rightarrow \pi_{t-s}(E_n^{dhH} \wedge X).
\end{equation} 
Also, by 
\cite[Thm.~1.7]{Davis}, there is a descent spectral sequence 
\begin{equation}\label{sstwo}
H^s_c(H;\pi_t(E_n \wedge X)) \Rightarrow \pi_{t-s}((E_n \wedge X)^{hH}).
\end{equation}

\begin{thm}\label{ss}
If $H$ is a closed subgroup of $\GG_n$ and 
$X$ is a finite spectrum, then spectral 
sequence $\mathrm{(}$\ref{ssone}$\mathrm{)}$ is isomorphic to spectral sequence 
$\mathrm{(}$\ref{sstwo}$\mathrm{)}$, from the $E_2$-terms onward.
\end{thm}

\begin{proof}
By
\cite[proof of Prop. 7.4]{HMS}, 
spectral sequence (\ref{ssone}) is 
the inverse limit over $\{I\}$ of 
$K(n)$-local $E_n$-Adams 
spectral sequences 
that have the form
\begin{equation}\label{ssthree} 
^{I}E_2^{s,t}(I) = H^s_c(H;\pi_t(E_n \wedge M_I \wedge X)) \Rightarrow \pi_{t-s}(E_n^{dhH} \wedge 
M_I \wedge X).
\end{equation}
Similarly, spectral sequence (\ref{sstwo}) is the inverse limit 
over $\{ I \}$
of conditionally convergent descent 
spectral sequences 
that have the form
\begin{equation}\label{ssfour} 
^{II}{E_2^{s,t}}(I) = 
H^s_c(H;\pi_t(E_n \wedge M_I \wedge X)) \Rightarrow \pi_{t-s}(E_n^{hH} \wedge 
M_I \wedge X).
\end{equation}
Henceforth, we write the $E_2$-terms
$^{I}E_2^{s,t}(I)$ and $^{II}E_2^{s,t}(I)$
as $^{I}E_2^{s,t}$ and $^{II}E_2^{s,t}$, respectively.
\par
Note that spectral sequence (\ref{ssthree}) is isomorphic to the 
strongly convergent $K(n)$-local $E_n$-Adams spectral sequence 
\begin{equation}\label{ssfive}
^{I}{E_2^{s,t}} \cong H^s_c(H;(E_n)^{-t}(DX \wedge DM_I)) 
\Rightarrow (E_n^{dhH})^{-t+s}(DX \wedge 
DM_I).
\end{equation}
Thus, to prove the theorem, it suffices to show that the spectral sequences in (\ref{ssfour}) and 
(\ref{ssfive}) are isomorphic to each other.
\par
Notice that 
\begin{align*}
^{I}E_2^{s,t} & \cong {^{II}E_2^{s,t}} \\ &  
\cong \colim_{N \trianglelefteq_o \GG_n} H^s(H/{(H \cap N)}; 
\pi_t(E_n^{dhN} \wedge M_I \wedge X)) 
\\  
& \cong \colim_{N \trianglelefteq_o \GG_n} H^s(NH/N; 
(E_n^{dhN})^{-t}(DX \wedge DM_I)) \\ & 
= \colim_{N \trianglelefteq_o \GG_n} {^{III}E_2^{s,t}}(N),
\end{align*} 
where 
${^{III}E_2^{s,t}}(N)$ is the $E_2$-term of the strongly convergent
spectral sequence 
$^{III}E_r^{\ast,\ast}(N)$, which has the form 
\begin{equation}\label{sssix}
H^s(NH/N; (E_n^{dhN})^{-t}(DX \wedge DM_I)) 
\Rightarrow (E_n^{dhNH})^{-t+s}(DX \wedge DM_I)
\end{equation}
and is the Adams spectral sequence constructed by Devinatz in
\cite[(0.1)]{LHS}.
\par
By Lemma \ref{lem:mapofSS} below, there is a map from spectral 
sequence (\ref{sssix}) to spectral sequence (\ref{ssfive}), such that the 
isomorphism 
\[{^{I}E_2^{s,t}} \cong \colim_{N \trianglelefteq_o \GG_n}
{^{III}E_2^{s,t}}(N)\] 
implies that the spectral sequence of (\ref{ssfive}) is isomorphic to the 
spectral sequence $\colim_{N \trianglelefteq_o \GG_n}
{^{III}E_r^{\ast,\ast}}(N)$. Thus, 
we only have to show that spectral sequences (\ref{ssfour}) and 
$\colim_{N \trianglelefteq_o \GG_n} {^{III}E_r^{\ast,\ast}}(N)$ are isomorphic to each other.
\par
By \cite[Thm.~A.1]{LHS}, $^{III}E_r^{\ast,\ast}(N)$ is isomorphic to the 
usual descent spectral sequence $^{IV}E_r^{\ast,\ast}(N)$
that has the form
\[H^s(NH/N; (E_n^{dhN})^{-t}(DX \wedge DM_I)) 
\Rightarrow ((E_n^{dhN})^{hNH/N})^{-t+s}(DX \wedge DM_I),\] 
since, in
the notation of \cite[App.~A]{LHS}, the ``homotopy
fixed point spectral sequence" with abutment
$$ [E_n^{dhNH} \wedge DX \wedge DM_I,
(E_n^{dhN})^{hNH/N}]^\ast_{E_{n}^{dhNH}}, $$
which is isomorphic to
$[DX \wedge DM_I, (E_n^{dhN})^{hNH/N}]^\ast$, is
equivalent to $^{IV}E_r^{\ast,\ast}(N)$.
Because of the isomorphism
\[\colim_{N \trianglelefteq_o \GG_n} {^{III}E_r^{\ast,\ast}}(N) 
\cong \colim_{N \trianglelefteq_o \GG_n} {^{IV}E_r^{\ast,\ast}}(N),\] 
our proof reduces to showing
that (\ref{ssfour}) and 
$\colim_{N \trianglelefteq_o \GG_n} {^{IV}E_r^{\ast,\ast}}(N)$ 
are isomorphic 
spectral sequences.
\par
The abutment of spectral sequence (\ref{ssfour}) is the homotopy of 
\begin{align*}
E_n^{hH} \wedge M_I \wedge X & \simeq (F_n \wedge M_I \wedge X)^{hH} 
\\ & \cong \holim_\Delta \mathrm{Map}^c(H^\bullet, 
\colim_{N \trianglelefteq_o \GG_n} (E_n^{dhN} \wedge M_I \wedge X))
\\ & \cong \holim_\Delta \colim_{N \trianglelefteq_o \GG_n} 
\mathrm{Map}^c((NH/N)^\bullet, E_n^{dhN} \wedge M_I \wedge X),
\end{align*} 
where the first equivalence is by
\cite[Cor. 9.8]{Davis}
and the second equivalence is
an identification, given by Theorem
\ref{thm:cochaincomplex}. 
Since $NH/N$ is a finite group, there is an identification
$$ (E_n^{dhN} \wedge M_I \wedge X)^{hNH/N} 
\cong \holim_\Delta \mathrm{Map}^c((NH/N)^\bullet, E_n^{dhN} \wedge M_I \wedge X). $$
Thus, for each $U \trianglelefteq_o \GG_n$, the canonical map 
$$ \mathrm{Map}^c((UH/U)^\bullet, E_n^{dhU} \wedge M_I \wedge X) 
\rightarrow \colim_{N \trianglelefteq_o \GG_n} 
\mathrm{Map}^c((NH/N)^\bullet, E_n^{dhN} \wedge M_I \wedge X) $$ of 
cosimplicial spectra induces a map 
$$ (E_n^{dhU} \wedge M_I \wedge X)^{hUH/U} 
\rightarrow (F_n \wedge M_I \wedge X)^{hH} $$ 
and a map 
$$ \psi_U \colon ^VE_r^{\ast,\ast}(U) \rightarrow {} ^{II}E_r^{\ast,\ast} $$  
of conditionally convergent spectral sequences, where $^{V}E_r^{\ast,\ast}(U)$ is the 
descent spectral sequence that has the form 
$$ H^s(UH/U;\pi_t(E_n^{dhU} \wedge M_I \wedge X)) \Rightarrow 
\pi_{t-s}((E_n^{dhU} \wedge M_I \wedge X)^{hUH/U}). $$ 
It will be helpful to note that the abutment of $^{V}E_r^{\ast,\ast}(U)$ can 
also be written as $\pi_{t-s}((E_n^{dhU})^{hUH/U} \wedge M_I \wedge X)$.
\par
Since the map $\colim_{N \trianglelefteq_o \GG_n} \psi_N$ 
induces an isomorphism 
$$ \colim_{N \trianglelefteq_o \GG_n} {^{V}E_2^{s,t}}(N) 
\cong {^{II}E_2^{s,t}}, $$ there is an 
isomorphism between spectral sequences (\ref{ssfour}) and 
$\colim_{N \trianglelefteq_o \GG_n} {^{V}E_r^{\ast,\ast}}(N) .$ Therefore, the proof 
is completed by showing that there is an isomorphism 
$$ \colim_{N \trianglelefteq_o \GG_n} {^{IV}E_r^{\ast,\ast}}(N) \cong 
\colim_{N \trianglelefteq_o \GG_n} {^{V}E_r^{\ast,\ast}}(N) $$ 
of spectral sequences; 
this follows from the fact that spectral sequences 
${^{IV}E_r^{\ast,\ast}}(N)$ and 
${^{V}E_r^{\ast,\ast}}(N)$ are equivalent to each other.
\end{proof}

The following results are needed for the above proof.

\begin{lem}\label{lem:adams}
Suppose that $A$ is a $k$-local commutative symmetric ring
spectrum and that $E$ is a $k$-local commutative $A$-algebra.
Then the canonical $k$-local $E$-resolution of $A$ in the
category of $A$-modules
\begin{equation*}
\ast \rightarrow A \rightarrow E \rightarrow (E \wedge_A E)_k 
\rightarrow (E \wedge_A E \wedge_A E)_k \rightarrow \cdots 
\end{equation*}
is a $k$-local $E$-resolution of $A$ in the category of
$S$-modules. 
\end{lem}

\begin{proof}
We use the terminology of \cite{Miller}, adapted to the category of
$k$-local $A$-modules as in \cite[Sec.~2]{LHS}.  To prove the
lemma, we will show that
the associated $k$-local $E$-Adams resolution of $A$-modules
\begin{equation}\label{eq:adamsres}
\xymatrix@C-1em{
A \ar@{=}[r] & A^0 \ar[dr]_{j_0}  \
&& A^1 \ar[dr]_{j_1} \ar[ll]
&& A^2 \ar[ll] && \cdots 
\\
&& E \ar[ur] && (E \wedge_A E)_k \ar[ur]
}
\end{equation}
is a $k$-local $E$-Adams resolution of $S$-modules.  It suffices to
verify that (1) for all $s > 0$, the spectra $(E^{\wedge_A s})_k$ are $k$-local
$E$-injective, and (2) that each of the maps $j_i$ in
(\ref{eq:adamsres})
is $k$-local $E$-monic.

Claim (1) follows from the fact that the map
$$ (E^{\wedge_A s})_k = (S \wedge E^{\wedge_A s})_k \rightarrow 
(E \wedge E^{\wedge_A s})_k $$
is split-monic.
Claim (2) follows from the fact that every
$k$-local $E$-monic map of $A$-modules is $k$-local $E$-monic as a map 
of $S$-modules.  Indeed, if $f: X \rightarrow Y$ is a $k$-local
$E$-monic map of $A$-modules, then consider the following diagram.
$$
\xymatrix{
X \ar[r]^f \ar[d]_{u} 
& Y \ar[d]
\\
(E \wedge_A X)_k \ar[r]_{1 \wedge f} & 
(E \wedge_A Y)_k
}
$$
The map $u$ is seen to be a $k$-local $E$-monic map of $S$-modules
because the map
$$ (E \wedge X)_k \xrightarrow{1 \wedge u} (E \wedge E \wedge_A
X)_k $$
is split-monic.  Because $f$ is a $k$-local $E$-monic map of 
$A$-modules, the map $1 \wedge f$ is split-monic, and therefore $1
\wedge f$ is a $k$-local $E$-monic map of $S$-modules.  
We deduce that $f$ is a $k$-local $E$-monic map of $S$-modules.
\end{proof}

\begin{lem}\label{lem:mapofSS}
Let $H$ be a closed subgroup of $\GG_n$ and let $N$ be an open normal subgroup. 
Then, for any spectrum $Z$, there is a natural map from the Adams 
spectral sequence 
\[H^s(NH/N;(E_n^{dhN})^{-t}(Z)) \Rightarrow (E_n^{dhNH})^{-t+s}(Z)\] 
of \cite[(0.1)]{LHS} 
to the Adams spectral sequence 
\[H^s_c(H;(E_n)^{-t}(Z)) \Rightarrow (E_n^{dhH})^{-t+s}(Z) \]
of \cite[Thm.~2(ii)]{DevinatzHopkins}.
When
$Z = DM_I \wedge Z',$ where $Z'$ is any finite
spectrum, the induced map on $E_2$-terms is
the usual map in continuous group cohomology
that is induced by the canonical maps
$H \rightarrow NH/N$ and
$E_n^{dhN} \rightarrow
\colim_{U \trianglelefteq_o \mathbb{G}_n}
E_n^{dhU}$.
\end{lem}

\begin{proof}
To ease our notation, we write $E_n^{hK}$ in place of $E_n^{dhK}$, whenever 
$K$ is closed in $\GG_n$. 
Also, we take implicit cofibrant replacements
as needed.
The first spectral sequence is formed from the resolution 
\[\ast \rightarrow E_n^{hNH} \rightarrow E_n^{hN} \rightarrow 
(E_n^{hN} \wedge_{E_n^{hNH}} E_n^{hN})_{K(n)} 
\rightarrow \cdots\] (see the 
discussions after \cite[(0.1) and Prop.~3.6]{LHS}). 
By Lemma~\ref{lem:adams}, the canonical
$K(n)$-local $E_n$-resolution
\[\ast \rightarrow E_n^{hH} \rightarrow E_n \rightarrow 
(E_n \wedge_{E_n^{hH}} E_n)_{K(n)} \rightarrow 
(E_n \wedge_{E_n^{hH}} E_n \wedge_{E_n^{hH}} E_n)_{K(n)} 
\rightarrow \cdots \] 
of $E_n^{hH}$
in the category of $E_n^{hH}$-modules is also a $K(n)$-local
$E_n$-resolution of $E_n^{hH}$ in the category of $S$-modules.
Thus, the second spectral sequence, which was originally
constructed by using such a resolution in the category of $S$-modules
(see \cite[p.~32, App.~A]{DevinatzHopkins}), can be regarded
as a $K(n)$-local $E_n$-Adams spectral sequence in the category of
$E_n^{hH}$-modules, so that we can also regard the second
spectral sequence as being given by \cite[(0.1)]{LHS} through the
resolution 
\[\ast \rightarrow E_n^{hH} \rightarrow E_n \rightarrow 
(E_n \wedge_{E_n^{hH}} E_n)_{K(n)} \rightarrow 
(E_n \wedge_{E_n^{hH}} E_n \wedge_{E_n^{hH}} E_n)_{K(n)} 
\rightarrow \cdots.\] 
\par
As in \cite[(3.7)]{LHS}, there is a canonical map to 
the preceding resolution, 
from the resolution 
\[(\colim_{U \trianglelefteq_o \GG_n} E_n^{hUH})_{K(n)} 
\rightarrow 
(\colim_{U \trianglelefteq_o \GG_n} E_n^{hU})_{K(n)} \rightarrow
(\colim_{U \trianglelefteq_o \GG_n} (E_n^{hU} \wedge_{E_n^{hUH}} 
E_n^{hU}))_{K(n)} \rightarrow \cdots\] 
and this map is a levelwise weak equivalence (at the beginning of
the last resolution, the usual 
``\ $\ast \rightarrow \ $" was omitted for the sake of space). 
This last 
resolution receives the obvious map 
from the first resolution
\[\ast \rightarrow E_n^{hNH} \rightarrow E_n^{hN} \rightarrow 
(E_n^{hN} \wedge_{E_n^{hNH}} E_n^{hN})_{K(n)} 
\rightarrow \cdots.\] 
Thus, composition gives a map $\lambda$ 
from the resolution for the first spectral sequence 
to the resolution for the second spectral sequence; $\lambda$ 
induces the desired 
map of spectral sequences.   
\par
By \cite[Cor.~3.9]{LHS},
$$ \pi_\ast((E_n^{hN} \wedge_{{E_n^{hNH}}} E_n^{hN})_{K(n)})
\cong \Map^c(NH/N,\pi_\ast(E_n^{hN})) $$ and
$$ \pi_\ast((E_n \wedge_{{E_n^{hH}}} E_n)_{K(n)})
\cong \Map^c(H,\pi_\ast(E_n)), $$
and, hence, the last statement
of the lemma follows easily from the definition of
$\lambda$ and \cite[proof of Thm.~3.1]{LHS}.
\end{proof}

\end{document}